\documentclass[11pt,reqno]{amsart}
\usepackage{amssymb, latexsym}
\usepackage[curve,matrix,arrow]{xy}
\usepackage{hyperref}
\usepackage{pifont}
\DeclareSymbolFont{matha}{OML}{txmi}{m}{it}
\DeclareMathSymbol{\varv}{\mathord}{matha}{118}

\usepackage[switch]{lineno}

\newcommand{\diag}{{\rm diag}}
\def\tr{{\rm tr}}
\def\q{\quad}
\newcommand{\SO}{{\mathrm{SO}}}
\newcommand{\Ot}{{\mathrm{O}}}
\def\PSp{{\rm PSp}}

\def\PSU{{\rm PSU}}
\newcommand{\GL}{{\mathrm{GL}}}
\newcommand{\SL}{{\mathrm{SL}}}
\newcommand{\Sp}{{\mathrm{Sp}}}
\newcommand{\SU}{{\mathrm{SU}}}
\newcommand{\GF}{{\mathrm{GF}}}

\newcommand{\PSL}{{\mathrm{PSL}}}

\newcommand{\myupsilon}{{\nu}}
\def\UU{{V}} 

\newtheorem{definition}{Definition}[section]
\newtheorem{lemma}[definition]{Lemma}

\newtheorem{theorem}[definition]{Theorem}
\newtheorem{corollary}[definition]{Corollary}

\newcommand{\CH}{{C}} 

\begin{document}


\title{Presentations on standard generators for classical groups}

\author{C.R.\ Leedham-Green}
\address{
School of Mathematical Sciences, Queen Mary,
University of London, London E1 4NS, United Kingdom
}
\email{c.r.leedham-green@qmul.ac.uk}

\author{E.A.\ O'Brien}
\address{Department of Mathematics, University of Auckland,
Auckland, New Zealand}
\email{e.obrien@auckland.ac.nz}

\thanks{
Both authors were supported by the Marsden Fund of New Zealand 
via grant UOA 1323. 
O'Brien thanks the Hausdorff Institute, Bonn,
for its hospitality while this work was completed.
We thank Peter Brooksbank, Marston Conder,
George Havas, 
Derek Holt, 
Alexander Hulpke, 
Bill Kantor, 
Martin Kassabov, 
and Martin Liebeck for helpful comment and discussion.}


\date{\today}
\maketitle

\begin{abstract}
For each family of finite classical groups, 
and their associated simple quotients, we provide an explicit 
presentation on a specific generating set of size at most 8. 
Since there exist efficient algorithms to construct
this generating set in any copy of the group, 
our presentations can be used to verify claimed isomorphisms
between representations of the classical group.
The presentations are available in {\sc Magma}. 
\end{abstract}

\begin{center}
\begin{small} 
 {\it In memory of Charles Sims who taught us so much} 
\end{small}
\end{center}

\tableofcontents
\setcounter{tocdepth}{1} 

\section{Introduction}
The finite simple groups of Lie type
play a fundamental role in many areas of mathematics. 
Various presentations of these groups satisfying specific requirements 
are known.
Steinberg \cite{Steinberg68,Steinberg81} gave 
the first presentations. An important
variation is that of Curtis-Steinberg-Tits (see \cite[Theorem 2.9.3]{GLS3}).
Motivated in part by a conjecture of \cite{BabaiSzemeredi84},
Babai {\it et al.\ }\cite{Babaietal97} provided presentations 
on a reduced set of Curtis-Steinberg-Tits generators.
Guralnick {\it et al.\ }\cite{GKKL,JEMS}
provide {\it short} presentations (defined below)
for every finite simple group of Lie type except the Ree groups $^2G_2(q)$.
In this paper, we focus on the classical groups.
For each family of finite classical groups, 
and their associated simple quotients, we provide an explicit 
presentation on a specific generating set of size at most 8. 
The most important feature of our work is that 
there exist efficient algorithms to construct
this generating set in any copy of the group, 
so our presentations can be used in practice to verify claimed isomorphisms
between representations of the classical group.
Encodings of the resulting presentations are available 
in {\sc Magma} \cite{Magma}.

The principal motivation for our work is 
the \emph{matrix group recognition project} \cite{OBrien09}; it aims 
to develop high-quality algorithms to compute efficiently 
with an arbitrary subgroup $G$ of $\GL(m,K)$, where $K$ is a finite field
and $G$ is specified by a generating set $Y$.
The {\sc CompositionTree} algorithm \cite{BHLO} 
constructs a composition 
series for $G$, and facilitates membership testing in $G$ and writing 
elements of $G$ as words in $Y$.  Additional functionality building on 
the data structure provided by {\sc CompositionTree} includes, 
for example, the computation of chief series, the 
soluble radical, and Sylow subgroups.  A highly optimised implementation
of this machinery is available in {\sc Magma}. 

Central to that work is a requirement to produce a concrete 
isomorphism between an arbitrary representation of a quasisimple 
group and (a central quotient of) its {\it standard copy}, 
a specific representation of the group.  
The standard copy of $A_n$ is the group acting on $n$ symbols.
Our principal focus here is on the classical groups: 
namely, $\SL(d, q)$, $\Sp(2n, q)$, $\SU(d, q)$, $\Omega^+(2n, q)$, 
$\Omega(2n + 1, q)$, and $\Omega^-(2n, q)$, where $d \geq 2$,
$n \geq 1$ and $q$ is a prime power; 
the standard copy of each is 
a natural representation of the group preserving a designated classical form. 
The standard copy of a quasisimple exceptional group of Lie type over 
$\GF(q)$ is defined in \cite{HRT}: a specific faithful irreducible 
representation of 
minimal dimension of the simply connected group. 

Let $\CH$ be the standard copy of a quasisimple group 
and let $H = \langle X \rangle \leq \GL(m, K)$ be isomorphic 
to $\CH / Z$ where $Z$ is central in $\CH$. 
Informally, a {\it constructive recognition} algorithm for $H$
constructs an isomorphism between $H$ and $\CH / Z $ and 
exploits this isomorphism
to write an arbitrary element of $H$ as a word in its generators $X$.
For a more formal definition, see \cite[p.~192]{Seress03}.

We can realise such an algorithm as follows. 
For each quasisimple group $\CH$, we define a specific 
sequence of \emph{standard generators} $\mathcal{S}$. The first task is to 
construct, as words in $X$, a sequence 
$\mathcal{S}'$ in $H$ 
such that $\mathcal{S} \to \mathcal{S}'$ induces an 
isomorphism from $\CH/Z$ to $H$.
The second task is to solve the
{\it constructive membership problem} for $H$ with
respect to $\mathcal{S}'$: namely, express any given $h\in H$ as a
word in $\mathcal{S}'$, and so as a word in $X$.
Now the isomorphism $\varphi\colon H\to \CH/Z$ that maps $\mathcal{S}'$ to
the images of $\mathcal{S}$ in $\CH/Z$ is \emph{constructive}: 
any given $h \in H$ 
may be written as a word $w(\mathcal{S}')$ in
$\mathcal{S}'$, and the image $\varphi(h)$
is immediately determined by the corresponding word in
the images of 
${\mathcal{S}}$ in $\CH/Z$. Similarly, the inverse of $\varphi$ is constructive.

Standard generators have been defined 
and associated constructive recognition
algorithms developed for all of the quasisimple groups.
The standard generators of $A_n$ are a 3-cycle 
and an $(n-1)$- or $n$-cycle depending on the parity of $n$; 
for a recognition algorithm, see \cite{MR3085037}.
For exceptional groups of rank at least 2, see \cite{CMT,LOB}. For the 
sporadic simple groups, see \cite{Wilson96, atlas_www}.
For the classical groups, see \cite{even,black,odd}; 
we use algorithms of 
Costi \cite{costi-phd}
and Schneider to solve the 
constructive membership problem. 
In all cases, the algorithms demonstrate both 
excellent theoretical and practical performance. 
For example, those in \cite{even,odd} have complexity $O(r^4 \log q)$ where 
the classical group has rank $r$.
Implementations are available in {\sc Magma}. 

All existing constructive recognition algorithms employ random selection 
in $H$ and are Monte Carlo; consequently, there is a small probability of error in the 
claimed isomorphism between $H = \langle X \rangle \leq \GL(m, K)$ and $\CH/Z$.
Presentations on standard generators for (central quotients of) 
quasisimple groups play a fundamental role:
namely, they permit us to verify the correctness of this isomorphism.
Assume we know a presentation $\{ \mathcal{S} \; |\;  \mathcal{R} \}$ 
for $\CH/Z$. 
The standard generators $\mathcal{S'}$ are constructed as words in $X$,
so demonstrating that $\langle X \rangle \geq 
\langle \mathcal{S'} \rangle$;
we use the constructive membership algorithm 
to write the input generators $X$ as words in $\mathcal{S'}$, so demonstrating
that $\langle X \rangle \leq \langle \mathcal{S'} \rangle$;
finally, if the relators in $\mathcal{R}$ evaluate to the identity
on $\mathcal{S'}$, then we have verified the claimed isomorphism.

As part of {\sc CompositionTree}, 
we must often construct both a homomorphism $\theta$
from $G \leq \GL(m, K)$ 
onto the simple quotient of a specified classical group $\CH$, 
and its associated kernel. 
We apply (a natural generalisation of) a constructive recognition algorithm to $G$,
so constructing $\mathcal{S'}$; 
now evaluating $\mathcal{R}$ on $\mathcal{S'}$ provides
normal generators for $\ker \theta$.
More generally, if we construct a composition tree 
for an arbitrary group $G = \langle Y \rangle \leq \GL(m, K)$, 
have a constructive membership test for each of its composition factors, 
and know a presentation on the standard generators of each factor, 
then, as we outline in \cite[Section 4]{BHLO}, 
we can construct a presentation for $G$ on $Y$; this allows
us to verify the correctness of the composition tree constructed for $G$.  

Presentations on standard generators for $A_n$ are well known.
For exceptional groups of rank at least 2, the 
requisite presentations are listed in \cite{LOB}; 
for (most of) the sporadic
simple groups, they are listed in the Atlas \cite{atlas_www}.
In this paper, we address the missing case by listing 
explicit presentations on standard generators for the 
classical groups and their central quotients. 

A more detailed account of the background and context of this work
appears in the survey \cite{OBrien09}.

\subsection{The main result}

Our motivations dictate the following
desiderata for presentations of classical groups.
\begin{enumerate}
\item The presentation is correct.
\item The presentation can be used to verify the
claimed isomorphism.
\item The number of relators is small.
\item The presentation length is ``short".
\end{enumerate}
The requirement that the presentation is correct may seem
obvious, but it is not easy to achieve. 
We have concentrated on producing presentations 
whose correctness is easy to verify. 

The second is realised by providing a presentation
on the standard generators of the group.

The desire for a small number of relators arises because 
we may have a group $G$ that maps onto a classical group $\CH$,  
and a normal subgroup $K$ of $G$ contained in the kernel of the
homomorphism of $G$ onto $\CH$. To verify that
$K$ is the full kernel, we must test
that a preimage of each relator belongs to $K$.

We seek a ``short" presentation
to minimise the cost of isomorphism verification. 
Babai and Szemer{\'e}di \cite{BabaiSzemeredi84}
define the {\it length} of a presentation $\{ X \,|\, R\}$
to be the number of symbols required to write down the presentation.
Typically, two metrics are applied to measure the length. 
One is {\it word length}: $|X|$ plus the sum of the lengths 
of the words in $R$ 
as strings in the alphabet $X \cup X^{-1}$.
A second is {\it bit-length}: $|X|$ plus the 
total number of bits required to encode the 
words in $R$ where all exponents are encoded as binary strings. 
As discussed in \cite{GKKL},
we can readily convert a presentation having bit-length $N$ into 
a presentation having word length $4N$. As we show in 
Section \ref{length}, for our applications the critical metric is bit-length.

We now state the main result, observing that it satisfies these desiderata. 
Each classical group has at most eight 
standard generators, as defined in \cite{even, odd}.
\begin{theorem}\label{main-result}
Every classical group of rank $r$ defined over $\GF(q)$ 
has a presentation on its standard generators
with $O(r)$ relations and total bit-length $O(r + \log q)$. 
\end{theorem}
For each family of classical groups, and their central quotients,
we exhibit such a presentation. 
Sometimes, for the convenience of theoretical arguments, we list a 
presentation on a (mildly) different generating set.
If so, we record explicitly the relationships between those
generators and the standard generators needed to rewrite the presentation  
on the standard generators; see, for example, \cite[\S 4.4, Remark 4]{Joh90}. 


Our presentation for a classical group $\CH$ of rank at least 2 
on its standard generators usually includes one for the 
normaliser of a maximal torus in a split $BN$ pair, and so
for $S_n$, the symmetric group of relevant degree $n$, or 
the corresponding group of signed permutation matrices of determinant 1. 
Our generators for $S_n$ in its action on $n$ symbols 
are a 2-cycle and $n$-cycle, and 
every presentation for $S_n$ on this generating set has word 
length $O(n)$ (see \cite[A2]{JEMS}). 
At most 56 additional relations are used to define $\CH$.

The structure of the paper is the following. In Section \ref{symmetric}
we provide presentations for $S_n$ and for the group
of signed permutation matrices of degree $n$ and determinant 1. 
In Sections $\ref{CRW}$  and $\ref{SU3q}$ we discuss 
important base cases: $\SL(2, q)$ and $\SU(3, q)$.
In subsequent sections, we provide presentations for 
each of $\SL(d, q)$, $\Sp(2n, q)$, $\SU(d, q)$, 
$\Omega^+(2n, q)$, $\Omega(2n + 1, q)$, and $\Omega^-(2n, q)$. 

The presentations and corresponding matrix generators are available 
\cite{code} in {\sc Magma}.  As we report in Section \ref{verify}, we 
used these extensively to verify the correctness of this work. 

\subsection{Related work}
Babai and Szemer{\'e}di \cite{BabaiSzemeredi84}
formulated the {\it Short Presentation Conjecture}:
there exists a constant $c$ such that
every finite simple group $G$ has a presentation of
bit-length $O(\log^c |G|)$.
The results of \cite{Babaietal97,HulpkeSeress01, Suzuki62} establish
this conjecture with $c=2$ for all finite simple groups, 
with the possible exception of the Ree groups ${}^2G_2(q)$.

The conjecture was motivated by potential 
complexity applications to questions about
matrix groups defined over finite fields 
(see \cite{BabaiSzemeredi84} for details); 
its proof also provided verification for the first constructive
recognition algorithms for classical groups, 
developed by Kantor and Seress \cite{KantorSeress01}. 


In two remarkable papers, 
Guralnick {\it et al.\ }\cite{GKKL, JEMS}
establish much more. We summarise just two of their results.
\begin{theorem}
Every non-abelian finite simple group of rank $r$ defined over ${\GF}(q)$,
with the possible exception of the Ree groups $^2G_2(q)$,
has a presentation with a bounded number of generators
and relations and word length $O(\log r + \log q)$.
Every finite quasisimple group of rank $r$ defined over ${\GF}(q)$,
with the possible exception of the Ree groups $^2G_2(q)$,
has a presentation with at most $9$ generators, $49$ relations,
and bit-length $O(\log r + \log q)$.
\end{theorem}
\noindent 
Word and bit-length $O(\log r + \log q)$ are both best possible, 
and such presentations are {\it short}.

Rarely do we know explicit words (short or otherwise) to express
standard generators in terms of the generators used in \cite{JEMS},
or vice versa, so it is not feasible to convert their presentations 
to ones on standard generators. 
One significant obstruction is that 
sometimes generators used there 
are identified only by specifying properties
they must satisfy.  Nor do constructive recognition algorithms employing 
their generators exist 
for classical groups, so these presentations cannot be
used directly to verify the necessary isomorphisms. 
Theorem 10.1 of \cite{JEMS} illustrates an additional concern: 
it employs a generator $\sigma$ specified only 
up to certain properties.  This element does not exist in $\Omega^{+}(2n, q)$; 
in private communication in 2012, the authors of \cite{JEMS} fixed the error.  
Huxford \cite{Huxford19} also corrects various errors in 
the presentations given there for $A_n$ and $S_n$. 
Thus, while the results of \cite{GKKL, JEMS} provide spectacular 
answers to long-standing challenging problems,
we believe that our presentations are necessary 
for our significant algorithmic application. 


\subsection{Our approach}
Our general strategy to construct a presentation for 
a classical group $\CH$ is straight-forward.

Our generating set usually contains generators for 
the normaliser $N$ of a maximal torus in a split $BN$ pair, 
and one root element of each type. 
For example, we take
$\left(\begin{smallmatrix}1&1\cr0&1\end{smallmatrix}\right)$ as our root 
element in $\SL(d,q)$,
and both a short and long root element in $\Sp(2n,q)$.  

We take, as our base cases, groups $H$ of minimum rank, 
subject to containing root 
elements of all types.  For example, $H=\SL(2,q)$ in the case of $\SL(d,q)$.
These groups lie naturally in the larger groups of the same type, and 
contain our chosen root elements.  

Our presentation  $\{ \mathcal{S} \; |\;  \mathcal{R} \}$ 
for $\CH$ always contain four sets of  
relations: a presentation for $N$; presentations for the subgroups $H$; 
relations that give the centraliser in $N$ of certain root elements;  
and certain Steinberg relations. 
In some cases extra relations are required when $\CH$ 
has an exceptional Schur multiplier.

For Chevalley groups, the sufficiency of the Steinberg relations 
is due to \cite{Steinberg62,Steinberg81};  
for the twisted groups it follows from \cite{Babaietal97}.
The Steinberg relations consist of relations that
hold in the root subgroups, and relations that express commutators of
root elements in terms of other root elements. 
A sufficient number of
relations of this second kind is included in $\mathcal{R}$ by 
ensuring that one is chosen, or is implied,  
from each orbit of the action of $N$ on the set of unordered
pairs of distinct root elements. 

\subsection{Comments on length} \label{length}
Formally, we define a presentation of a group $G$ to be 
of the form $\{ X,Y \, | \,W,R\}$, where $X$ is the given generating set, 
$Y=\{y_1,\ldots,y_n\}$ are {\it auxiliary} generators, 
$W=\{w_1,\ldots,w_n\}$ are words where $w_i$ is a 
word on $X\cup\{w_1,\ldots,w_{i-1}\}$ which defines $y_i$, and
$R$ is a defining set of relators for $G$, expressed as words in 
$X\cup Y$, and hence, with the use of $W$, as words in $X$.
The sets $Y$ and $W$ are simply tools for expressing the
relators $R$ efficiently as words in $X$,
and should not be counted either as generators or relations.  
For example, the generators $a=\Delta^x$ and $b=\Delta^y$ in 
the presentation for $\SU(3,q)$ given in Section \ref{SU3q}
are superfluous, so we do not count $a$ and $b$ as generators, 
nor their definitions as relations.
Variations of this approach may be used to convert 
a presentation of bit-length $N$ into one of word length $4N$.
We also store words as straight-line programs on $X \cup Y$; for 
a formal definition
and discussion of their significance, see \cite[p.\ 10]{Seress03}.
The most important use of straight-line programs rather 
than words is when we raise a word to a high power. There are 
fast algorithms for raising a matrix over a finite field to a high power 
(see \cite[Section 10]{odd}); so our straight-line programs encode 
a cell as a power of an earlier cell.  Hence, for our purposes, the 
critical metric to assess presentation length is bit-length.

\subsection{Notation}
The dihedral group of order $2n$, where $n$ is positive,  is 
$$D_{2n} := \langle a,b \; \vert  \; b^n=1, b^a=b^{-1}, a^2=1\rangle.$$
The generalised quaternion group of order $2n$, where $n$ is positive and
even, is 
$$Q_{2n} := \langle a,b \; \vert  \; b^n=1, b^a=b^{-1}, a^2=b^{n/2}\rangle.$$
The semidihedral group of order $2n$, where $n$ is positive and 
$n \equiv 0 \bmod 4$, is 
$$SD_{2n} := \langle a,b \; \vert  \; b^{n}=1, b^a=b^{-1 + n/ 2}, a^2=1\rangle.$$
Note that we generalise the usual definitions of these groups.

\section{Presentations for the symmetric group and related groups}
\label{symmetric}
Let $d$ be a positive integer. Define $\epsilon_d=(-1)^{d+1}$. 
If $\epsilon = \pm 1$, and $(a_1,a_2,\ldots,a_d)$ is a cyclic permutation, 
then we define the signed permutation 
$(a_1,a_2,\ldots,a_d)^{\epsilon}$ by $a_i\mapsto a_{i+1}$ for $i<d$, 
and $a_d\mapsto \epsilon a_1$.
In particular, $(a)^-$ is the transposition that interchanges $a$ and~$-a$.

For the symmetric group on $\{1,2,\ldots,d\}$, our generating set 
is $\{U,V\}$, where $U$ and $V$ 
stand for the permutations $(1,2)$ and $(1,2,\ldots,d)$ respectively. 

Let $U'=(1,2)^-$ and $V'=(1,2,\ldots,d)^{\epsilon_d}$ be
permutations on $\{\pm 1,\pm 2,\ldots, \pm d\}$;
so $\langle U',V'\rangle$ may be taken as the group 
of signed permutation matrices of determinant~$1$, 
a subgroup of index 2 in the group of all signed permutations of degree $d$.

Moore \cite[p.\ 358]{Moore} gave the following presentation for $S_d$. 
\begin{theorem}\label{symmetricpres}
The symmetric group $S_d$, for $d>2$, has the following presentation: 
$$\{U,V\;\vert\; U^2=V^d=(UV)^{d-1}=
(U U^V)^3 = (U U^{V^j})^2 = 1 {\rm\ for\ } 2 \leq j \leq d/2 \}.$$  
\end{theorem}
This presentation for $S_d$ on its standard generators has 
word length $O(d^2)$ and bit-length $O(d\log d)$,
and can easily be rewritten to have word length $O(d)$ by 
introducing auxiliary generators $V_j=V^j$, with relations $V_j = VV_{j-1}$.  

At the cost of introducing relations with 
bit-length $O(d \log d)$,
we could convert the presentation for $S_d$ in 
\cite[Theorem 1.2]{BCLO} of bit-length $O(\log d)$ to one on generators 
$U$ and $V$ that has an absolutely bounded number of relations.

We now give presentations for the group of signed permutation matrices 
of determinant~$1$ and degree $d$. 
By introducing auxiliary generators, these also achieve
word length $O(d)$. 
\begin{theorem}\label{signedsymmetricodd}
The group of signed permutation matrices of determinant $1$ and degree $d$, 
when $d>2$ is odd, has the following presentation:
$$\{U',V'\;\vert\; {U'}^4 = {U'}^{2V'U'}{U'}^2{U'}^{2V'}=
{V'}^d = (U'V')^{d-1}=(U' U'^{V'})^3=$$ 
$$ [U', U'^{V'^j}]  = 1
\; {\rm for}\; 2\le j\le(d-1)/2\}.$$
\end{theorem}
\begin{proof}
These relations are readily verified. 
The relations 
$[U', U'^{{V'}^j}] = 1\; {\rm for}\; 2\le j\le (d - 1) /2$ extend, 
by conjugation with $V'^{d-j}$,
to allow $2\le j\le d-2$.
Let $G$ denote the presented group, and $A$ the normal 
closure in $G$ of $a := {U'}^2$.  
By the previous theorem, $G/A$ is isomorphic to $S_d$, 
the relation $[U, U^{V^j}]=1$ being equivalent
to the relation $(UU^{V^{j}})^2=1$ in the presence of the relation $U^2=1$.  

We prove that $B := \langle \{a_{j+1}:=a^{{V'}^j} \,|\, 0\le j <d \}\rangle$ 
is an elementary abelian $2$-group.  
This generating set is permuted cyclically by $V'$. 
The relation ${U'}^{2V'U'}{U'}^2{U'}^{2V'}=1$ states 
that $a_2^{U'} a_1a_2 = 1$, so $a_1a_2$ is an involution,  
and hence $a_1$ commutes with $a_2$.
It follows, by conjugation with $V'^{-1}$,
that $a_d$ commutes with $a_1$.
But $a_1$ commutes with $a_j$
for $3\le j\le d-1$ since $U'$ commutes with $U'^{V^j}$ for $2\le j\le d-2$. 
So $a_1$ commutes with $a_j$ for all $j$, and
conjugating by ${V'}^i$ then proves that $a_i$ and $a_j$ commute for  
all $i$ and $j$.

We now prove that $a_1a_2\ldots a_d=1$.
Clearly $U'$ commutes with $a_{j+1}=(U'^{{V'}^j})^2$ for  
$2\le j\le d-1$, so
$$a_1^{(U'V')^{d-1}} = a_2^{(U'V')^{d-2}} = 
(a_2a_3)^{(U'V')^{d-3}} = 
(a_2a_3a_4)^{(U'V')^{d-4}} = \cdots = a_2a_3\ldots a_d.$$
But $(U'V')^{d-1} = 1$, so $a_1a_2\ldots a_d=1$, as required.
Hence $a_d^{U'}=a_1a_d$; so $B$ is a normal subgroup of $G$, 
and the result follows.
\end{proof}

The following is proved similarly.
\begin{theorem}\label{signedsymmetriceven}
The group of signed permutation matrices of determinant $1$ and degree $d$, 
when $d>2$ is even, has the following presentation:
$$\{U',V'\;\vert\; {U'}^4 = {U'}^{2V'U'}{U'}^2{U'}^{2V'}=
(U' U'^{V'})^3=  
[U', U'^{V'^j}] = 1\; {\rm for}\; 2\le j\le d/2,$$  
$${V'}^d = (U'V')^{d-1}, [V'^d,U'] = 1, V'^{2d} = 1 \}.$$
\end{theorem}
\noindent 
If $d=2$, then the group is cyclic of order 4.


We shall frequently need the following result.

\begin{theorem}\label{subgp} 
Let $d > 3$.
The group $H$ generated by $U'^{{V'}^2}$ and $V'U'^{-1}U'^{-V'}$ 
is the group of signed permutation matrices of 
$1,\ldots,d$ having determinant $1$ that 
fixes both $1$ and $2$. The group $G$ generated  by
$U'^{{V'}^2}$ and $V'U'U'^{V'}$ is the direct product of $H$ 
with the cyclic group of order $2$ generated by $(1)^-(2)^-$.
\end{theorem}

\begin{proof}
Observe that 
$U'^{{V'}^2}=(3,4)^-$ and $V'U'^{-1}U'^{-V'}=(3,4,\ldots,d)^{\epsilon_d}$,  
which proves the first statement.
Since $V'U'U'^{V'}=(1)^-(2)^-(3,4,\ldots,d)^{\epsilon_d}$ it follows
that $G$ maps onto $H$, with kernel $\langle(1)^-(2)^-\rangle$. 
The subgroup
of $G$ generated by $(3,4)^-$ and its conjugates is isomorphic to $H$, 
so the extension splits.
\end{proof}

\section{A presentation for $\SL(2,q)$}\label{CRW}

\subsection{Generators and notation}
Let $q = p^e$ for a prime $p$. Let $\omega$ be a primitive element of $\GF(q)$. 
We define the following elements of $\SL(2,q)$.

\smallskip
$\tau(\alpha) = \left(\begin{smallmatrix}1&\alpha\cr0&1\end{smallmatrix}\right)$ for $\alpha\in\GF(q)$;

\smallskip
$\delta =  \left(\begin{smallmatrix}\omega^{-1}&0\cr0&\omega\end{smallmatrix}\right)$;

\smallskip
$U =  \left(\begin{smallmatrix}0&1\cr-1&0\end{smallmatrix}\right)$.

\subsection{Presentations for $\SL(2, q)$ for $e > 1$}
We use presentations from \cite{CRW} rewritten on 
the generating set $\{\tau = \tau(1), \delta, U\}$.
A complication arises from the fact that
$\tau(\alpha)^\delta=\tau(\alpha\omega^2)$, 
so $\tau(\omega^i)$ is treated differently
according to the parity of $i$.  
In the presented groups given below, $\tau$ is
a generator, and $\tau_1$, standing for $\tau(\omega)$, 
is defined as a product of conjugates of $\tau$ 
by powers of $\delta$.  
Now $\tau_i$ may be defined for $0 \le i \le e$ by 
$\tau_0=\tau$ and $\tau_i = \tau_{i-2}^{\delta}$ if $i>1$,
so $\tau_i$ stands for $\tau(\omega^i)$.  
If $f(t) = \sum_{i=0}^e u_i t^i$ is a polynomial over $\GF(p)$ of degree $e$,
then we write $\tau^{f(t)}$ for $\prod_i\tau_i^{u_i}$.  

\begin{theorem} \label{CRW-PSL-odd} 
Let $q = p^e$ for an odd prime $p$ and $e> 1$.
Let $f(t)$ be the minimal polynomial of $\omega$ over $\GF(p)$. 
Write $\omega=\sum_{i=0}^{e-1}a_i\omega^{2i}$.
Let $m$ satisfy the equation 
$\omega^{2m}=1+\omega$ if $1+\omega \in \GF(q)^2$, 
and $\omega^{2m+1}=1+\omega$ otherwise.
 Then $\PSL(2, q)$ has the following presentation:
\begin{eqnarray*}
\{ \tau,\delta,U & \vert & 
\tau_1 := \prod_i\tau^{a_i\delta^i}\,,\,(\tau U)^3=(U\delta)^2=U^2=(\tau_1U\delta)^3=
\delta^{(q-1)/2}=\tau^p=1, \\
& & [\tau,\tau_1]=[\tau_1,\tau^\delta]=\tau^{f(t)}=1; 
\end{eqnarray*}
\hspace*{0.7cm} if $1+\omega\in\GF(q)^2$, then 
$\tau^{\delta^m}=\tau\tau_1,\tau_1^{\delta^m}=\tau_1\tau^\delta $; 
otherwise $\tau_1^{\delta^m}=\tau\tau_1,
\tau^{\delta^{m+1}}=\tau_1\tau^{\delta}\}$.
\end{theorem}

This is essentially due to Campbell, Robertson, and Williams 
\cite[Theorem 2.2]{CRW}.
We have replaced their 
generators $w,x,y,z$ by
$U\tau^{-1}, \tau,\tau_1,$ and $\tau\delta^{-1}$. 
The proof that the elements $\tau_i$ in the presented group $G$ commute 
goes through unchanged with our presentation.
They have relations equivalent to $\tau^p=1$ and $\tau_1^p=1$.  
The second of these is redundant, as $\tau_1$ is
a product of commuting conjugates of $\tau$.  
The relation $\tau^{f(t)}=1$ gives 
the action of $\delta$ on the
irreducible module generated by the $\tau_i$; 
their equivalent relation with 
$\tau$ replaced by $\tau_1$ is also redundant.  
Their relation $z^{(q-1)/2}=1$ translates to 
$(\tau\delta^{-1})^{(q-1)/2}=1$ in our notation, and we have 
replaced it by the relation $\delta^{(q-1)/2}=1$.
Given the action of $\delta$ on the module generated 
by the $\tau_i$, these relations are equivalent.
Finally we have replaced their relation $(wyz)^3=1$, which 
translates to $(U\tau^{-1}\tau_1\tau\delta^{-1})^3=1$,
by the relation $(\tau_1U\delta)^3=1$.  These are equivalent, 
since both $[\tau,\tau_1]=1$ and $U\delta=\delta^{-1}U$ hold in $G$.

\begin{theorem}\label{CRW-SL-odd}
A presentation for $\SL(2,q)$, where $q=p^e$ is odd and $e>1$, 
is obtained from the presentation for $\PSL(2,q)$ in 
Theorem $\ref{CRW-PSL-odd}$ by replacing the relations 
$$(\tau U)^3=1, \q (U\delta)^2=1, \q U^2=1, \q
\delta^{(q-1)/2}=1$$ by the relations 
$$(\tau U^{-1})^3=U^2, \q
(U\delta)^2=U^2, \q U^4=1, \q \delta^{(q-1)/2}=U^2.$$
\end{theorem}
\begin{proof}
The given relations all hold in $\SL(2,q)$. 
It remains to prove that $U^2$ is central 
in the presented group.  
The relations $(\tau U^{-1})^3=U^2$ and $(U\delta)^2=U^2$ 
imply that $\tau$ and $\delta$ commute with $U^2$, and the result follows.
\end{proof}

These presentations may be significantly simplified if $q\equiv 3\bmod 4$.
\begin{theorem}\label{CRW-PSL-3mod4}
With the notation of Theorem $\ref{CRW-PSL-odd}$, 
if $q\equiv 3\bmod 4$ and $e>1$,
then $\PSL(2,q)$ has the following presentation:
 $$\{ \tau,\delta,U\;\vert\;
(\tau U)^3=(U\delta)^2=U^2=[\tau,\tau^{\delta^{(q+1)/4}}]=\tau^{f(t)}=1, 
\delta^{(q-1)/2}=\tau^p,
\tau^{\delta^m}=[\tau^{-1}, \delta^r] \},$$
where $r=(q+1)/4$ if $1+\omega\in\GF(q)^2$, and $r= (q-3)/4$ otherwise.
\end{theorem}
 
This is essentially 
\cite[Theorem 2.4]{CRW}.  We have replaced the relation that  
translates to $(\tau\delta^{-1})^{(q-1)/2}=1$ by the 
simpler relation $\delta^{(q-1)/2}=1$, as in Theorem \ref{CRW-PSL-odd}.
The relation $\tau^p=1$ may be retrieved as follows.  
We have the relations $\delta^{(q-1)/2}=\tau^p$ and
$\tau^{\delta^m} \tau^{\delta^r} = \tau$, where the latter is 
a rewriting of $\tau^{\delta^m}=[\tau^{-1}, \delta^r]$.
Since $\tau^p$ commutes with $\delta$ it follows 
that $\tau^p \tau^{p} = \tau^p$.

Note that $\tau^{\delta^{(q+1)/4}}$ is $\tau(-\omega)$, and plays the 
role of $\tau_1=\tau(\omega)$ in Theorem \ref{CRW-PSL-odd}.  
It is because we can write this element as a conjugate of $\tau$ that 
we have a simpler presentation when $q\equiv3\bmod4$.

\begin{theorem}\label{CRW-SL-3mod4}
 If $q\equiv3\bmod4$ and $e > 1$, then a presentation 
 for $\SL(2,q)$
is obtained from the presentation for
 $\PSL(2,q)$ in Theorem $\ref{CRW-PSL-3mod4}$ by replacing the relations 
$$(\tau U)^3=1, \q (U\delta)^2=1, \q U^2=1, \q \delta^{(q-1)/2}=\tau^p$$ 
by the relations 
 $$(\tau U^{-1})^3=U^2, \q (U\delta)^2=U^2, \q U^4=1, \q 
\delta^{(q-1)/2}=\tau^pU^2.$$
 \end{theorem}
\noindent 
The proof is similar to that of Theorem \ref{CRW-SL-odd}.

Similar results hold for characteristic 2. 
Translated into our generating system, 
\cite[Theorem 3.2]{CRW} states the following.
\begin{theorem}\label{CRW-SL-even}
If $e>1$, and $m$ satisfies the equation $\omega^{2m}=1+\omega^2$, 
and the minimum polynomial of $\omega^2$ over $\GF(2)$
is $\sum_i u_i t^i$, 
then $\SL(2,2^e)$ has the following presentation:
$$\{\tau,\delta,U\;\vert\;(U\tau)^3=U^2 = (U\delta)^2=(\tau\delta)^{q-1}=\tau^2 = 1\;,\;
\tau^{\delta^m}=[\tau,\delta]\;,\; 
\prod_i\tau^{u_i\delta^i}=1\}.$$
\end{theorem}

\subsection{Presentations for $\SL(2, p)$} 
We use the following presentations from \cite{Cam-Rob80} rewritten on 
the generating set $\{\tau, U\}$.
\begin{theorem}
Let $p$ be an odd prime and let $k = \lfloor p / 3 \rfloor$.
If $p \equiv 1 \bmod 3$ then 
$\SL(2,p)$ has presentation
$$\{ \tau, U \; | \; 
U^{2} = (U \tau U^2)^3, {(U (\tau U^2)^4  U (\tau U^2)^{(p + 1) / 2})}^2 (\tau U^2)^p U^{2k}  = 1\},$$
else 
$$\{ \tau, U \; | \; 
U^{-2} = (U^{-1} \tau)^3,
{(U^{-1} \tau^4 U^{-1} \tau^{(p + 1) / 2})}^2  \tau^p  U^{-2k} = 1 \}.$$
\end{theorem}
Note that $\SL(2,2)$ has presentation 
$\{ \tau,U \, \vert  \, (\tau U)^3=U^2=\tau^2=1\}$.
If $e = 1$ then we may take $\omega$ as an integer. 
Observe that 
$\delta = {(\tau^{\omega-\omega^2})}^{U} \tau^{\omega^{-1}} {(\tau^{\omega - 1})}^{U}\tau^{-1}$;
in particular, $\delta$ is the identity and $U^2$ for $p = 2, 3$ respectively.

A presentation for $\PSL(2,p)$ is obtained by imposing the 
additional relation $U^2=1$.

\subsection{Standard generators for $\SL(2, q)$}
In Table 1 of \cite{even,odd} the non-trivial standard generators for 
$\SL(2, q)$ are labelled $s, t, \delta$.
Observe that $s = U; t=\tau$; and the standard generator
$\delta$ is the inverse of the presentation generator $\delta$.

\section{A presentation for $\SU(3,q)$}\label{SU3q}

\subsection{Generators and notation}
Let $q=p^e$ for a prime $p$.
Let $\omega$ be a primitive element of $\GF(q^2)$, and 
let $\omega_0=\omega^{q+1}$, 
so $\omega_0$ is a primitive element of $\GF(q)$.

Let $\xi=1/(1+\omega^{q-1})$ if $q$ is even, 
and $\xi=-1/2$ if $q$ is odd, and let $\zeta=-\omega^{(q^2+q)/2}$.  
So $\xi$ has trace $-1$, and $\zeta^2=\omega_0$, and $\zeta$ has trace 0.
Note that $\zeta=\omega^{(q+1)/2}$ if $q$ is odd.

We work with respect to the basis
$(e_1,v,f_1)$ for the underlying vector space, 
where the hermitian form is defined by $e_1.f_1=f_1.e_1=v.v=1$, 
and the form vanishes on all other pairs of basis vectors.  
Let $K$ be the root group of order $q^3$
consisting of all matrices of the form
$$\myupsilon(\alpha,\beta)=\left(\begin{matrix}
1&\alpha&\beta\cr
0&1&-\alpha^q\cr
0&0&1\cr
\end{matrix}\right),$$
where $\alpha$ and $\beta$ lie in $\GF(q^2)$, 
and $\beta$ has trace $-\alpha^{q+1}$,
for example, $\beta=\xi\alpha^{q+1}$.

Let $\myupsilon=\myupsilon(1,\xi)$,
and let $\tau=\myupsilon(0,\zeta)$ if $q$ is odd, 
and $\tau=\myupsilon(0,1)$ if $q$ is even.
Define 
$$\Delta(\alpha)=\left(\begin{matrix}
\alpha&0&0\cr
0&\alpha^{q-1}&0\cr
0&0&\alpha^{-q}\cr
\end{matrix}\right),$$
and let $\Delta=\Delta(\omega)$. 

Finally let $t=\left(\begin{matrix}
0&0&1\cr
0&-1&0\cr
1&0&0\cr
\end{matrix}\right)$.

If $q > 2$ then $\{\myupsilon,\tau,\Delta, t\}$ generates  
$\SU(3, q)$, but it generates only a subgroup of index 2 in $\SU(3, 2)$.

\subsection{A presentation for $\SU(3, q)$ for $q > 2$} \label{main-su3}
Using the work of \cite{HulpkeSeress01}, 
Guralnick {\it et al.} \cite[Theorem 4.10]{JEMS} give
a short presentation for $\SU(3, q)$ for $q > 2$ 
on $\{\myupsilon, \Delta^{-q}, t \}$.

Here, for completeness, we provide a variation of their presentation 
on generators $\{\myupsilon, \tau, \Delta, t \}$.
Let $H=\langle\myupsilon,\tau, \Delta\rangle$, a Borel subgroup 
of order $(q^2-1)q^3$.  
If $q$ is even, then relation $R_1$(iii) given below defines $\tau$, 
otherwise our presentation is independent of the choice of
a non-trivial element of 
$\langle \tau \rangle^H$; so this difference between the 
presentations is insignificant.
Following the existing approach, 
we first construct a presentation for $H$, and then
identify explicitly the additional relations implied 
by \cite{HulpkeSeress01} to obtain one for $\SU(3, q)$.


\subsubsection{A presentation for $H$}
\begin{lemma}\label{field_elts}
If $q\ne 2,3,5$ then there exist integers $x$ and $y$ such that
\begin{enumerate}
\item[(i)] $\omega^{x(q-2)}+\omega^{y(q-2)}=1$;
\item[(ii)] $\omega^{-x(q+1)}+\omega^{-y(q+1)}=1$;
\item[(iii)] $\GF(p)[\omega^{x(q+1)}] = \GF(q)$;
\item[(iv)] $\GF(p)[\omega^{x(q-2)}]$ is
equal to $\GF(q^2)$ or $\GF(q)$, according as $p$ is odd or even.
\end{enumerate}
\end{lemma}

\begin{proof}
This is \cite[Lemma 4.23]{GKKL}, with $a$ and $b$ replaced by $\omega^{-x}$
and $\omega^{-y}$ respectively, as can be seen by raising (i) to the power $q$.
\end{proof}

The proof of this lemma in \cite{GKKL} allows us to obtain $x$ and $y$ readily. 
If $q$ is even, then $\omega^{-x} = \omega_0$ and $\omega^{-y} = 1-\omega_0$.
Hence $x = -(q + 1)$, and $y = -c(q + 1)$ where we solve the equation 
$\omega_0^{c} = 1 - \omega_0$ for $c$ via a discrete log computation 
in $\GF (q)$.
Otherwise, define $\psi := \omega^{(q+1)/2}$
and $c_t := t(t^2+3\psi^2)(t+\psi)(t^2-\psi^2)^{-2}$ for $t\in\GF(q^2)$.
If $x$ and $y$ satisfy conditions (i) and (ii), 
then $\omega^{-x} = c_t$ and $\omega^{-y} = c_t - c_t^{q-1}$ for 
some $t\in\GF(q^2)$.
The number of such $t$ in $\GF(q^2)$ is always positive and is 
at least $q - 5 - 6(e-1)q^{1/2}$. 
Hence we determine $t$ by a random search, and so $x$ and $y$
via a discrete log computation in $\GF(q^2)$.


We repeatedly use the following which is 
\cite[Lemma 4.5]{GKKL}.
\begin{lemma} \label{trick}
Let $U_0$ and $W_0$ be subgroups of a group $G$, and let
$u,w,a,b$ be elements of $G$ satisfying the following conditions:
\begin{enumerate}
\item[(i)] $[a,b]=1$;
\item[(ii)] $\langle u,u^a,U_0\rangle = \langle u,u^b,U_0\rangle = \langle u^a,u^b,U_0\rangle$;
\item[(iii)] $\langle w,w^a,W_0\rangle = \langle w,w^b,W_0\rangle = \langle w^a,w^b,W_0\rangle$;
\item[(iv)] $[u^a,w]=[u^b,w]=1$;
\item[(v)] $U_0$ and $W_0$ are normalised by $\langle a,b\rangle$;
\item[(vi)] $[U_0,w] = [u,W_0] = 1$.
\end{enumerate}
Then $[\langle\, \{ u^c: c \in \langle a,b\rangle \} \,\rangle, \;
\langle\, \{ w^c : c \in \langle a,b\rangle \, \} \rangle]=1$.
\end{lemma}


We build up our presentation for $H$ on the generating 
set $\{ \myupsilon,\tau, \Delta \}$ in steps.
Let $G$ be the presented group, so that $G$ changes as relations are imposed.

Let $a=\Delta^{x}$ and $b=\Delta^{y}$, with $x$ and $y$ as 
in Lemma \ref{field_elts}.
When $e>1$, we choose $x$ such that $\gcd(x, q^2-1)=1$ whenever possible.
This guarantees conditions (iii) and (iv) of the lemma.

Let $U$ be the normal closure in $H$ of 
$\langle\myupsilon,\tau\rangle$, so $U$ is the group
of upper unitriangular matrices in $H$. 
Let $W$ be the normal closure of $\langle\tau\rangle$ in $H$, so 
$W$ is the centre of $U$.  

We use three sets of relations to present $H$. 
The first $R_1$ is the following. 
\begin{enumerate}
\item[(i)] $\Delta^{q^2-1} = 1$;
\item[(ii)] $a = \Delta^{x}$ and $b = \Delta^{y}$;
\item[(iii)] $\myupsilon^p = 1$  for odd $p$, and $\myupsilon^2=\tau$ for even $p$;
\item[(iv)] $\tau^p=1$.
\end{enumerate}

Suppose that $m(s)=\sum_ia_is^i$ is a polynomial over $\GF(p)$.  If $w\in W$ and $h\in\langle\Delta\rangle$,
we define $w^{m(h)}$ to be $\prod_i(w^{a_i})^{h^i}$, 
and if $u\in U$ we define $u^{m(h)}$ similarly, but here
the order in which the terms are multiplied is relevant.  Any fixed order, such as the natural order, may be used.

We now introduce a set $R_2$ of relations which, as we shall see, 
determines, with $R_1$, 
the structure of $W = \langle\{\tau^\delta : \delta \in \langle \Delta\rangle\}\rangle$.
\begin{enumerate}
\item[(i)] $\tau= \tau^a\tau^b$, and, if $p$ is odd, then $\tau = \tau^b\tau^a$;
\item[(ii)] if $e>1$ then $\tau^{m_1(a)} = 1$, 
where $m_1$ is the minimum polynomial of $\omega_0^{-x}$ over $\GF(p)$;
\item[(iii)] if $e=1$, or $\gcd(x,q^2-1) > 1$, then $\tau^{\Delta} = \tau^{m_2(a)}$, where $m_2$
is a polynomial of degree at most $e-1$ over $\GF(p)$, and 
$\omega_0^{-1} = m_2(\omega_0^{-x})$.
\end{enumerate}
Relation (i) is satisfied in $H$,
since $\omega^{-x(q+1)}+\omega^{-y(q+1)}=1$, 
Note that $\tau = \tau^b\tau^a$ 
also holds in $G$ in the case of even $p$, by $R_1$(iv) and $R_2$(i).
Observe that $[\tau,\tau^a]=[\tau,\tau^b]=1$,
since $\tau = \tau^a \tau^b$ and $[\tau^a,\tau^b] = 1$.

To see that $R_1$(ii), $R_2$(i), and this last observation 
imply that $X := \langle\{\tau^c:c\in\langle a,b\rangle\}\rangle$
is abelian, apply Lemma \ref{trick}, with
$U_0=W_0=\langle1\rangle$, and $u=w=\tau$.  
Since $a$ and $b$ commute, $R_2$(i) implies that
$X = \langle\{\tau^c:c\in\langle a\rangle\}\rangle$.
If $e=1$ then (iii) asserts that 
$\Delta$ normalises $\langle\tau\rangle$, so this group is equal to $W$,
as required.  If $e>1$ then (ii) implies that $X$ has 
order $p^e$.
If also $\gcd(x,q^2-1)=1$
then $\Delta\in\langle a \rangle$, so $\Delta$ normalises $X$, 
which is hence equal to $W$, as required.
But if $\gcd(x, q^2-1)>1$ then (iii) implies that $X$ is normalised by $\Delta$, as required.

We now introduce a set $R_3$ of relations that imply, 
with $R_1$ and $R_2$, that
$U/W$ is elementary abelian of order $q^2$.
\begin{enumerate}
\item[(i)] $\myupsilon = \myupsilon^a\myupsilon^b w_1$ for some $w_1\in W$, 
and if $p$ is odd then 
$\myupsilon = \myupsilon^b\myupsilon^a w_2$ for some $w_2\in W$;
\item[(ii)] $[\myupsilon^a,\tau]=[\myupsilon^b,\tau]=1$;
\item[(iii)] $[\myupsilon,\myupsilon^a] = w_3$ for some $w_3\in W$;
\item[(iv)] if $p$ is even, so $\GF(p)[\omega^{x(q-2)}] = \GF(q)$, 
and $e>1$, then 
$[\myupsilon^{\Delta},\myupsilon^a] = w_4$ 
and $[\myupsilon^{\Delta},\myupsilon^b] = w_5$ 
for some $w_4$ and $w_5\in W$;
\item[(v)] $\myupsilon^{m_3(a)} = w_6$ for some 
$w_6\in W$, where $m_3$ is the minimum polynomial of 
$\omega^{x(q-2)}$ over $\GF(p)$; 
\item[(vi)]  if $p$ is odd and $\gcd(x,q^2-1)>1$
then $\myupsilon^{\Delta}=\myupsilon^{m_4(a)}w_6$ for some $w_6\in W$, 
where $m_4$ is the
polynomial of degree at most $2e-1$ over $\GF(p)$ satisfying 
$\omega^{q-2} = m_4(\omega^{x(q-2)})$;
otherwise, if $p$ is even, then
$\myupsilon^{\Delta^2}=\myupsilon^{m_5(a)}\myupsilon^{\Delta m_6(a)}w_6$ for 
some $w_6 \in W$, 
where $m_5$ and $m_6$ are the polynomials 
of degree at most $e-1$ over $\GF(p)$ satisfying 
$\omega^{2q-4} = m_5(\omega^{x(q-2)}) + \omega^{q-2}m_6(\omega^{x(q-2)})$.
\end{enumerate}
Observe that (i) is satisfied in the matrix group since
$\omega^{x(q-2)}+\omega^{y(q-2)}=1$.
The elements $w_i$ need to be determined
by calculation in the matrix group.  

We first show that these relations imply that $U/W_1$ is elementary 
abelian, of order $q^2$, where $W_1=\langle\{\tau^g:g\in G\}\rangle$.
Let $Y=\langle\{\langle\myupsilon,\tau\rangle^c:
c\in\langle a,b\rangle\}\rangle<G$, and $Z = YW_1/W_1$.
Suppose first that $p$ is odd.
The fact that $Z$ is abelian
follows from Lemma \ref{trick} with $u=w=\myupsilon$, and $U_0=W_0=W_1$.
Since $a$ and $b$ commute, it follows from (i) that $Z=\langle\{\myupsilon^c:c\in\langle a\rangle\}\rangle W_1/W_1$.
The fact that $Z$ is elementary abelian follows from $R_1$(iii).
It now follows from (v) that $Z$ has order at most $q^2$, 
and hence has order exactly $q^2$.
Since $\Delta$ commutes with $a$, it follows from (vi) that 
$Y$ is normalised by $\Delta$.
So $Y=U$, and $U/W_1$ has order $q^2$, as required.
If $p$ is even, then we need to prove that
$Z$ has order $q$, and that $U/W_1$ is the direct sum of $Z$ and $Z^\Delta$.
The proof that $Z$ is elementary abelian goes through as 
in the case when $p$ is odd.
The fact that $Z$ has order $q$ follows from (v).
The fact that $Z$ and $Z^\Delta$ commute 
follows from Lemma \ref{trick} with $u=\myupsilon$, 
and $w=\myupsilon^{\Delta}$, and $U_0=W_0=W_1$.
The fact that the direct sum of these two groups is 
normalised by $\Delta$ follows from (vi).  
It remains to prove that $W_1=W$; in other words, that
$W$ is normalised by $U$.  
Since $W$ is normalised by $\Delta$, 
it suffices to prove that it is normalised by
$\myupsilon$, or, equivalently, by $\myupsilon^a$.  
We have seen that $W=\langle\{\tau^c:c\in\langle a,b\rangle\}\rangle$.
But $[\langle\{\tau^c:c\in\langle a,b\rangle\}\rangle, \langle\{\myupsilon^c:c\in\langle a,b\rangle\}\rangle]=1$ by Lemma \ref{trick},
with $u=\tau$, and $w=\myupsilon$, and $U_0=\langle1\rangle$, and $W_0=W$.

It follows that $G$ has order at most $q^3(q^2-1)$, and hence 
has precisely this order. 
Our analysis implies the following.
\begin{theorem} \label{Hpres}
If $H = \langle \Delta,\myupsilon,\tau\vert R_1\cup R_2\cup R_3\rangle$,
then $H$ is isomorphic to a Borel subgroup of $\SU(3,q)$.
\end{theorem}

\subsubsection{Completing the presentation for $\SU(3,q)$}
We follow \cite{HulpkeSeress01} with some variations.
Our presentation  for $G$ 
is $$\{ \myupsilon,\tau, \Delta, t \; \vert  \; 
R(1),R(2),\Delta^t=\Delta^{-q}, t^2 = 1\},$$
where $R(1)$ are the relations in a presentation for
$H$ on $\{\myupsilon,\tau, \Delta\}$, and $R(2)$ is to be specified.  
We omit the relation $\Delta^{q^2 - 1} = 1$ from $R(1)$: it is 
an obvious consequence of $\Delta^t = \Delta^{-q}$
and $t^2=1$.
Since our defining relations for $H$ are satisfied in
$\SU(3,q)$, we may take $H$ to be a subgroup of $G$.  
The relations $R(2)$ are of the form 
$u^t=u_Ldtu_R$ 
for $u, u_L, u_R \in U$, 
and $d$ in $D := \langle\Delta\rangle$; 
in more detail, if $u=\myupsilon(\alpha,\beta)$ and $\beta\ne0$, so $u\ne1$, then 
$d=\diag(\beta^{-q}, \beta^{q-1}, \beta)$, 
$u_L=\myupsilon(-\alpha\beta^{-q},\beta^{-1})$,  
and $u_R=\myupsilon(-\alpha\beta^{-1},\beta^{-1})$.
Note that $\SU(3,q)$ is the disjoint union of 
$UDtU$
 and $H$. 
Thus 
$UDtU$ 
is the set of elements that move the isotropic projective point
$[0,0,1]$, and $H$ is the set of elements that
fix $[0,0,1]$.  Hence it suffices to find relations $R(2)$ that imply 
that every element of $G$ lies in 
$UDtU\cup H$, 
where we abuse notation 
by using the same letter for $t$ as an
element of $G$, and as a matrix.  
We denote the relation 
$u^t=u_Ldtu_R$ 
by $P(u)$.  

Let $\UU$ denote the set of $u \in U\setminus 1_G$ for which the relation $P(u)$
may be deduced from the above presentation for a given choice of $R(2)$.

\begin{lemma} If $\UU = U\setminus\{1\}$
then $G$ is isomorphic to $\SU(3,q)$.
\end{lemma}
\begin{proof} It is easy to see from these relations that the product of two 
elements of $G$ of the form
$u_Ldtu_R$
 can again be written in this form, or as an element of $H$.  
Thus the union of these subsets of $G$ is a subgroup that contains 
the generators of $G$.  
Hence $G$ is equal to this union, and so is isomorphic to $\SU(3,q)$.
\end{proof}

\begin{lemma}\label{conj} $\UU$ is closed under conjugation by $\Delta$.
If $u$ and $v$ lie in $\UU$ then $uv\in \UU$ if and only if $u_Rv_L\in \UU$.
\end{lemma}
\begin{proof} Observe that $(u^\Delta)^t=(u^t)^{(\Delta^t)}$, 
and $\Delta^t=\Delta^{-q}$.  
Also $t^\Delta= (\Delta^{-1}\Delta^t)t$.  The result follows.
\end{proof}

\begin{lemma}\label{prod} If $u$ and $v$ lie in $\UU$ then $uv\in \UU$ if and only if $u_Rv_L\in \UU$.
\end{lemma}
\begin{proof}
This is simply a rearrangement of terms.
\end{proof}

\begin{lemma}\label{gamma}
Let $\beta\in\GF(q^2)$ and let $\eta\in\GF(q^2)$ with $\tr(\eta)\ne0\ne\tr(\beta)$.
There exists
$\gamma\in\GF(q^2)^{\times}$ with $\tr(\gamma)=0$, and $\beta+\gamma\in\GF(q)^{\times}\eta$.
\end{lemma}
\begin{proof}
Let $t=\tr(\beta)\tr(\eta)^{-1}$, and $\gamma=t\eta-\beta$.
\end{proof}

\begin{theorem}\label{PSU3qa} Suppose that $q\not\equiv2\bmod3$. 
Let $\beta_0=\omega\zeta=\omega^{1+(q^2+q)/2}$. 
Pick $\alpha_0$ such that $\alpha_0^{q+1}=-\tr(\beta_0)$.
By Lemma $\ref{gamma}$ there exists $\gamma_0$ with $\tr(\gamma_0)=0$ and 
$\beta_0+\gamma_0\in\GF(q)^{\times}\omega^{-1}\zeta$.
Let $\myupsilon_0 = \myupsilon(\alpha_0,\beta_0)$, 
let $\tau_0=\myupsilon(0,\gamma_0)$,
and let $U_0=\{\myupsilon_0,\tau_0,\myupsilon_0\tau_0\}$.
If $R(2) = \{P(u):u\in U_0\}$ then $G$ is isomorphic to $\SU(3,q)$.
\end{theorem}

\begin{proof}
Note that $\beta_0$ is chosen so that $\beta_0\zeta\in\GF(q)^{\times}\omega$, 
so $\beta_0\gamma\in\GF(q)^{\times}\omega$
for any $\gamma$ in $\GF(q^2)$ of trace 0, since $\tr(\zeta)=0$.

It suffices to prove that this choice of $R(2)$ implies that  $\UU=U\setminus\{1\}$.

Let $U_i=\{\myupsilon(\alpha,\beta) \, | \, \beta\in\GF(q)^{\times} \omega^{-i}\zeta, 
\alpha^{q+1}=-\tr(\beta)\}$.
So $U\setminus\{1\}$ is the disjoint union of the $U_i$, as $i$ runs from $0$ to $q$.  
Moreover $U_0$ contains $q-1$ elements,
and $U_i$ contains $q^2-1$ elements for $1\le i\le q$.  
Each $U_i$ is normalised by $\langle\Delta\rangle$, the
action being transitive in all cases.

We prove, by induction on $i$, that
$U_i$ is contained in $\UU$ for $0\le i\le q$. 
The case $i=0$ holds, since $\tau_0\in \UU$, 
and the case $i=1$ holds, since 
$\myupsilon_0\tau_0= \myupsilon(\alpha_0,\beta_0+\gamma_0)\in \UU$.  
Now assume inductively that $U_i$ is contained in $\UU$, where $0<i<q$.
By Lemma \ref{gamma} there exists $\gamma\in\GF(q^2)$ such 
that $\tr(\gamma)=0$ and $\beta_0+\gamma\in\GF(q)^{\times}\omega^{-i}\zeta$;
for, since $\tr(\zeta)=0$, the restriction on $i$ implies that the elements of 
this coset of $\GF(q)^{\times}$ do not have trace 0.
So $\myupsilon(\alpha_0, \beta_0+\gamma)$ lies in $\UU$, 
by the induction hypothesis, 
as do $\myupsilon(\alpha_0,\beta_0)$
and $\myupsilon(0,\gamma)$.  
Lemma \ref{prod} implies 
that $\myupsilon(-\alpha\beta_0^{-q}, \beta_0^{-1}+\gamma^{-1})\in \UU$.
But $\beta_0^{-1}+\gamma^{-1}=
(\beta_0+\gamma)/\beta_0\gamma\in\GF(q)^{\times}\omega^{-i}\zeta\omega^{-1}
=\GF(q)^{\times}\omega^{-(i+1)}\zeta$.
(Note that if $\tr(\gamma)=0$ then $\tr(\gamma^{-1})=0$.)
The result follows.
\end{proof}

\begin{theorem}\label{PSU3qb} Suppose that $q\equiv2\bmod3$. 
Let $\beta_0=\omega^{1+(q^2+q)/2}$. For $0\le i\le 2$ pick 
$\alpha_i\in\omega^i\left(\GF(q^2)^{\times}\right)^3$ 
such that $\alpha_i^{q+1}=-\tr(\beta_0)$.
By Lemma $\ref{gamma}$ there exists $\gamma_0$ with $\tr(\gamma_0)=0$ and 
$\beta_0+\gamma_0\in\GF(q)^{\times}\omega^{-1}\zeta$.
Let $\myupsilon_i = \myupsilon(\alpha_i,\beta_0)$, 
let $\tau_0=\myupsilon(0,\gamma_0)$,
and let $U_0=\{\myupsilon_i,\tau_0,\myupsilon_i\tau_0:0\le i\le2\}$.
If $R(2) = \{P(u):u\in U_0\}$ then $G$ is isomorphic to $\SU(3,q)$.
\end{theorem}

\begin{proof}
Note that $\omega^{q-1}$ is an element of $\GF(q^2)^\times$ that has 
norm 1 and is not a cube.
If $\alpha^{q+1}=-\tr(\beta_0)$ then $\alpha$ and $\alpha\omega^{q-1}$ and 
$\alpha\omega^{2(q-1)}$ all have norm equal to $-\tr(\beta_0)$, and
lie in different cosets of 
$\left(\GF(q^2)^{\times}\right)^3$ in $\GF(q^2)^\times$; 
so they may be taken to be the elements $\alpha_i$ of the theorem.

Define $U_i$, where $0\le i\le q$, as in the proof of Theorem \ref{PSU3qa}.  
Again $\langle\Delta\rangle$ acts transitively
on $U_0$, but partitions $U_i$, for $1\le 1\le q$, into three orbits, 
the orbit to which a given element $\myupsilon(\alpha,\beta)$ of
$U_i$ belongs depending on the coset of 
$\left(\GF(q^2)^{\times}\right)^3$ in $\GF(q^2)^\times$ to 
which $\alpha$ belongs.

The proof now follows that of Theorem \ref{PSU3qa}.
\end{proof}

\subsubsection{The remaining cases}

A presentation for $\SU(3, 3)$ is the following:
\begin{eqnarray*}
\{ \myupsilon, \tau, \Delta, t & | & 
   \myupsilon^3 = 1, \,
    t^2 = 1, \,
 \Delta^{-1} \myupsilon^{-1} \Delta^{-1} \tau^{-1} \myupsilon \Delta^2 \myupsilon^{-1} = 1 , \,
\Delta^{-1} \myupsilon \Delta^{-1} \tau^{-1} \myupsilon^{-1} \Delta^2 \myupsilon = 1,  \\
& & 
 \tau t \Delta^{-2} (\tau t)^2 = 1, \, 
 \Delta \tau^{-1} \myupsilon^{-1} t \Delta \tau^{-1} \myupsilon         
\Delta t \myupsilon \tau^{-1} t = 1 \}. 
\end{eqnarray*}

A presentation for $\SU(3, 5)$ is the following:
\begin{eqnarray*}
\{ \myupsilon, \tau, \Delta, t & | & 
    \myupsilon^5 = 1, \,
    \Delta \tau^{2} \Delta^{-1} \tau = 1, \, 
    \Delta^{2} \myupsilon^2 \Delta^{-2} \myupsilon^{-1} = 1, \, 
    \Delta^5 t \Delta t = 1, \, 
    \Delta^{-1} \myupsilon t \myupsilon^{-2} t \myupsilon \Delta t = 1, \\
& & 
    \Delta \myupsilon \Delta^{-1} \tau^{-1}  
    \myupsilon^{-1} \Delta \myupsilon^{-1} \Delta^{-1} \myupsilon = 1, \, 
    \myupsilon^{-1} t \tau^{-1} \myupsilon^{-1} t \Delta^{-1} 
     \myupsilon^2 t \Delta \myupsilon^{-1} t \Delta \tau^{-1} = 1 \}. 
\end{eqnarray*}
Their correctness is readily established using coset enumeration.

\subsection{A presentation for $\SU(3,2)$}\label{SU32}
We use defining relations for this (soluble) group in our 
presentations for $\SU(2n+1,2)$ when $n>1$.
Our chosen generating set is 
$\{\myupsilon = \myupsilon(1,\omega), \myupsilon' = \myupsilon(\omega^2,\omega), 
\Delta = \Delta(\omega),t\}$ where $\omega$ is a primitive element for $\GF(4)$.
Observe that 
 $\SU(3,2)$ is a split extension of a normal subgroup isomorphic to 
the group 
of order 27 and exponent 3 
by $Q_8$. This structure motivates the following presentation:
\begin{eqnarray*}
\{ \myupsilon,\myupsilon', \Delta, t & | & 
a := [\myupsilon,t], b := a^{2\myupsilon}, 
 a^{\myupsilon}=b^{-1}, b^{\myupsilon}=a\Delta, a^{\myupsilon'}=a b a,
 b^{\myupsilon'}=a b \Delta,  \\
& &   \myupsilon^2=\myupsilon'^2=[\myupsilon,\myupsilon'], t = \myupsilon^2 a^2 b\}.
\end{eqnarray*}
Its correctness is readily established using coset enumeration.

%
%
%
%

\subsection{A presentation for $\PSU(3, q)$}
The centre of $\SU(3,q)$ has order $\gcd (q + 1,3)$; 
if it is non-trivial, then it is generated by $\Delta^{(q^2-1) / 3}$. 

 \subsection{Standard generators for $\SU(3,q)$}
In Table 1 of \cite{even,odd} the non-trivial standard generators for 
$\SU(3, q)$ are labelled $s, t, \delta, x, y$.
Observe that $t = \tau; \delta = \Delta^{q + 1};
x = \myupsilon$; $y = \Delta$;
if $q$ is odd, then the standard generator $s = \Delta^{(q+1)/2}t$, 
else $s = t$.
If $q$ is odd, then the presentation generator 
$t = y^{-(q+1)/2}s$, else $t = s$.

\section{A presentation for $\SL(d,q)$ for $d>2$}
 
\subsection{Generators and notation}
Let $q=p^e$ for a prime $p$. Let $\omega$ be a primitive element of $\GF(q)$. 

Let $d>2$. 
We take a basis $(e_1,e_2,\ldots,e_d)$ of the natural module.  
We define the following elements of $\GL(d,q)$, where 
$1 \leq i, j \leq d$ and $i \ne j$, and $s\in\GF(q)$.

$\tau_{ij}(s)=(e_i\mapsto e_i+se_j)$; 

$\tau_{i,-j}(s)=\tau_{-i,j}(s)=\tau_{ij}(-s)$; 

$\tau_{-i,-j}(s)=\tau_{ij}(s)$; 

$\delta_{ij}=(e_i\mapsto\omega^{-1}e_i, e_j\mapsto\omega e_j)$; 

$U=(e_1,e_2)$;

$V=(e_1,e_2,\ldots,e_d)$;

$U'=(e_1,e_2)^-$;

$V'=(e_1,e_2,\ldots,e_d)^{\epsilon_d}$. 

\smallskip
All these elements of $\GL(d, q)$ centralise those basis 
elements that they are not stated to move.
All lie in $\SL(d,q)$ except for $U$ when $q$ is odd, 
and $V$ when $q$ is odd and $d$ is even.
If $q$ is even, then $U'=U$ and $V'=V$.

\subsection{A presentation for $\SL(d,q)$ for $e > 1$} 
We give a presentation for $\SL(d,q)$  on the generating set
$\{\tau=\tau_{12}(1),\delta=\delta_{12}, U', V'\}$.

Let $N$ be the subgroup of $\SL(d,q)$ generated by $\{\delta, U',V'\}$.
So $N$ is an extension of the direct product of $d-1$ copies of $C_{q-1}$, 
generated by
$\{\delta_{12},\delta_{23},\ldots,\delta_{d-1,d}\}$, by a 
copy of $S_d$, and has order $(q-1)^{d-1}d!$.

\begin{theorem}\label{SLN}
Let $G$ be the group generated by $\{\delta, U',V'\}$ subject to the 
relations given below. 
Then $G$ is isomorphic to $N$.
\begin{enumerate}
\item[(i)] Defining relations for $\langle U',V'\rangle$ as signed permutations 
of $\{1,2,\ldots,d\}$ if $q$ is odd,
or as unsigned permutations if $q$ is even.
\item[(ii)] $[\delta,U'^{V'^2}]=1$ if $d>3$, and $[\delta,V'U'U'^{V'}]=1$ if $d>4$.
\item[(iii)] $\delta\delta^{V'}=\delta^{V'U'}$.
\item[(iv)] $\delta^{U'}=\delta^{-1}$.
\item[(v)] $[\delta,\delta^{V'}]=1$, and if $d>3$ then $[\delta,\delta^{V'^2}]=1$.
\item[(vi)] If $q$ is even, then $\delta^{q-1}=1$, else $\delta^{(q-1)/2}=U'^2$.
\end{enumerate}
\end{theorem}
\begin{proof}
Since $q \geq 4$, the elements $\delta_{ij}$ are distinct.
The centraliser of $\delta$ in $\langle U',V'\rangle$ is 
generated by the subgroup 
that, as signed permutation group, fixes $1$ and $2$,
and by ${U'}^2$ and ${U'}^{2V'}$ if
$q$ is odd. 
For these elements centralise $\delta$, and generate a subgroup 
of $\langle U',V'\rangle$ of the
appropriate index.  
But ${U'}^2\in\langle\delta\rangle$, and ${U'}^{2V'}\in\langle\delta^{V'}\rangle$, 
so, by (ii) and (v), 
we may define $\delta_{ij}$ in $G$, 
these elements being permuted by the natural action of
$\langle U',V'\rangle$ acting as unsigned permutations.
By (v) $\delta_{12}$ commutes with $\delta_{23}$, 
and with $\delta_{34}$ if $d>3$; so all $\delta_{ij}$ commute 
with each other.
By (iv) $\delta_{12}=\delta_{21}^{-1}$, 
so $\delta_{ij}=\delta_{ji}^{-1}$ for all $i\ne j$.
By (iii) $\delta_{12}\delta_{23}=\delta_{13}$, 
so $\delta_{ij}\delta_{jk}=\delta_{ik}$ for all
$i\ne j\ne k\ne i$.
By (vi) $\delta_{12}^{q-1}=1$, so $\delta_{ij}^{q-1}=1$ for all $i\ne j$.
Thus $G$ has order at most $(q-1)^{d-1}d!$.  
But $N$ is a homomorphic image of $G$ of this order,
so $G$ is isomorphic to $N$, as required.
\end{proof}

\begin{theorem}\label{SLCent}
The centraliser of $\tau= \tau_{12}(1)$ in $N$ 
has index $(q-1)d(d-1)$ in $N$ and is generated by the following elements:
\begin{enumerate}
\item[(i)] ${U'}^{V'^2}$ if $d>3$;
\item[(ii)] $V'U' U'^{V'}$ if $d>4$;
\item[(iii)] $\delta^{V'^2}$ if $d>3$;
\item[(iv)] $\delta\delta^{2V'}$.
\end{enumerate}
\end{theorem}

\begin{proof}
Clearly $N$ acts transitively by conjugation on 
$\{\tau_{ij}(s) \, | \, s \in \GF(q)^\times \}$ since $d>2$.  Moreover
${U'}^{V'^2}$ (if $d>3$) and $V'U'U'^{V'}$ (if $d>4$) stand for the signed 
permutation matrices $(3,4)^-$ and
$(1)^-(2)^-(3,4,\ldots,d)^{\epsilon_d}$ respectively, and 
 $\delta\delta^{2V'}$ and $\delta^{V'^2}$ (if $d>3$) stand for the matrices 
 $\diag(\omega^{-1},\omega^{-1},\omega^2,1,\ldots,1)$ 
and $\diag(1,1,\omega^{-1},\omega,1,\ldots,1)$
 respectively.   These elements centralise $\tau$, 
and generate a subgroup of $N$ of order
 $(q-1)^{d-2}(d-2)!$. 
\end{proof}

\begin{theorem}\label{SLq}
Let $q=p^e$ for a prime $p$ and $e > 1$, and let $d>2$. 
Let $G$ be the group generated by $\{\tau,\delta, U',V'\}$
subject to the relations given below.  Then $G$ is isomorphic to $\SL(d,q)$.
\begin{enumerate}
\item[(i)] Defining relations for $N=\langle \delta, U',V'\rangle$ as 
in Theorem $\ref{SLN}$, 
but omitting relations (iv) and (vi) of that theorem. 

\item[(ii)] 
Relations that present $\SL(2, q)$ on $\{\tau, \delta, U'\}$.

\item[(iii)] Relations that state that the elements listed in 
Theorem $\ref{SLCent}$ centralise $\tau$.

\item[(iv)] The following instances of Steinberg relations:
\begin{enumerate}
\item[(a)] $[\tau,\tau^{V'}]=\tau^{U'^{V'}}$; \q  $([\tau_{12}(1),\tau_{23}(1)]=\tau_{13}(1))$;
 \item[(b)] $[\tau,\tau^{U'^{V'}}]=1$; \q $([\tau_{12}(1),\tau_{13}(1)]=1)$;
 \item[(c)] $[\tau,\tau^{U'V'}]=1$;  \q $([\tau_{12}(1),\tau_{32}(-1)]=1)$;
 \item[(d)] if $d>3$ then 
$[\tau,\tau^{V'^2}]=1$; \q $([\tau_{12}(1),\tau_{34}(1)]=1)$;
\item[(e)] if $q=4$ and $d=3$ then $[\tau,\tau^{\delta V'}]=\tau^{\delta U'^{V'}}$; 
\q $([\tau_{12}(1), \tau_{23}(\omega^2)]=\tau_{13}(\omega^2))$;
\item[(f)] if $q=4$ and $d=3$ then $[\tau,\tau^{\delta U'^{V'}}]=1$; \q $([\tau_{12}(1),\tau_{13}(\omega^2)]=1)$;
\item[(g)] if $q=4$ and $d=3$ then $[\tau,\tau^{\delta U'V'}]=1$; \q 
$([\tau_{12}(1), \tau_{32}(\omega^2)]=1)$. 
\end{enumerate}
\end{enumerate}
\end{theorem}


\begin{proof}
The omitted relations follow from (ii). 

By (iii) the elements $\tau_{ij}(s)$ (for $s\ne0$) may be uniquely defined in $G$ as 
$N$-conjugates of $\tau$.

The Steinberg relations for $\SL(d,q)$ are as follows, 
where $i$, $j$, $k$ and $l$ are distinct:
\begin{enumerate}
\item[(1)] $\tau_{ij}(s)\tau_{ij}(t)=\tau_{ij}(s+t)$;  $\tau_{ij}(s)^p=1$;
\item[(2)] if $d>3$ then $[\tau_{ij}(s),\tau_{kl}(t)]=1$;
\item[(3)] $[\tau_{ij}(s),\tau_{ik}(t)] = 1$;
\item[(4)] $[\tau_{ij}(s),\tau_{kj}(t)] = 1$;
\item[(5)] $[\tau_{ij}(s),\tau_{jk}(t)]=\tau_{ik}(st)$.
\end{enumerate}

Relations (1) hold in $G$ for $i=1$ and $j=2$ from (ii), 
and for other suffices by conjugation in $N$.

In $\SL(d,q)$, if $d>4$, or if $d=4$ and $q$ is even, 
then $N$ acts transitively by conjugation 
on the set of ordered pairs $(\tau_{ij}(s),\tau_{kl}(t))$, where
$i,j,k,l$ are distinct and $s\ne0\ne t$.  So in our presentation, 
for these values of $d$ and $q$, we need one relation of the
form $[\tau_{ij}(s),\tau_{kl}(t)]=1$ with suffices satisfying 
these conditions.  
Relation (d) is one such. 
Thus relations (2) hold in $G$ if $d>4$, or if $q$ is even.
If $d=4$ and $q$ is odd then the action of $N$ on these pairs gives rise
to two orbits, one containing the pairs $(\tau_{ij}(s),\tau_{kl}(t))$ 
where $st$ is a square in $\GF(q)^\times$, and the
other containing the pairs where $st$ is not a square.  
Now for fixed $i,j,k,l,t$, the set of $s\in\GF(q)$  
for which $[\tau_{ij}(s),\tau_{kl}(t)]=1$ in $G$ is closed under addition 
(by (a)),  and under multiplication by $\omega^2$ 
(by conjugation with $\delta_{ij}$). 
Since $\omega^2$ generates $\GF(q)$ as a field, 
$[\tau_{ij}(s),\tau_{kl}(t)]=1$ 
for all $s$ in $\GF(q)$ provided that this holds 
for one $s\in\GF(q)^{\times}$, and so (2) holds in $G$.

Similarly, if $d>3$ then $N$ acts transitively by conjugation on the set of 
pairs $(\tau_{ij}(s),\tau_{ik}(t))$, 
where $i,j,k$ are distinct and
$s\ne0\ne t$; and we have the relation (b), standing for 
$[\tau_{12}(1),\tau_{13}(1)]=1$.  But if $d=3$ then the 
setwise stabiliser 
of the set of pairs 
$(\tau_{12}(s),\tau_{13}(t))$, with $s\ne 0\ne t$, is the group of 
diagonal matrices in $\SL(d,q)$,
and this group, which has order $(q-1)^2$, acts faithfully 
unless $q\equiv1\bmod 3$, when there
is a kernel of order $3$.
So if $q\equiv1\bmod3$ then we have
three orbits of $N$ on these pairs.  
For fixed $s\in\GF(q)^{\times}$, let $S$ be the set of $t\in\GF(q)$
such that $[\tau_{12}(s),\tau_{13}(t)]=1$ in $G$.
Now $S$ is closed under addition, 
and (by conjugation with $\delta_{12}\delta_{23}^2$) is 
closed under multiplication by $\omega^3$, and hence is 
a vector space over the
field generated by $\omega^3$.  If $q\ne 4$ then this field is $\GF(q)$, 
and a similar result 
holds with the roles of $s$ and $t$ interchanged.  
Thus, in this case, the relations (3) may be deduced from a single case.  
But if $q=4$ the field generated by $\omega^3$ is the prime field, and we need 
relation (f), which asserts that $[\tau_{12}(1),\tau_{13}(\omega^2)]=1$. 
Hence $[\tau_{12}(1),\tau_{13}(1+\omega^2)]=1$, 
so we have representatives of the three orbits of $N$, as above, 
thus (3) holds in all cases.

By a similar argument, relations (4) may be deduced, 
using relations (c) and (g).  

The proof of (5) is similar. If $q\not\equiv1\bmod3$, or $d>3$, 
then the result follows from (a) by conjugation 
with $N$.  Otherwise, we first need to prove that in
$G$, for given $s\in\GF(q)^\times$, the 
set $S$ consisting of those $t$ for which 
$[\tau_{12}(s),\tau_{23}(t)] = \tau_{13}(st)$ is closed under addition, 
and under multiplication by $\omega^3$.
This follows from the identity $[a,bc]=[a,c][a,b]^c$,
using the relation (a), and by conjugation with $\delta_{12}\delta_{23}^2$.  
This completes the proof
of (5), except in the case when $q=4$ and $d=3$, when we 
use relation (e).
 
Finally one must check that the given generators lie in the group 
generated by the $\tau_{ij}(s,t)$.
This holds for $\tau$, $\delta$ and $U'$, thus for 
all $N$-conjugates of $U'$, and so for $V'$.
\end{proof}
Coset enumeration shows that relations (f) and (g) are redundant.

\subsection{A presentation for $\SL(d, p)$} \label{pres-sldp}
We give a presentation for $\SL(d, p)$ on the generating set $\{\tau,U',V'\}$. 
Clearly $\langle U',V'\rangle$ acts transitively by conjugation 
on the set of elements $\tau_{ij}(\pm1)$, a set of size
$2d(d-1)$ if $p$ is odd, and of size $d(d-1)$ if $p=2$.  
If $p$ is odd, then the centraliser of $\tau$ in $\langle U',V'\rangle$ 
is the direct product of the cyclic group of order 2 
generated by $U'^2$ with the subgroup 
of $\langle U',V'\rangle$
that, as signed permutation group, fixes $1$ and $2$.  
This subgroup has index $2d(d-1)$, as required.
If $p=2$, then this centraliser is the subgroup 
of $\langle U,V\rangle$ that fixes $1$ and $2$, and has index  
$d(d-1)$, as required.

\begin{theorem}\label{SLp}
Let $p$ be a prime and let $d>2$. 
Let $G$ be the group generated by $\{\tau, U',V'\}$
subject to the relations given below.  Then $G$ is isomorphic to $\SL(d,p)$.
\begin{enumerate}
\item[(i)] Defining relations for $\langle U',V'\rangle$ as signed permutations 
of $\{1,2,\ldots,d\}$ if $p$ is odd, or as unsigned permutations if $p=2$.
\item[(ii)] 
Relations that present $\SL(2, p)$ on $\{\tau, U'\}$.
\item[(iii)] If $d>3$ then $[\tau,U'^{V'^2}]=1$, 
and if $d>4$ then $[\tau,V' U'^{-1} (U'^{-1})^{V'}]=1$.
\item[(iv)] The Steinberg relations (iv)(a) to (iv)(d) of Theorem $\ref{SLq}$.
\end{enumerate}
\end{theorem}
\begin{proof}
That $U'^2$ centralises $\tau$ follows from (ii).
Hence, by virtue of (i) and (iii), if $p$ is odd then we may define
$\tau_{ij}(\pm 1)$ as elements of $G$ by conjugation with 
$\langle U',V'\rangle$.
The proof is concluded as in Theorem \ref{SLq}.
\end{proof}
\noindent
Observe that 
$\delta = {(\tau^{\omega-\omega^2})}^{U'} \tau^{\omega^{-1}} {(\tau^{\omega - 1})}^{U'}\tau^{-1}$.

\subsection{A presentation for $\PSL(d,q)$}
The centre of $\SL(d,q)$ has order $\gcd(q-1,d)$; if it is non-trivial, then   
it is generated by $(\delta U' {V'}^{-1})^{(d-1)(q-1)/\gcd(q-1,d)}$.
For $(\delta U'{V'}^{-1})^{d-1}$ is a diagonal matrix whose entries, with
one possible exception, equal $-\epsilon_d\omega$. 

\subsection{Standard generators for $\SL(d, q)$}
In Table 1 of \cite{even,odd} the non-trivial standard generators for 
$\SL(d, q)$ are labelled 
$s, t, \delta, v$ where, for ease of exposition,
we make one uniform choice for $v$: namely, 
$$v := \left(\begin{matrix} 0 & 1 \cr -I_{d - 1} & 0 \cr \end{matrix}
\right).$$
Observe that $s = U'$; $t=\tau$; the standard generator
$\delta$ is the inverse of the presentation generator
$\delta$; and 
$v=(V'U'^{-1})^{d-1}V'^{-1}$ if $d$ and $q$ are 
both odd, else $v=V'^{d-1}$. 
The presentation generator $V'=v^{-1}\left(s^{-1}v\right)^{d-1}$
if $d$ and $q$ are odd, else $V'=v^{d-1}$. 






\section{A presentation for $\Sp(2n,q)$ for $n > 1$}

\subsection{Generators and notation}
Let $q=p^e$ for a prime $p$. Let $\omega$ be a primitive element of $\GF(q)$. 

Let $n>1$. 
We take a hyperbolic basis $(e_1,f_1,\ldots,e_n,f_n)$ of the natural module,  
and define the following elements of $\Sp(2n,q)$, 
where $1\le i,j\le n$ and $i\ne j$, and $s\in\GF(q)$.

$Z_i=(e_i\mapsto f_i, f_i\mapsto -e_i)$;

$\delta_i=(e_i\mapsto\omega^{-1}e_i, f_i\mapsto\omega f_i)$;

$\tau_i(s)=(e_i\mapsto e_i+sf_i)$;

$\tau_{-i}(s) =(f_i\mapsto f_i-se_i) = \tau_i(s)^{Z_i}$;

$\sigma_{ij}(s)=(e_i\mapsto e_i+se_j, f_j\mapsto f_j-sf_i)$;

$\sigma_{-i,j}(s)= (f_i\mapsto f_i+se_j, f_j\mapsto f_j+se_i) =\sigma_{ij}(s)^{Z_i}$;

$\sigma_{i,-j}(s) =(e_i\mapsto e_i+sf_j, e_j\mapsto e_j+sf_i) =\sigma_{ij}(s)^{Z_j}$;

$\sigma_{-i,-j}(s) =(f_i\mapsto f_i+sf_j, e_j\mapsto e_j-se_i)=\sigma_{ji}(-s) =\sigma_{ij}(s)^{Z_iZ_j}$;

$U=(e_1,e_2)(f_1,f_2)$;

$V=(e_1,e_2,\ldots,e_n)(f_1,f_2,\ldots,f_n)$ if $n>2$.

\smallskip
All these elements of $\Sp(2n,q)$ centralise those basis elements that they 
are not stated to move.  
Note that $\sigma_{-i,j}^{Z_i}(s) = \sigma_{i,-j}^{Z_j}(s) = \sigma_{ij}(-s)$.

\subsection{A presentation for $\Sp(2n, q)$ for $e > 1$}
We give a presentation for $\Sp(2n,q)$  on the generating set
$\{ \sigma=\sigma_{12}(1), \tau=\tau_1(1), \delta=\delta_1, Z = Z_1, U,V\}$, 
omitting $V$ if $n=2$.

If $q$ is even, then the subgroup of $\Sp(2n,q)$ obtained by 
omitting $\tau$ is $\SO^+(2n,q)$.

Let $N_1$ be the subgroup of $\Sp(2n, q)$ generated by $\{ Z,U,V\}$,
omitting $V$ if $n=2$.  
It is isomorphic to $C_2\wr S_n$ if $q$ is even 
and to $C_4\wr S_n$ if $q$ is odd. 
The following is clear.
\begin{theorem}\label{SpN1}
Let $G$ be the group generated by $\{Z, U,V\}$, 
omitting $V$ if $n=2$,  
subject to the relations given below. 
Then $G$ is isomorphic to $N_1$.
\begin{enumerate}
\item[(i)] Defining relations for $S_n$ on $\{U,V\}$, omitting $V$ if $n=2$.
\item[(ii)] $Z^m=1$, where $m=2$ if $q$ is even, and $m=4$ if $q$ is odd.  
\item[(iii)] $[Z,U^V]=1$ if $n>2$, and $[Z,VU]=1$ if $n>3$.
\item[(iv)] $[Z,Z^U]=1$.
\end{enumerate}
\end{theorem}

Let $N$ be the subgroup of $\Sp(2n,q)$ generated by 
$\{\delta,Z, U,V\}$.
Now $K=\langle \delta, Z\rangle$ is a copy of 
$D_{2(q-1)}$ if $q$ is even, 
and of $Q_{2(q-1)}$ if $q$ is odd; and $N$ is isomorphic to $K\wr S_n$.

\begin{theorem}\label{SpN}
Let $G$ be the group generated by $\{\delta,Z,U,V\}$ subject 
to the relations given below. 
Then $G$ is isomorphic to $N$.
\begin{enumerate}
\item[(i)] Defining relations for $\langle Z, U, V \rangle$, 
as in Theorem $\ref{SpN1}$.
\item[(ii)] $\delta^{q-1}=1$ if $q$ is even, and 
$\delta^{(q-1)/2}=Z^2$ if $q$ is odd, and $\delta^Z=\delta^{-1}$.  
\item[(iii)] If $n>2$ then $[\delta,U^V]=1$, and if $n>3$ then
$[\delta,VU]=1$.
\item[(iv)]  $[Z,\delta^U]=[\delta,\delta^U]=1$.
\end{enumerate}
\end{theorem}

\begin{proof}
Since $U^V$ and $VU$ stand for the permutations 
$(2,3)$ and $(2,3,\ldots,n)$ respectively, 
the result is clear.
\end{proof}

%
%
%
%
%
%
%
%
%
%
%
%
%
%
%
%
%
%
%

\begin{theorem}\label{SpCen}  
\mbox{}
The centraliser of $\sigma = \sigma_{12}(1)$ in $N$ is generated 
by the following elements:
\begin{enumerate}
\item[(i)] $ZUZ^{-1}$;
\item[(ii)] $U^{V^2}$ if $n>3$;
\item[(iii)] $VUU^V$ if $n>4$; 
\item[(iv)] $Z^{V^2}$ if $n>2$;
\item[(v)] $\delta^{V^2}$ if $n>2$;
\item[(vi)] $\delta\delta^V$.
\end{enumerate}
It is isomorphic to $K \times (K\wr S_{n-2})$.  

The centraliser of $\tau = \tau_1(1)$ in $N$ is generated by the following elements:
\begin{enumerate}
\item[(i)] $Z^U$; 
\item[(ii)] $U^V$ if $n>2$; 
\item[(iii)] $VU$ if $n>3$; 
\item[(iv)] $Z^2$ if $q$ is odd;
\item[(v)] $\delta^U$. 
\end{enumerate}
If $q$ is even it is isomorphic to $K\wr S_{n-1}$,
else to $\langle Z^2\rangle \times (K\wr S_{n-1})$. 
\end{theorem}

\begin{proof}
Observe that $N$ acts transitively by conjugation on the 
set of $2n(n-1)(q-1)$ short root elements 
$\sigma_{ij}(s)$, where $0<|i|\ne |j|\le n$
and $s\in\GF(q)^{\times}$, bearing in mind that
$\sigma_{ij}(s)=\sigma_{-j,-i}(-s)$ for any $i,j,s$.  
The centraliser of $\sigma$ in $N$ contains the 
direct product of 
a subgroup that is isomorphic to $K\wr S_{n-2}$ with another copy of $K$, 
namely $\langle \delta_1\delta_2, ZUZ^{-1} \rangle$.
Since this direct product has index $2n(n-1)(q-1)$ in $N$, 
it is the centraliser of $\sigma$.

Also, $N$ acts by conjugation on the 
set of $2n(q-1)$ long root elements $\tau_i(s)$,
where $0<|i|\le n$ and $s\in\GF(q)^{\times}$.  
The action is transitive if $q$ is even, and there
are two orbits of equal size if $q$ is odd, 
$\tau_i(s)$ falling into one or other orbit according
as $s$ is or is not a square.  
The elements $U$ and $V$ permute the suffices, positive or negative, of the 
root elements $\tau_i(s)$ in the natural way, and $Z_i$ interchanges the 
suffices $i$ and $-i$ and fixes all other suffices.
If $i>0$ then $\tau_i(s)^{\delta_i}=\tau_i(s\omega^2)$, 
and $\tau_{-i}(s)^{\delta_i}=\tau_{-i}(s\omega^{-2})$;
and $\delta_i$ centralises $\tau_j$ if $|j|\ne i$.  
Thus the stated elements centralise $\tau$, and 
generate a subgroup of $N$ of index $2n(q-1)$ if $q$ is even, and
$n(q-1)$ if $q$ is odd.
\end{proof}

\begin{theorem}\label{SpSL2}
Let $\Delta=[U,\delta]$; and $U'=UZ^2$ if $q$ is odd, $U'=U$ if $q$ is even.
The subgroup of $\Sp(2n,q)$ generated by 
$\{\sigma, \Delta, U'\}$ is isomorphic to $\SL(2,q)$.
\end{theorem}
\begin{proof}
The isomorphism is obtained by restriction to the subspace $\langle e_1,e_2\rangle$, 
or $\langle f_1,f_2\rangle$,
these two representations being connected by the inverse transpose automorphism.
\end{proof}

\begin{theorem}\label{Sp}
Let $q=p^e$ for a prime $p$ and $e>1$, and let $n\ge 2$.  
Let $G$ be the group generated by $\{\sigma, \tau,\delta,Z, U,V\}$, 
omitting $V$ if $n=2$, subject to the relations given below.  
Then $G$ is isomorphic to $\Sp(2n,q)$.  
\begin{enumerate}
\item[(i)] Defining relations for $N=\langle \delta, Z, U,V\rangle$, 
omitting $V$ if $n=2$, as in 
Theorem $\ref{SpN}$, but omitting relations (ii) of that theorem,
and relation (ii) of Theorem $\ref{SpN1}$.

\item[(ii)] Relations that state that the elements 
listed in Theorem $\ref{SpCen}$ centralise $\sigma$ and $\tau$, 
but omitting the relation $[\tau,VU]=1$ if $q$ is odd.

\item[(iii)] Relations that present $\SL(2,q)$ on $\{\sigma,\Delta,U'\}$
as in Theorem $\ref{SpSL2}$.

\item[(iv)] Relations that present $\SL(2,q)$ on $\{\tau, \delta,Z \}$.

\item[(v)] The following instances of Steinberg relations:
\begin{enumerate}
\item[(a)] if $n>2$ then $[\sigma,\sigma^V]=\sigma^{VU}$; \q$([\sigma_{ij}(s),\sigma_{jk}(t)]=\sigma_{ik}(st))$;
\item[(b)] if $n>2$ then $[\sigma,\sigma^{VU}]=1$; \q$([\sigma_{ij}(s),\sigma_{ik}(t)]=1)$;
\item[(c)] if $n>2$ then $[\sigma,\sigma^{UV}]=1$; \q$([\sigma_{ij}(s),\sigma_{kj}(t)]=1)$;
\item[(d)] if $n>3$ then $[\sigma,\sigma^{V^2}]=1$; \q$([\sigma_{ij}(s),\sigma_{kl}(t)]=1)$;
\item[(e)] if $q$ is odd then $[\sigma,\sigma^Z]=\tau^{2ZU}$, 
else $[\sigma,\sigma^Z]=1$; \q 
\\ $([\sigma_{ij}(s),\sigma_{-i,j}(t)]=\tau_{-j}(2st))$;
\item[(f)] $[\sigma,\tau]=1$; \q$([\sigma_{ij}(s),\tau_i(t)]=1)$;
\item[(g)] $[\sigma,\tau^U]=\sigma^{Z^U}\tau^{-1}$; \q$([\sigma_{ij}(s),\tau_j(t)]=\sigma_{i,-j}(st)\tau_i(-s^2t))$;
\item[(h)] if $n>2$ then $[\sigma,\tau^{V^2}]=1$; \q$([\sigma_{ij}(s),\tau_k(t)]=1)$;
\item[(j)] $[\tau,\tau^U]=1$; \q$([\tau_i(s),\tau_j(t)]=1)$. 
\end{enumerate}
The given relations are instances of the more general relations 
in parentheses, with $s=t=1$ and $(i,j,k,l)=(1,2,3,4)$.  
In the parenthetical version, the moduli of $i,j,k,l$ are distinct.
\end{enumerate}
\end{theorem}

\begin{proof}
The double suffices $(i,j)$ for $\sigma$ and the suffices $i$ for $\tau$ correspond to 
short and long roots respectively; and if $r$ is the root corresponding to $(i,j)$ 
or $i$, then $-r$ is the root corresponding 
to $(-i,-j)$ (equivalently to $(j,i)$) or $-i$ respectively.

The relations omitted from (i) follow from (iv).  

By (i) we may define the elements $\{\delta_i,Z_i, U,V\}$ in $G$, and 
hence also the elements $\Delta$ and $U'$.

By Theorem \ref{SpCen}, and using (i) and (ii), we may define 
$\{\sigma_{ij}(s) \, | \, 0<|i|\ne|j|\le n,s\in\GF(q)^\times\}$, as
a subset of $G$.

The relation $[\tau,VU]=1$, omitted from (ii) if $q$ is odd, 
may now be supplied as follows.  
It follows from (v)(e) that the subgroup of 
$\langle U,V\rangle$ that centralises 1 and 2, 
as permutations of $\{1,2,\ldots,n\}$, 
centralises $\tau$ as a subgroup of $G$.
But this subgroup, together with $U^V$ (see (ii)),
generates the symmetric group on
$\{2,3,\ldots,n\}$, and hence contains $VU$.

By Theorem \ref{SpCen}, and using (i) and (ii), we may define 
$\{\tau_i(s) \, | \, 0<|i|\le n,s\in\GF(q)^2\}$ as
a subset of $G$, and obtain the centralisers of these elements in $N$.  
Suppose now that $q$ is odd.
Using (iv) we may define $\tau_1(\omega)$ in $G$,
and writing $\omega$ as a sum of squares gives $\tau_1(\omega)$ as a product
of elements of the form $\tau_i(t^2)$ for $t\in\GF(q)^\times$.  But these elements all
have the same centraliser $C$ in $N$.  So $C$ also centralises $\tau_1(\omega)$;
since it has the expected index in $N$ it must be the full centraliser.
So we may define $\tau_i(s)$ for all $i$, and for all $s\in\GF(q)\setminus\GF(q)^2$,
as $N$-conjugates of $\tau_1(\omega)$.

We now turn to the Steinberg relations.

The relations $\sigma_{ij}(s+t) =\sigma_{ij}(s)\sigma_{ij}(t)$ and $\sigma_{ij}(s)^p=1$
follow from (ii) and (iii), and the
relations $\tau_i(s+t)=\tau_i(s)\tau_i(t)$ 
and $\tau_i(s)^p=1$ follow similarly from (ii) and (iv).

The remaining Steinberg relations are (a) to (j), and we prove that the general case, 
given in parentheses, is implied by the given relation.  

The general case of relations (a) to (d) clearly follow from the 
one given instance by conjugation with $N$.

For fixed $i$, $j$, and $s$ the set of values $t$ for which the general 
case of relation (f) holds is 
closed under addition and
multiplication by $\omega^2$ (conjugate by $\delta_i\delta_j$) and hence is either 
$\GF(q)$ or $\{0\}$, since $\omega^2$ generates $\GF(q)$ as a field.  
It follows, from conjugation by $N$, that all instances of (f) follow from the given instance.

Let $P$ be the set of pairs $(s,t)$ for which relation (e) holds for one 
(and hence every) value of $(i,j)$.  Conjugation by $\delta_i$ and $\delta_j$ shows that 
if $(s,t)\in P$, 
then $(\omega s,\omega^{-1}t)\in P$, and $(\omega^{-1}s,\omega t)\in P$; so $(\omega^2 s,t)\in P$.  
If $(s_1,t)\in P$ and $(s_2,t)\in P$ then $(s_1+s_2,t)\in P$, thanks to (f), 
which implies that 
$[\sigma_{ij}(s),\tau_{-j}(k)] =[\sigma_{-i,j}(s),\tau_{-j}(k)]=1$.  
As in the proof of (f), this suffices to prove the general case of (e).

Similarly the general cases of (h) and (j) follow from a single case.

To prove the general case of (g), note that $\tau_j(t)$ commutes 
with the RHS of this 
identity by (f) and (j), since $\sigma_{i,-j}(st)=\sigma_{j,-i}(-st)$.  
The proof is similar to that of (e).

Variants of (e), (f), and (g) may be obtained by changing signs of suffices,
but these follow from the given relations.  
For example, for $i$ and $j$ positive,
\begin{eqnarray*} 
[\sigma_{ij}(s),\tau_{-i}(t)] = [\sigma_{-i,j}(-s),\tau_i(t)]^{Z_i} & = & 
[\sigma_{-j,i}(s),\tau_i(t)]^{Z_i} 
 = [\sigma_{ji}(s),\tau_i(t)]^{Z_iZ_j} \\ 
& = & \sigma_{j,-i}(st)\tau_j(-s^2t)=\sigma_{i,-j}(-st)\tau_j(-s^2t).
\end{eqnarray*}

So all the Steinberg relations hold in $G$.
Thus the subgroup $H$ of $G$ generated by the Steinberg generators
is isomorphic to $\Sp(2n,q)$.  
Also, $Z$ and $\delta$ and $U'$ lie in this group by (iii) and (iv),
and hence $U$ lies in $H$.  But $H$ is normalised by $\langle U,V\rangle$, 
and $\langle U,V\rangle$ is isomorphic to $S_n$, by (i), so $V\in H$.  
Thus $G=H$, as required.
\end{proof}

\subsection{A presentation for $\Sp(2n, p)$}
We give a presentation for $\Sp(2n,p)$  on the generating set
$\{ \sigma, \tau, Z, U, V\}$, omitting $V$ if $n=2$.
Let $N_1$ be the subgroup of $\Sp(2n, p)$ generated by $\{ Z,U,V\}$,
omitting $V$ if $n=2$.  
\begin{theorem}\label{SpN1Cent}
The centraliser of $\sigma = \sigma_{12}(1)$ in $N_1$ is generated by 
the corresponding elements $(i)-(iv)$ listed in Theorem $\ref{SpCen}$, 
and $[Z^2,U]$ if $p$ is odd.
It is isomorphic to $C_2\times (C_2\wr S_{n-2})$ if $p=2$, and to 
$C_2\times C_2\times (C_4\wr S_{n-2})$ if $p$ is odd.

The centraliser of $\tau = \tau_1(1)$ in $N_1$ is generated by the 
corresponding elements $(i)-(iv)$ listed in Theorem $\ref{SpCen}$.
It is isomorphic to $C_2\wr S_{n-1}$ if $p=2$,
and to $C_2\times(C_4\wr S_{n-1})$ if $p$ is odd.
\end{theorem}

\begin{proof}
The listed elements clearly generate groups having the claimed isomorphism 
types. They generate
subgroups of $N_1$ of index $2n$ in the case of the centraliser of $\tau$; 
and of index $4n(n-1)$ if $p$ is odd, and index $2n(n-1)$ if $p=2$, 
in the case of the centraliser of $\sigma$.
This corresponds to the fact that $N_1$ permutes transitively the 
$2n$ elements $\tau_{\pm i}(1)$
with $1\le i\le n$, and the 
$4n(n-1)$ elements $\sigma_{\pm i,\pm j}(\pm 1)$ with
$1\le i\ne j\le n$ if $p$ is odd, 
and the $2n(n-1)$ elements $\sigma_{\pm i,\pm j}(1)$ with
$1\le i\ne j\le n$ if $p=2$. 
(Recall that, in either characteristic, these elements are equal in pairs.)
\end{proof}

\begin{theorem}\label{SpPrime}
Let $p$ be a prime and let $n \geq 2$. Let $G$ be the group 
generated by $\{\sigma, \tau, Z, U,V\}$, omitting $V$ if $n=2$,
subject to the relations given below.  Then $G$ is isomorphic to $\Sp(2n,p)$.  
\begin{enumerate}
\item[(i)] Defining relations for $N_1=\langle Z,U,V\rangle$ as in Theorem $\ref{SpN1}$.
\item[(ii)] Relations that state that the elements listed in Theorem $\ref{SpN1Cent}$ 
centralise $\sigma$ and $\tau$,
but omitting the relation $[\tau,VU]=1$ if $p$ is odd.  
\item[(iii)] Relations that present $\SL(2,p)$ on $\{\sigma,U'\}$. 
\item[(iv)] Relations that present $\SL(2,p)$ on $\{\tau,Z\}$. 
\item[(v)] The Steinberg relations (v) of Theorem $\ref{Sp}$.
\end{enumerate}
\end{theorem}
\begin{proof}
The proof is similar to that of Theorem \ref{Sp}.
\end{proof}
\noindent 
Observe that 
$\delta = {(\tau^{\omega-\omega^2})}^{Z} \tau^{\omega^{-1}} {(\tau^{\omega - 1})}^{Z}\tau^{-1}$.

\subsection{A presentation for $\PSp(2n,q)$}
If $q$ is even, then $\Sp(2n,q)$ is simple.  If $q$ is odd, then the centre of 
$\Sp(2n,q)$ has order 2, and is generated by $(ZV)^{2n}$. 

\subsection{Standard generators for $\Sp(2n,q)$}
In Table 1 of \cite{even,odd} the standard generators for 
$\Sp(d, q)$ are labelled $s, t, \delta, u, v, x$. 
Observe that $s = Z; t = \tau; u = U; v = V; x = \sigma^Z$; and 
the standard generator $\delta$ is the 
inverse of the presentation generator $\delta$. 

\section{A presentation for $\SU(2n,q)$ for $n > 1$}

\subsection{Generators and notation}\label{notation-SU2n}
Let $q=p^e$ for a prime $p$.  
Let $\omega$ be a primitive element of $\GF(q^2)$, 
and let $\omega_0=\omega^{q+1}$, so $\omega_0$ is a primitive element of $\GF(q)$. 
Define $\psi = \omega^{(q+1)/2}$ if $q$ is odd, and $\psi = 1$ if $q$ is even.  

Let $n>1$.  
We take a hyperbolic basis $(e_1,f_1,\ldots,e_n,f_n)$ of the natural module.  
We define the following elements of $\SU(2n,q)$, where $1\le i, j\le n$ 
and $i \not=j$, 
and $s\in\GF(q)$, and $\alpha\in\GF(q^2)$.

$Z_i=(e_i\mapsto -\psi f_i, f_i\mapsto \psi^{-1}e_i)$;  

$\Delta_{ij}=(e_i\mapsto\omega^{-1}e_i, f_i\mapsto\omega^q f_i, e_j\mapsto\omega e_j, f_j\mapsto\omega^{-q} f_j)$;

$\delta_i=(e_i\mapsto\omega_0^{-1}e_i, f_i\mapsto\omega_0 f_i)$;

$\tau_i(s)=(e_i\mapsto e_i-s\psi f_i)$;  

$\tau_{-i}(s)=
(f_i\mapsto f_i+s\psi^{-1}e_i)= \tau_i(s)^{Z_i}$;

$\sigma_{ij}(\alpha)=(e_i\mapsto e_i+\alpha e_j, f_j\mapsto f_j-\alpha^q f_i)$;

$\sigma_{-i,j}(\alpha)=
(f_i\mapsto f_i-\alpha\psi^{-1}e_j, f_j\mapsto f_j-\alpha^q\psi^{-1}e_i)  
=\sigma_{ij}(\alpha)^{Z_i}
=\sigma_{-j,i}(\alpha^q)$;

$\sigma_{i,-j}(\alpha)
=(e_i\mapsto e_i-\alpha\psi f_j, e_j\mapsto e_j-\alpha^q\psi f_i)  
=\sigma_{ij}(\alpha)^{Z_j}
=\sigma_{j,-i}(\alpha^q)$; 

$\sigma_{-i,-j}(\alpha) =(f_i\mapsto f_i+\alpha f_j, e_j\mapsto e_j-\alpha^q e_i)
=\sigma_{ij}(\alpha)^{Z_iZ_j} =\sigma_{ji}(-\alpha^q)$;  

$U=(e_1,e_2)(f_1,f_2)$;

$V=(e_1,e_2,\ldots,e_n)(f_1,f_2,\ldots,f_n)$ if $n>2$.

\smallskip
All these elements of $\SU(2n,q)$ centralise those basis elements that they 
are not stated to move.
As in the case of $\Sp(2n,q)$, 
$\sigma_{ij}(\alpha)^{Z_i^2}=\sigma_{ij}(-\alpha)=\sigma_{ij}(\alpha)^{Z_j^2}$.
Note that the change of basis that fixes $e_i$ and 
sends $f_i$ to $-\psi^{-1}f_i$ for all $i$ has the effect  
of replacing $\psi$ by 1 in these definitions.  
Thus the isomorphism type of the group generated by
these elements is independent of $\psi$.  
So the value of $\psi$ plays no role in the presentation
and proofs, apart from the fact that $\psi\ne0$.  
We choose $\psi$ to ensure that these elements preserve the given form.

\subsection{A presentation for $\SU(2n,q)$} 
We give a presentation for $\SU(2n,q)$  on the generating set
$\{ \sigma=\sigma_{12}(1), \tau=\tau_1(1), Z = Z_1, \delta = \delta_1, 
\Delta=\Delta_{12}, U=U_{12},V\}$, 
omitting $V$ if $n=2$.  Note that $\delta=\Delta\Delta^{-Z}$. 
If $q = 2$ then $\delta = 1$, permitting obvious simplifications
to the listed presentations.

Let $H$ be the subgroup of $\SU(2n, q)$ generated by $\{\Delta,U,V\}$, 
omitting $V$ if $n = 2$. 
Clearly $H$ has order $(q^2-1)^{n-1}n!$.

\begin{theorem}\label{SUDeltaUV}
Let $n \geq 2$ and let $G$ be the group generated by 
$\{\Delta, U, V\}$, omitting $V$ if $n = 2$, 
subject to the relations given below.
Then $G$ is isomorphic to $H$.
\begin{enumerate}
\item[(i)] Defining relations for $S_n$ on $\{U,V\}$. 
\item[(ii)] If $n>3$ then $[\Delta,U^{V^2}]=1$, and if $n>4$ then $[\Delta,VUU^V]=1$.
\item[(iii)] $\Delta^U=\Delta^{-1}$.
\item[(iv)] $\Delta\Delta^V=\Delta^{VU}$. 
\item[(v)] $[\Delta,\Delta^V]=1$, and if $n>3$ then $[\Delta,\Delta^{V^2}]=1$.
\item[(vi)] $\Delta^{q^2-1}=1$.
\end{enumerate}
We omit those relations that involve $V$ if $n=2$.  
\end{theorem}
\begin{proof}
The result is clear if $n=2$, so assume $n>2$.
Relations (ii) state that $\Delta$ is centralised by 
the subgroup of $\langle U,V\rangle$ that 
corresponds to the copy of $S_{n-2}$ that fixes $1$ and $2$, 
so the conjugates of $\Delta$ 
under $\langle U,V\rangle$ may be labelled as $\Delta_{ij}$, where $1\le i\ne j\le n$.  
Relation (iii) states that $\Delta_{ji}=\Delta_{ij}^{-1}$ when $(i,j)=(1,2)$, and 
hence for any suffices $(i,j)$.  Relation (v) states that $\Delta_{12}$ commutes with 
$\Delta_{23}$ and $\Delta_{34}$, when these are defined, 
and hence all $\Delta_{ij}$ commute with each other.
Finally relation (iv) states that 
$\Delta_{i,i+1}\Delta_{i+1,i+2}=\Delta_{i,i+2}$
when $i=1$, and hence for all $i$.  Hence, in $G$, 
every element of the group generated by the 
$\Delta_{ij}$ may be written in the form 
$\Delta_{12}^{m_1}\Delta_{23}^{m_2}\cdots\Delta_{n-1,n}^{m_{n-1}}$, 
and the theorem is proved.
\end{proof}

Let $K$ be the subgroup of $\SU(2n, q)$ generated 
by $\{\delta,\Delta,U,V\}$, omitting $V$ if $n = 2$. 
Clearly $K$ has order $(q-1)(q^2-1)^{n-1}n!$.

\begin{theorem}\label{SUDeltadeltaUV}
Let $G$ be the group generated by $\{\delta,\Delta,U,V\}$, omitting $V$ if $n = 2$,
subject to the relations listed below.
Then $G$ is isomorphic to $K$.
\begin{enumerate}
\item[(i)] Defining relations for $\langle \Delta, U, V \rangle$ as 
in Theorem $\ref{SUDeltaUV}$, omitting $V$ if $n = 2$. 
\item[(ii)] $\delta^{q-1}=1$.
\item[(iii)] $[\delta,U^V]=1$, and if $n>3$ then $[\delta, VU]=1$.
\item[(iv)] $[\delta,\Delta]=1$, and $[\delta,\Delta^V]=1$.
\item[(v)] $\Delta^{q+1}=\delta\delta^{-U}$.
\end{enumerate} 
We omit those relations that involve $V$ if $n=2$.  
\end{theorem}
\begin{proof}
Suppose  first that $n > 2$.
Relations (iii) state that $\delta$ is centralised by the 
subgroup of $\langle U,V\rangle$ that
corresponds to the subgroup of $S_{n-1}$ that fixes 1, 
so the conjugates of $\delta$ under 
$\langle U,V\rangle$ may be labelled as $\delta_i$, where $1\le i\le n$.  
Relation (v) states that $\delta_2=\delta_1\Delta^{-(q+1)}$; so, using (iv), $\delta_1$ 
commutes with $\delta_2$, and hence $\delta_i$ commutes with $\delta_j$ for all $i,j$.  
By (iv), $\delta$ commutes with $\Delta_{12}$ and with $\Delta_{23}$, 
and hence $\delta_i$ commutes with $\Delta_{jk}$ for all $i$ and for all $j\ne k$.  
It follows from (v) that $\delta_i\delta_j^{-1}=\Delta_{ij}^{q+1}$ for all $i\ne j$.
Thus if $A$ is the subgroup of $G$ generated by  
$\{\delta_i \, | \, i>0\}\cup\{\Delta_{ij} \, | \, i\ne j\}$
and $B=\langle\{\Delta_{ij}\}\rangle$, then $A$ is an abelian group,  
and $A/B$ is generated by $\{\delta_i \, | \, i>0\}$  
subject to the relations $\delta_i=\delta_j$, and $\delta_1^{q-1}=1$.  
So the group presented has order at most $(q-1)(q^2-1)^{n-1}n!$.  
But the corresponding matrix group has this order, so the theorem is proved when $n>2$.
The case $n=2$ is proved similarly.
\end{proof}

Let $N$ be the subgroup of $\SU(2n, q)$   
generated by $\{Z,\delta,\Delta,U,V \}$, omitting $V$ if $n=2$. 
It has order $(q-1)(q^2-1)^{n-1}2^n n!$. 

\begin{theorem}\label{SU2ncore}
Let $G$ be the group generated by $\{ Z, \delta,\Delta, U, V \}$, omitting $V$ if $n=2$, 
subject to the relations listed below.
Then $G$ is isomorphic to $N$.
\begin{enumerate}
\item[(i)] Defining relations for $\langle \delta,\Delta,U,V \rangle$
as in Theorem $\ref{SUDeltadeltaUV}$.
\item[(ii)]  $Z^2=\delta^{(q-1)/2}$ if $q$ is odd; $Z^2=1$ if $q$ is even. 
\item[(iii)] $[Z,U^V]=1$, and if $n>3$ then $[Z,VU]=1$.
\item[(iv)] $[Z,Z^U]=1$.
\item[(v)] $\delta=[\Delta^{-1},Z]$.
\item[(vi)] $[Z,\Delta^V]=1$.
\item[(vii)] $\delta^Z=\delta^{-1}$.
\item[(viii)] $[\delta, Z^U]=1$ if $n=2$.
\end{enumerate}
We omit those relations that involve $V$ if $n=2$.  
\end{theorem}

\begin{proof}
Suppose first that $n>2$. As with the previous theorem, 
relations (iii) imply that we have elements $Z_i$ 
in $G$ that are permuted by $\langle U,V\rangle$ via the natural action of $S_n$ 
on the suffices.  Relation (iv) asserts that $[Z_1,Z_2]=1$; 
so the group generated by the $Z_i$ is abelian.
Relation (vi) states that $\Delta_{23}$ is centralised by $Z_1$;
so $\Delta_{ij}$ is
centralised by $Z_k$ for all $i\ne k\ne j$.  
It follows from (v) that $\delta_i=[\Delta_{ij}^{-1},Z_i]$ for all 
$j\ne i$, and hence $\delta_i$ is centralised by $Z_k$ for all $k\ne i$  since $n>2$.
It now follows, from (vii), that the group $A$ in the proof of the previous theorem 
is normalised by $Z_1$, 
and hence by $Z_i$ for all $i$.  
By (ii), $A$ contains $Z_1^2$, and hence contains $Z_i^2$ for all $i$.
Thus $G$ has at most the stated order; since
the corresponding matrix group has this order the theorem is proved for $n>2$.

The proof for $n=2$ is similar, except for the proof that $\delta_i$ commutes with $Z_j$ 
for all $i\ne j$.
This follows from (viii), which states that $\delta_1$ commutes with $Z_2$ if $n=2$.
\end{proof}

\begin{theorem}\label{SU2nSL2} The subgroup of $\SU(2n,q)$ generated by 
$\{\tau, \delta, Z\}$ is isomorphic to $\SL(2,q)$, 
and the subgroup $H$ generated 
by $\{\sigma, \Delta, UZ^2\}$ is isomorphic to $\SL(2,q^2)$.
\end{theorem}
\begin{proof}
The first statement is obvious.
Observe that $H$ acts on $\langle e_1,e_2\rangle$ as $\SL(2,q^2)$, and 
similarly on $\langle f_1,f_2\rangle$, the actions 
 corresponding under the inverse transpose automorphism 
followed by the Frobenius automorphism.
\end{proof}

\begin{theorem}\label{SU2nsigma}
The centraliser $H$ of $\sigma = \sigma_{12}(1)$ in $N$ 
is generated by the following elements:
\begin{enumerate}
\item[(i)] $U^{V^2}$ if $n>3$, and $VUU^V$ if $n>4$;
\item[(ii)] $\Delta^{V^2}$ if $n>3$ and $q$ is odd, 
and $\delta^{V^2}$ if $n=3$ and $q$ is odd;
\item[(iii)] $Z^{V^2}$ if $n>2$;
\item[(iv)] $\Delta\Delta^{2V}$ if $n>2$;
\item[(v)] if $n=2$ then $\delta\delta^U$ if $q$ is even, 
and $\Delta^{(q+1)/2}\delta^{-1}$ if $q$ is odd;
\item[(vi)] $ZUZ^{-1}$.
\end{enumerate}
If $n > 2$ then $H$ has order $(q-1)(q^2-1)^{n-2}2^{n-1}(n-2)!$; 
if $n = 2$ then $H$ has order $4(q-1)$ if $q$ is odd, and order $2(q-1)$ if $q$ is even.
\end{theorem}
\begin{proof}
Let $n>2$. The orbit of $\sigma$ under conjugation by $N$ is 
$\{\sigma_{ij}(\alpha) \, | \, 1\le |i|\ne |j|\le n,\alpha\in\GF(q^2)^\times\}$. 
But the $4n(n-1)(q^2-1)$ matrices defined 
are equal in pairs, as shown in Section \ref{notation-SU2n}.
The orbit has size $2n(n-1)(q^2-1)$, so $H$ has order
$(q-1)(q^2-1)^{n-2}2^{n-1}(n-2)!$.
Let $K$ be the group generated by the itemised elements.  
Clearly $K$ is a subgroup of $H$.
The subgroup $K_0$ of $K$ generated by the elements in (i) is the copy 
of $S_{n-2}$ that permutes the basis vectors
$\{e_1,e_2,\ldots,e_n\}$, fixing $e_1$ and $e_2$, 
and also permutes $\{f_1,\ldots,f_n\}$, fixing $f_1$ and $f_2$.

Let $n>2$ and $q$ be even. 
Observe that $$\Delta\Delta^{2V} = 
\diag(\omega^{-1},\omega^q,\omega^{-1},\omega^q,\omega^2,\omega^{-2q},1,1,\ldots,1).$$ 
Its $K_0$-conjugates generate
an abelian group $A$ of order $(q^2-1)^{n-2}$.
Define $$b:=[\Delta\Delta^{2V},Z^{V^2}]=\diag(1,1,1,1,\omega_0^{-2},\omega_0^2,1,1,\ldots,1).$$  
Clearly $B := \langle A, b\rangle$ is a diagonal subgroup of $K$ of order $(q-1)(q^2-1)^{n-2}$ that is normalised by $K_0$.  
Also, $ZUZ^{-1}$ and the $K_0$-conjugates of $Z^{V^2}$
generate an
abelian group $C$ of order $2^{n-1}$ that intersects $B$ in the identity.  
Since $\langle K_0, B, C \rangle$ has the correct order, it  is the 
centraliser of $\sigma$.

Let $n>3$ and $q$ be odd.  Observe that 
$$\Delta^{V^2}=\diag(1,1,1,1,\omega^{-1},\omega^q,\omega,\omega^{-q},1,1,\ldots,1).$$ 
Its $\langle K_0,Z^{V^2}\rangle$-conjugates generates the group
$$\{\diag(1,1,1,1,\alpha_3,\alpha_3^{-q},\ldots,\alpha_n,\alpha_n^{-q}) \, : \, 
\prod_i\alpha_i\in\GF(q)^\times\},$$ 
which has order $(q^2-1)^{n-3}(q-1)$.  This group and 
$\Delta\Delta^{2V}$
generate a group $D$ of order $(q^2-1)^{n-2}(q-1)$, 
which is normalised by the group $C=\langle K_0,Z^{V^2},ZUZ^{-1}\rangle.$
Now $C$ has order $2^{2n-3}(n-2)!$, and intersects $D$ in a group of 
order $2^{n-2}$, 
so $CD$ has order $(q-1)(q^2-1)^{n-2}2^{n-1}(n-2)!$ and 
so is the centraliser of $\sigma$ in $N$.

Let $n=3$ and $q$ be odd.  
Now $D = \langle \Delta\Delta^{2V}, \delta^{V^2}\rangle$ has order $(q-1)(q^2-1)$, 
and is normalised by
$C=\langle Z^{V^2}, ZUZ^{-1}\rangle$, which has order 8, and 
$C\cap D$ has order 2, so $CD$ has order $4(q-1)(q^2-1)$, as required.

Finally let $n = 2$.
The orbit of $\sigma$ under conjugation by $N$ is 
$\{\sigma_{\pm1,\pm2}(\alpha^2) \, | \, \alpha\in\GF(q^2)^\times\}$,
and hence has size $4(q^2-1)$ in $N$ if $q$ is even, 
and $2(q^2-1)$ if $q$ is odd.
So the centraliser of $\sigma$ has order $2(q-1)$ if $q$ is even, and
order $4(q-1)$ if $q$ is odd.
Items (v) and (vi) generate a copy of 
$D_{2(q-1)}$ if $q$ is even, else a copy of $SD_{4(q-1)}$; 
and these groups have the required order.
\end{proof}

\begin{theorem}\label{SU2ntau}
The centraliser of $\tau = \tau_1(1)$ in $N$ 
has order $(q^2-1)^{n-1}2^{n - 1}(n-1)!$ and is generated by the following elements:
\begin{enumerate}
\item[(i)] $U^V$ if $n>2$, and $VU$ if $n>3$;
\item[(ii)] if $n=2$ and $q$ is odd then $\delta^U$, and if $n>2$ then $\Delta^V$;
\item[(iii)] $Z^U$;
\item[(iv)] $\Delta^2\delta^{-1}$.
\end{enumerate}
\end{theorem}
\begin{proof}
Recall that the order of $N$ is $(q-1)(q^2-1)^{n-1}2^nn!$.
The orbit of $\tau$ under conjugation by $N$ 
is $\{\tau_{\pm i}(s) \, | \, 1\le i\le n, s\in\GF(q)^{\times}\}$, 
as $\tau_i(s)^{\Delta}=\tau_i(s\omega_0)$, and has size $2(q-1)n$.

The elements in (i) imply that $\tau$ is centralised by 
the copy of $S_{n-1}$ in $\langle U,V\rangle$ 
that fixes the basis vectors $e_1$ and $f_1$.  

Suppose that $n>2$ and $q$ is odd.
Now $\Delta^V=\diag(1,1,\omega^{-1},\omega^q,\omega,\omega^{-q},1,\ldots,1)$, 
together with its 
conjugates under this copy of $S_{n-1}$, 
generates a group $A$ that is the direct product of 
$n-2$ cyclic groups of order $q^2-1$.
This group intersects the cyclic group of order $(q^2-1)/2$ generated by 
$\Delta^2\delta^{-1}=\diag(\omega^{q-1},\omega^{q-1},\omega^2,\omega^{-2q},1,\ldots,1)$ 
in the identity,
so this element and $A$ generate a group $B$ of order $(q^2-1)^{n-1}/2$.
Define $g := (\Delta^V)^{Z^U} = \Delta^V\times\diag(1,1,\omega_0,\omega_0^{-1},1,\ldots,1)$.
Now $g^2= \Delta^{2V} (\Delta^2\delta^{-1})^{q+1}$ lies in $B$, 
but $g$ does not, and $C := \langle B,g\rangle$
has order $(q^2-1)^{n-1}$.  
Thus the group generated by the listed elements 
contains
an abelian subgroup $C$ of order $(q^2-1)^{n-1}$, 
extended by an elementary abelian group of
order $2^{n-1}$ generated, modulo $C$, by $\{Z_i \, | \, i>1\}$, 
extended by a copy of $S_{n-1}$.
This group has index $2(q-1)n$ in $N$, and so is the centraliser of $\tau$.

If $n>2$ and $q$ is even, then the proof is simpler: 
$\Delta^2\delta^{-1}$ and the $S_{n-1}$-conjugates 
of $\Delta^V$ now generate a group of order $(q^2-1)^{n-1}$.

The proof is similar for $n = 2$.  We need to prove that the given elements 
generate a group that contains a
diagonal subgroup of order $q^2-1$.  
If $q$ is even, then $\Delta^2\delta^{-1}$ generates such a subgroup.
If $q$ is odd, then this element has order $(q^2-1)/2$, 
and $\delta^U=\diag(1,1,\omega_0^{-1},\omega_0)$
does not lie in $\langle \Delta^2\delta^{-1} \rangle$, 
but its square, which is $(\Delta^2\delta^{-1})^{-(q+1)}$, does.
\end{proof}

\begin{theorem}\label{SU2npres}
Let $q=p^e$ for a prime $p$ and let $n\ge 2$.  
Let $G$ be the group generated by $\{ \sigma, \tau, Z, \delta, \Delta, U, V \}$, omitting 
$V$ if $n=2$, subject to the relations given below. 
Then $G$ is isomorphic to $\SU(2n,q)$.

\begin{enumerate}
\item[(i)] Defining relations for $N=\langle Z, \delta,\Delta, U,V \rangle$
as in Theorem $\ref{SU2ncore}$, omitting $V$ if $n=2$,
and omitting relations (iii) and (vi) of Theorem $\ref{SUDeltaUV}$, 
relation (ii) of Theorem $\ref{SUDeltadeltaUV}$,
and relations (ii) and (vii) of Theorem $\ref{SU2ncore}$.

\item[(ii)] Relations that state that the elements listed in 
Theorem $\ref{SU2nsigma}$ centralise $\sigma$.

\item[(iii)] Relations that state that the elements listed in 
Theorem $\ref{SU2ntau}$ centralise $\tau$, but omitting the relation $[\tau,VU]=1$.

\item[(iv)] Relations that present 
$\SL(2,q)$ on $\{\tau, \delta, Z\}$,
and $\SL(2,q^2)$ on $\{\sigma, \Delta, UZ^2\}$
as in Theorem $\ref{SU2nSL2}$. 

\item[(v)] The following instances of Steinberg relations:
\begin{enumerate}
\item[(a)] $[\sigma,\tau]=1$;
\item[(b)]  $[\sigma,\sigma^Z]=\tau^{2ZU}$, and
$[\sigma^{\Delta},\sigma^Z]=\tau^{ZU\Delta^m}$, where $\omega_0^m=\omega^2+\omega^{2q}$;
the RHS of the first relation may be replaced by $1$ if $q$ is even, 
and the RHS of the second relation must be replaced by $1$ if $q=3$;
\item[(c)] $[\sigma,\tau^Z]=\sigma^Z\tau^{-ZU}$;
\item[(d)] if $n>2$ then $[\sigma,\sigma^{U^V}]=1$, and if $n=3$ and $q=2$ then 
$[\sigma,\sigma^{U^V\Delta}]=1$;
\item[(e)] if $n>2$ then $[\sigma,\sigma^V]=\sigma^{VU}$, and if $n=3$ and $q=2$ then
$[\sigma,\sigma^{V\Delta}]=\sigma^{VU\Delta^{-1}}$;
\item[(f)] if $n<4$ then $[\tau,\tau^U]=1$;
\item[(g)] if $n=3$ then $[\tau,\sigma^V]=1$;
\item[(h)] if $n>3$ then $[\sigma,\sigma^{V^2}]=1$.
\end{enumerate}
\item[(vi)] If $n=2$ and $q=3$ then $\tau = [\sigma^{-UZ},\sigma]^\Delta$.
\end{enumerate}
\end{theorem}

\begin{proof}
The relations omitted from (i) are supplied by (iv).  
In odd characteristic the relation $\Delta^U=\Delta^{-1}$, which is 
(iii) of Theorem \ref{SUDeltaUV}, follows from the 
relation $\Delta^{UZ^2}=\Delta^{-1}$ implied by 
(iv) of Theorem \ref{SU2nSL2}, since $Z^2=\delta^{(q-1)/2}$
by the same relations, and $[\delta,\Delta]=1$ by (iv) of 
Theorem \ref{SUDeltadeltaUV}.

The relation $[\tau,VU]=1$ omitted from (iii) may be supplied using 
relation (v)(b) exactly as with the presentation for $\Sp(2n,q)$;
here the argument applies in all characteristics.

As to (v)(b), note that $\omega^2+\omega^{2q}=0$ if $q=3$.

Relations (i), (ii) and (iii), together with (iv) if $n=2$, 
allow us to define elements $\sigma_{ij}(\alpha)$ and $\tau_i(s)$ of
$G$ that are permuted by $N$.  

The following Steinberg relations must now be proved to hold in $G$.  
Here $i,j,k,l$ are distinct positive integers in the range $[1,n]$, 
and $\alpha,\beta\in\GF(q^2)$, and $s,t\in\GF(q)$.
\begin{enumerate}
\item[(1)] $\sigma_{ij}(\alpha)=\sigma_{-j,-i}(-\alpha^q)$; 
$\sigma_{i,-j}(\alpha)=\sigma_{j,-i}(\alpha^q)$;
$\sigma_{-i,j}(\alpha)=\sigma_{-j,i}(\alpha^q)$; 
\item[(2)] $\tau_i(s)\tau_i(t)=\tau_i(s+t)$, $\tau_i(s)^p=1$;
\item[(3)] $\sigma_{ij}(\alpha)\sigma_{ij}(\beta)=\sigma_{ij}(\alpha+\beta)$, $\sigma_{ij}(\alpha)^p=1$;
\item[(4)] $[\sigma_{ij}(\alpha),\tau_i(s)]=[\sigma_{ij}(\alpha),\tau_{-j}(s)]=1$;
\item[(5)] $[\sigma_{ij}(\alpha),\sigma_{-i,j}(\beta)]=\tau_{-j}(\alpha^q\beta+\alpha\beta^q)$;
\item[(6)] $[\sigma_{ij}(\alpha),\sigma_{i,-j}(\beta)]=\tau_i(\alpha^q\beta+\alpha\beta^q)$;
\item[(7)] $[\sigma_{ij}(\alpha),\tau_{-i}(s)]=\sigma_{-i,j}(\alpha s)\tau_{-j}(-\alpha^{q+1}s)$, \\
$[\sigma_{ij}(\alpha),\tau_j(s)]=\sigma_{j,-i}(\alpha^qs)\tau_i(-\alpha^{q+1}s)$;
\item[(8)] $[\sigma_{ij}(\alpha),\sigma_{ik}(\beta)]=1$ if $n>2$;
\item[(9)] $[\sigma_{ij}(\alpha),\sigma_{kj}(\beta)]=1$ if $n>2$;
\item[(10)] $[\sigma_{ij}(\alpha),\sigma_{jk}(\beta)]=\sigma_{ik}(\alpha\beta)$ if $n>2$;
\item[(11)] 
$[\sigma_{ij}(\alpha),\sigma_{-i,k}(\beta)]=\sigma_{-j,k}(-\alpha^q\beta)$, \\
$[\sigma_{ij}(\alpha),\sigma_{k,-j}(\beta)]=\sigma_{i,-k}(\alpha\beta^q)$,\\
$[\sigma_{ij}(\alpha),\sigma_{k,-i}(\beta)]=[\sigma_{ij}(\alpha),\sigma_{-j,k}(\beta)]=1$, 
where in all cases $n>2$;
\item[(12)] $[\tau_i(s),\tau_j(t)]=[\sigma_{ij}(\alpha),\tau_k(s)]=[\sigma_{ij}(\alpha),\sigma_{kl}(\beta)]=1$ 
for all $n$, if $n>2$, and if $n>3$ respectively.
\end{enumerate}

Similar relations hold when we lift the restriction that $i,j,k$ and $l$ are positive. 
These relations are obtained from those listed by conjugating 
with $Z_i$, $Z_j$, $Z_k$, and $Z_l$, 
as appropriate, and may differ from them in the signs of the coefficient: 
$\sigma_{-i,j}(\alpha)^{Z_i} = \sigma_{i,j}(-\alpha)$ 
for $i>0$ and $j$ positive or negative, and
similarly $\sigma_{i,-j}(\alpha)^{Z_j}=\sigma_{i,j}(-\alpha)$ for $j>0$ and 
$i$ positive or negative.
It suffices to prove that these relations hold in $G$ when the 
suffices $i$ to $l$ are positive.
We assume that $(i,j,k,l) = (1,2,3,4)$; the general case 
follows by conjugation in $\langle U,V\rangle$.

Relation (vi) is also required, since the Steinberg relations define the 
universal covering group of 
the simple group, and this is not the classical group in the case 
of $\PSU(2n,q)$ when $(n,q)$ is
$(2,2)$, or $(2,3)$, or $(3,2)$.  Our presentations include relations that 
are not Steinberg 
relations, in 
particular those for $N$.  We find, by coset enumeration, that our 
presentations in these
three cases are correct; in particular, no further relation is needed if 
$(n,q)$ is $(2,2)$ or $(3,2)$.
Note that $\SU(4,2)$ is simple, and its universal covering group is 
isomorphic to $\Sp(4,3)$, which has centre of order 2.

Relations (1) hold by definition, using (i) and (iii), but 
$Z_i^2$ commutes with $\tau_{\pm i}(s)$
and conjugates $\sigma_{\pm i,\pm j}(\alpha)$ to its inverse.

Relations (2) and (3) follow from Theorem \ref{SU2nSL2} and relations (i) to (iv).

We now prove the first relation in (4).  
The instance $i=1$, $j=2$, $\alpha=s=1$ is given by relation (a).  
Let $P$ be the set of pairs $(\alpha,s)$ for which (4) holds with $i=1$ and $j=2$.  
Clearly if $(\alpha,s)\in P$ and
$(\beta,s)\in P$, then $(\alpha+\beta,s)\in P$; 
and if $(\alpha,s)\in P$ and $(\alpha,t)\in P$, then
$(\alpha,s+t)\in P$.
Conjugating by $\Delta$ shows that if $(\alpha,s)\in P$ then 
$(\omega^2\alpha,\omega_0s)\in P$, and so $(\omega^{2(q-1)},1)\in P$.  
If $q \ne 3$ then $\omega^{2(q-1)}$ generates $\GF(q^2)$ as a field, 
so $(\alpha,1) \in P$ for all $\alpha\in\GF(q^2)$.
Since $\omega_0$ generates $\GF(q)$ as a field, it follows 
that $(\alpha, s)\in P$ for all $\alpha$ in $\GF(q^2)$ and
all $s\in\GF(q)$.  For $q=3$ (or more generally prime), 
if $(\alpha,s)\in P$ for one $s$ in $\GF(q)^\times$ then
$(\alpha,s)\in P$ for all $s$ in $\GF(q)$, 
and since $\omega^2$ generates $\GF(q^2)$ as a field, again $(\alpha,s)\in P$
for all $\alpha$ in $\GF(q^2)$, and all $s\in\GF(q)$.
The second relation in (4) follows from the first by applying (1), 
and conjugating by $Z_iZ_j$.

We now prove relation (5).  Let $P$ be the set of pairs $(\alpha,\beta)$ such that 
(5) holds 
with $i=1$ and $j=2$.  Relation (b) states that $(1,1)\in P$ and $(\omega^2,1)\in P$.  
Relations (2) and (4) show that if $(\alpha,\beta_1)$ and $(\alpha,\beta_2)$ lie in $P$, 
then so does $(\alpha,\beta_1+\beta_2)$.  Similarly, 
if $(\alpha_1,\beta)$ and $(\alpha_2,\beta)$
lie in $P$, then so does $(\alpha_1+\alpha_2,\beta)$.
Conjugating by $\Delta$ shows that if  $(\alpha,\beta)\in P$ then 
$(\omega^2\alpha,\omega^{1-q}\beta)\in P$.
If $q$ is odd then 
$(\omega^{q+1}\alpha,-\beta)=(\omega_0\alpha,-\beta)\in P$, so
$(\omega_0\alpha,\beta)\in P$.  
Similarly, if $q$ is even, then $(\omega_0^2\alpha,\beta)\in P$,
and hence $(\omega_0\alpha,\beta)\in P$.
Conjugating by $\Delta_{21}$ shows that if
$(\alpha,\beta)\in P$ then $(\omega^{-2}\alpha,\omega^{q-1}\beta)\in P$ and 
so $(\alpha,\omega_0\beta)\in P$.
It follows that $P=\GF(q^2)\times\GF(q^2)$ since $\omega^2\not\in\GF(q)$.  

We now prove relation (6).  By (5), 
$[\sigma_{ij}(\alpha),\sigma_{-i,j}(\beta)]=\tau_{-j}(\alpha^q\beta+\alpha\beta^q)$.
Conjugating by $Z_j$ now gives 
$[\sigma_{i,-j}(\alpha),\sigma_{-i,-j}(\beta)]=\tau_j(\alpha^q\beta+\alpha\beta^q)$.
Equivalently 
$[\sigma_{j,-i}(\alpha),\sigma_{-j,-i}(\beta)]=\tau_i(\alpha^q\beta+\alpha\beta^q)$.
By (1) this can be written as
$[\sigma_{i,-j}(\alpha^q),\sigma_{i,j}(-\beta^q)] = \tau_i(\alpha^q\beta+\alpha\beta^q),$
which is equivalent to (6).

We now prove the first relation in (7).  
An instance with $i=1$, $j=2$, and $\alpha=s=1$ is given by
relation (c).  
Let $P$ be the set of pairs $\{ (\alpha,s) \, | \, \alpha\in\GF(q^2),s\in\GF(q) \}$
for which (7) holds with $(i,j)=(1,2)$.
Suppose that $(\alpha,s)$ and $(\beta,s)$ lie in $P$.  Now 
\begin{eqnarray*}
[\sigma_{12}(\alpha+\beta),\tau_{-1}(s)] & = & 
[\sigma_{12}(\alpha),\tau_{-1}(s)]^{\sigma_{12}(\beta)}[\sigma_{12}(\beta),\tau_{-1}(s)] \\
& = & (\sigma_{-1,2}(\alpha s)\tau_{-2}(-\alpha^{q+1}s))^{\sigma_{12}(\beta)}\sigma_{-1,2}(\beta s)\tau_{-2}(-\beta^{q+1}s) \\
& = & 
\sigma_{-1,2}(\alpha s)\tau_{-2}(-\alpha^q\beta s-\alpha\beta^qs)\tau_{-2}(-\alpha^{q+1}s)\sigma_{-1,2}(\beta s)\tau_{-2}(-\beta^{q+1}s) \\ 
& = & 
\sigma_{-1,2}((\alpha+\beta)s)\tau_{-2}(-(\alpha+\beta)^{q+1}s).
\end{eqnarray*}
Here we used relations (2)--(5).
So if $(\alpha,s)$ and $(\beta,s)$ lie in $P$ then so does $(\alpha+\beta,s)$.
Similarly, using (2), (3), (4), and (12), one proves that,
if $(\alpha,s)$ and $(\alpha,t)$ lie in $P$ then so does $(\alpha,s+t)$.
Conjugating by $\Delta$ shows that if $(\alpha,s)\in P$ then 
$(\alpha\omega^2,s\omega_0)\in P$, so $(\alpha\omega^{2(q-1)},s)\in P$.
As in the proof of (4), 
it follows that if $(\alpha,s)$ lies in $P$ for one value of $\alpha$ then 
this holds for all $\alpha$; and hence $(\alpha,s)\in P$ for all $\alpha$ and all $s$.
The second relation in (7) follows from the first by applying (1), 
and conjugating by $Z_iZ_j$.

We now prove relation (8).  
Clearly if this holds for $(\alpha_1,\beta)$ and $(\alpha_2,\beta)$ then
it also holds for $(\alpha_1+\alpha_2,\beta)$, and similarly with the roles 
of $\alpha$ and $\beta$ reversed.  
Let $P$ be the set of all
pairs $(\alpha,\beta)$ for which (8) holds. 
If $(\alpha,\beta)\in P$, then 
conjugating by $\Delta_{ij}$ and $\Delta_{ik}$ shows that
both $(\omega^2\alpha,\omega\beta)\in P$
and $(\omega\alpha,\omega^2\beta)\in P$, so $(\alpha,\omega^{-3}\beta)\in P$.  
If $q>2$ then $\omega^3$
generates $\GF(q^2)$ as a field, so (8) follows from one instance, 
such as (d).
If $q=2$ and $n>3$ then the general case also follows from one instance, 
by conjugation with $\Delta_{jl}$ and $\Delta_{kl}$.  If $(n, q) = (3, 2)$ 
then we need the additional relation in 
(d), which asserts that $(1,\omega)\in P$.

Now (9) follows from (8) since $\sigma_{ij}(\alpha)=\sigma_{-j,-i}(-\alpha^q)$.

Using (8) and (9) it follows that if (10) holds 
for $(\alpha_1,\beta)$ and $(\alpha_2,\beta)$ then
it holds for $(\alpha_1+\alpha_2,\beta)$, and similarly 
with $\alpha$ and $\beta$ reversed.
The result now follows from (e) exactly as in the proof of (8), 
and again we need an extra relation if $(n,q)=(3,2)$,
which here gives $[\sigma_{12}(2), \sigma_{23}(\omega)]=\sigma_{13}(\omega)$.

The relations in (11) follow from 
(8)--(10) by conjugating with 
$Z_i$, $Z_j$ and $Z_k$, as appropriate, and using (1).

If $n>3$ then relations (f) and (g) may be supplied using (h) and (b).
Relations (12) now follow. 

\medskip 
It follows that the conjugates of $\sigma$ and $\tau$ under 
$N=\langle Z,\Delta,U,V\rangle$, omitting $V$ if $n=2$, 
generate a subgroup $H$ of $G$ that is isomorphic to $\SU(2n,q)$.  
But $Z$ and $\Delta$ and $U$ lie in $H$, as do their conjugates under $N$, and 
hence if $n>2$ then $V$ lies in $H$, as a product of $N$-conjugates of $U$.  
Thus $G=H$, as required.
\end{proof}

We shall need the following result.  
\begin{corollary}\label{SU(4,q)/-I4} 
If $q$ is odd, then a presentation for $\SU(4,q)/\langle -I_4\rangle$  
is obtained from the presentation for $\SU(4,q)$ of Theorem $\ref{SU2npres}$ by 
replacing the presentation for $\SL(2,q^2)$ on $\{\sigma,\Delta,UZ^2\}$ by a 
presentation of $\PSL(2,q^2)$ on these generators as in Theorem $\ref{CRW-PSL-odd}$.
\end{corollary}

\subsection{A presentation for $\PSU(2n,q)$} 
The centre of $\SU(2n,q)$ has order $\gcd(q+1, 2n)$.

\begin{lemma}
Let $v_2(x)$ denote the $2$-adic value of the integer $x$.
If $v_2(n)\ge v_2(q+1)$ then the centre of $\SU(2n,q)$ is generated by
$(\Delta^{q-1}UV^{-1})^{(n-1)(q+1)/\gcd(q+1,n)}$.  If $v_2(q+1)>v_2(n)$
then the centre is generated by $Z^2(\Delta^{q-1}UV^{-1})^{(n-1)(q+1)/(2\gcd(q+1,n))}$.
\end{lemma}

\begin{proof}
Let $\theta := \omega^{1-q}$, so 
$\Delta^{q-1}= \diag(\theta,\theta,\theta^{-1},\theta^{-1},1,1,\ldots,1)$. 
If $g := (\Delta^{q-1}UV^{-1})^{n-1}$ then 
$g = \diag(\theta^{(n-1)},\theta^{(n-1)},
\theta^{-1},\theta^{-1},\ldots,\theta^{-1})$.
So $g^k$ is a scalar matrix if and only if $\theta^{nk}=1$; 
equivalently, 
$k$ is a multiple of $(q+1)/\gcd(q+1,n)$.  
So $g^{(q+1)/\gcd(q+1,n)}$ is a scalar element of 
$\SU(2n,q)$ of order $\gcd(q+1,n)$.  
If $v_2(n)\ge v_2(q+1)$ then $\gcd(q+1,n)=\gcd(q+1,2n)$,
as required.  Now suppose that $v_2(q+1)>v_2(n)$.
Observe that $Z^2g^{(q+1)/(2\gcd(q+1,n))}$ 
has eigenvalues $-\theta^{(n-1)(q+1)/(2\gcd(q+1,n)}$ 
and $\theta^{-(q+1)/(2\gcd(q+1,n))}$,
and these are equal, so we have a scalar element of order $2\gcd(q+1,n)=\gcd(q+1,2n)$,
as required.
\end{proof}

\subsection{Standard generators for $\SU(2n, q)$} 
In Table 1 of \cite{even,odd} the standard generators for 
$\SU(2n, q)$ are labelled $s, t, \delta, u, v, x, y$. 
Observe that $s = Z^{-1}; t = \tau^{-1}; 
u = U; v = V; x = \sigma; y = \Delta^{-1}$; and 
the standard generator $\delta$ is the 
inverse of the presentation generator $\delta$. 

\section{A presentation for $\SU(2n+1,q)$ for $n>1$}

\subsection{Generators and notation}
Let $q=p^e$ for a prime $p$.
Let $\omega$ be a primitive element of $\GF(q^2)$,  
and let $\omega_0=\omega^{q+1}$, so $\omega_0$ is a primitive element of $\GF(q)$. 
Define $\psi = \omega^{(q+1)/2}$ if $q$ is odd, and $\psi = 1$ if $q$ is even.   
Let $\phi$ be a fixed element of $\GF(q^2)$ of trace $-1$:  
define $\phi=-1/2$ if $q$ is odd, 
and $\phi=\omega/(\omega+\omega^q)$ if $q$ is even.  

Let $n>1$.
We take a basis $(e_1,f_1,e_2,f_2,\ldots,e_n,f_n,w)$ of the natural module,
with symmetric bilinear form defined by 
$e_i.f_i=f_i.e_i=1$ for all $i$, and 
$w.w=1$, and the form vanishes on all other pairs of basis vectors.  
So $\SU(2n+1,q)$ contains $\SU(2n,q)$ as defined in Section \ref{notation-SU2n}.
We retain the notation of that section and define the following new elements
of $\SU(2n+1, q)$, where  $1\le i\le n$ and $\alpha\in\GF(q^2)$. 

\smallskip
 $t_i = (e_i\mapsto f_i, f_i\mapsto e_i, w\mapsto -w)$;

\smallskip
$\Gamma_i = 
(e_i\mapsto\omega^{-1} e_i, f_i\mapsto \omega^{q}f_i, w\mapsto\omega^{1-q}w)$;

\smallskip
$\myupsilon_i(\alpha)  = 
(e_i\mapsto e_i+\phi\alpha^{q+1}f_i+\alpha w, w\mapsto w-\alpha^q f_i)$;

\smallskip
$\myupsilon_i(\alpha)^{-1} = 
(e_i\mapsto e_i-(\phi+1)\alpha^{q+1}f_i-\alpha w,
w\mapsto w+\alpha^q f_i)$;

\smallskip
$\myupsilon_{-i}(\alpha)=
(f_i\mapsto f_i+\phi\alpha^{q+1}e_i -\alpha w,  w\mapsto w +\alpha^q e_i)
= \myupsilon_i(\alpha)^{t_i}$. 

\smallskip
Recall that $\sigma_{ij}(\alpha) = 
(e_i\mapsto e_i+\alpha e_j, f_j\mapsto f_j-\alpha^q f_i)$ 
and $\tau_i(s) = (e_i\mapsto e_i-s\psi f_i)$,  
where $1 \leq i,j \leq n$ and $i \not=j$, and $s\in \GF(q)$. 

All these elements of $\SU(2n+1,q)$ centralise those basis elements 
that they are not stated to move.

\subsection{A presentation for $\SU(2n + 1, q)$}
We give a presentation for $\SU(2n+1,q)$ on the generating set 
$\{ \myupsilon = \myupsilon_1(1), \sigma = \sigma_{12}(1), \tau = \tau_1(1), 
\Gamma = \Gamma_1, t = t_1, U, V\}$, 
omitting $V$ if $n=2$.

The remaining generators used in the presentation for $\SU(2n,q)$ are defined as follows:  
$\Delta=\Gamma\Gamma^{-U}$;
if $q$ is even, then $Z=t$, else $Z=t\Gamma^{(q+1)/2}$;
and $\delta=\Delta\Delta^{-Z}=\Gamma^{q+1}$.

%
%
%
%
%
%
%
%
%



Let $N$ be the subgroup of $\SU(2n+1,q)$ generated by $\{\Gamma,t,U,V\}$, 
omitting $V$ if $n=2$.
It is isomorphic to $(C_{q^2-1}:C_2)\wr S_n$, where the $2$-cycle generated by $t$ acting 
on the $(q^2-1)$-cycle generated by $\Gamma$ acts as $g\mapsto g^{-q}$. 

\begin{theorem}\label{SU2n+1top}
Let $G$ be the group generated by 
$\{\Gamma,t,U,V\}$, omitting $V$ if $n=2$,
 subject to the relations given below. 
Then $G$ is isomorphic to $N$.
\begin{enumerate}
\item[(i)] Defining relations for $S_n$ on $\{U,V\}$, omitting $V$ if $n=2$.
\item[(ii)] $\Gamma^{q^2-1}=t^2=1$, $\Gamma^t=\Gamma^{-q}$.
\item[(iii)] If $n>2$ then $[\Gamma,U^V]=[t,U^V]=1$, 
and if $n>3$ then $[\Gamma,VU]=[t,VU]=1$.
\item[(iv)] $[\Gamma,\Gamma^U]=[t,t^U]=[\Gamma,t^U]=1$.
\end{enumerate}
\end{theorem}
\begin{proof}
This is a standard wreath product presentation.
\end{proof}

\begin{theorem}\label{SU2n+1ups}
The centraliser $H$ of $\myupsilon=\myupsilon_1(1)$ in $N$
has index $2(q^2-1)n$ in $N$ and 
is generated by the following elements:
\begin{enumerate}
\item[(i)] $U^V$ if $n>2$;
\item[(ii)] $VU$ if $n>3$;
\item[(iii)] if $q$ is even, then $t^U$, else $t^U\Delta^{(q^2-1)/2}$;
\item[(iv)] if $n=2$ and $q\equiv 1\bmod 3$ then 
$\Gamma^{q-1}\Gamma^{3U}[t,\Gamma^{-1}]^U$, else $\Gamma^{q-1}\Gamma^{3U}$; 
\item[(v)] $\Gamma^U\Gamma^{-V^2}$ if $n>2$.
\end{enumerate}
\end{theorem}
\begin{proof}
Since the matrix $\diag(a_1,a_2,\ldots,a_{2n+1})$ centralises $\myupsilon$ if 
and only if $a_1=a_2=a_{2n+1}$, 
these elements centralise $\myupsilon$.

The orbit of $\myupsilon$ under conjugation by $N$ is 
$\{\myupsilon_{\pm i}(\alpha) \, | \, 1\le i\le n, 
\alpha\in\GF(q^2)^{\times}\}$, and has size $2(q^2-1)n$.  

Consider the case $n>2$. 
Item (v) is the matrix 
$g = \diag(1,1,\omega^{-1},\omega^q,\omega,\omega^{-q},1,\ldots,1)$; 
clearly the conjugates of $g$ under the copy of $S_{n-1}$ generated by 
items (i) and (ii) generate a group $A$ of order $(q^2-1)^{n-2}$, 
namely 
$$\{\diag(1,1,a_2,a_2^{-q},a_3,a_3^{-q},\ldots,a_n,a_n^{-q},1) \, | \, 
a_i\in\GF(q^2)^{\times}, a_2a_3\ldots a_n=1\}.$$

If $h$ is the matrix given by item (iii) then 
$[g,h]=\diag(1,1,\omega_0,\omega_0^{-1},1,\ldots,1)$; 
this commutator and $A$
generate a group $B$ of order $(q-1)(q^2-1)^{n-2}$, namely 
$$\{\diag(1,1,a_2,a_2^{-q},a_3,a_3^{-q},\ldots,a_n,a_n^{-q},1) \, | \, 
a_i\in\GF(q^2)^{\times}, a_2a_3\ldots a_n\in \GF(q)^{\times}\}.$$

Also, $\Gamma^{q-1}\Gamma^{3U} = 
\diag(\omega^{1-q},\omega^{1-q},\omega^{-3},\omega^{3q},1,\ldots,1,\omega^{1-q})$;
it and $B$ generate a group of order $(q^2-1)^{n-1}$.

Thus $H$ has order 
$(q^2-1)^{n-1}2^{n-1}(n-1)!$, the factor $2^{n-1}$
coming from conjugates of $h$, and the factor of $(n-1)!$ from items (i) and (ii).
So $H$ has index $2(q^2-1)n$ in $N$, as required.

The case $n=2$ is obvious.
\end{proof}

\begin{theorem}\label{SU2n+1sig}
The centraliser of $\sigma=\sigma_{12}(1)$ in $N$ 
has index $2(q^2-1)n (n-1)$ in $N$ and 
is generated by the following elements:
\begin{enumerate}
\item[(i)] $U^{V^2}$ if $n>3$, and $VUU^V$ if $n>4$;
\item[(ii)] both $\Gamma^{V^2}$ and $t^{V^2}$ if $n>2$;
\item[(iii)] if $q$ is even, then $U^t$, else $U^tZ^2$;
\item[(iv)] $\Gamma\Gamma^U$.
\end{enumerate}
\end{theorem}
\begin{proof}
The elements 
$\{\sigma_{\pm i\pm j}(\alpha) \, | \, i\le i \ne j\le n, 
\alpha\in\GF(q^2)^{\times}\}$ of $\SU(2n+1,q)$,
each of which may be written in exactly two ways in this form, 
are permuted transitively by $N$, forming an orbit of size $2(q^2-1)n(n-1)$.

Clearly the itemised elements centralise $\sigma$.
The first two sets generate a copy of $(C_{q^2-1}:C_2)\wr S_{n-2}$,  
and the last two a copy of $C_{q^2-1}:C_2$. Their 
direct product has the claimed index in $N$.
\end{proof}

\begin{theorem}
Let $q=p^e$ for a prime $p$ and let $n\ge 2$.  
Let $G$ be the group generated by 
$\{ \myupsilon, \sigma, \tau, \Gamma, t, U,V \}$, 
omitting $V$ if $n=2$, 
subject to the relations given below. Then $G$ is isomorphic to $\SU(2n + 1,q)$.
\begin{enumerate}
\item[(i)] Defining relations for 
$N=\langle \Gamma,t,U,V\rangle$ as in Theorem $\ref{SU2n+1top}$, 
but omitting relations (ii) of that theorem.

\item[(ii)] Relations that state that the elements listed in 
Theorem $\ref{SU2n+1ups}$ centralise $\myupsilon$.

\item[(iii)] Relations that state that the elements listed in 
Theorem $\ref{SU2n+1sig}$ centralise $\sigma$, 
but omitting items (i) and (ii) of that theorem.

\item[(iv)] Defining relations for $\SU(3,q)$ on 
$\{\myupsilon,\tau^{-1},\Gamma^{-1},t\}$ if $q>2$, 
and on $\{\myupsilon,\myupsilon^{\Gamma^U},\Gamma^{-1},t\}$ 
as in Section $\ref{SU32}$ if $q=2$. 

\item[(v)] Defining relations for $\SL(2,q^2)$ on $\{\sigma,\Delta,W\}$, where 
$\Delta=\Gamma\Gamma^{-U}$, and $W=U$ if $q$ is even, else $W=UZ^2$. 

\item[(vi)] The following instances of Steinberg relations:
\begin{enumerate}
\item[(a)] $[\myupsilon,\myupsilon^U]=\sigma^{-t^U}$; 
\q $([\myupsilon_i(\alpha), \myupsilon_j(\beta)] = \sigma_{i,-j}(\psi^{-1} \alpha \beta^q))$;  
\item[(b)] if $q=4$ then $[\myupsilon, \myupsilon^{\Gamma U}]=\sigma^{\Gamma^7t^U}$;
\q $([\myupsilon_i(\alpha), \myupsilon_j(\omega^{2-q} \beta)] = \sigma_{i,-j}(\omega^7\alpha\beta^q))$; 

\item[(c)] $[\myupsilon,\sigma]=1$;
\q $([\myupsilon_i(\alpha), \sigma_{ij}(\beta)] = 1)$;

\item[(d)] if $q=4$ then $[\myupsilon^\Gamma,\sigma]=1$;
\q $([\myupsilon_i(\omega^{2 - q} \alpha), \sigma_{ij}(\beta)] = 1)$;

\item[(e)] if $q$ is even, then $[\myupsilon,\sigma^U]=\sigma^{\Gamma^rZ^U}\myupsilon^U$, 
where
$\omega^r=\omega^q/(\omega+\omega^q)$; \\ if $q$ is odd then
$[\myupsilon,\sigma^U]=\sigma^{\Gamma^rZ^U}\myupsilon^{-U}$, 
where $\omega^r={1\over2}\omega^{-(q+1)/2}$;  
{\rm (see (5) below)};

\item[(f)] if $q=2$ then 
$[\myupsilon,\sigma^{\Gamma U}]=\sigma^{\Gamma Z^U}\myupsilon^{U^{\Gamma^{-1}}}$;
{\rm (see (5) below with $\beta$ replaced by $\omega \beta)$};

\item[(g)] if $n>2$ then $[\myupsilon,\sigma^V]=1$;
\q $([\myupsilon_i(\alpha),\sigma_{jk}(\beta)] = 1)$;  

\item[(h)] $[\myupsilon,\tau^U]=1$;
\q $($see $(2)$ below$)$;

\item[(j)] the Steinberg relations (v) of Theorem $\ref{SU2npres}$.
\end{enumerate}
The given relations are instances of the more general relations in 
parentheses, with $s=\alpha=\beta=1$ and $(i,j,k,l)=(1,2,3,4)$.  
\end{enumerate}
\end{theorem}
\begin{proof}
The relations omitted from (i) are supplied by (iv), so we may 
define $N$ as a subgroup of $G$.  
The relations omitted from (iii) may be retrieved using (i), (ii), and (a).  

The Steinberg relations (j) involve 
$\Delta=\Gamma\Gamma^{-U}$ and $Z$ where 
$Z = t$ if $q$ is even, else $Z = t\Gamma^{(q+1)/2}$. 

In $\SU(2n+1,q)$, $N$ acts transitively by conjugation on 
$\{\myupsilon_{\pm i}(\alpha) \, | \, 1\le i\le n, \alpha\in\GF(q^2)^{\times}\}$, 
and on 
$\{\sigma_{\pm i\pm j}(\alpha) \, | \, 1\le i \ne j\le n, \alpha\in\GF(q^2)^{\times}\}$,  
so these elements may be defined in $G$, 
by (i), (ii), and (iii), and the omitted relations. 

In $\SU(2n+1,q)$ the centraliser of $\tau$ in $N$ is the direct product 
of the cyclic group of 
order $q+1$ generated by $\Gamma^{q-1}$ with the subgroup of $N$ that fixes 
$e_1$ and $f_1$, and it has order $2^{n-1}(q+1)(q^2-1)^{n-1}(n-1)!$. 
This direct product has index $2(q-1)n$ in $N$, which corresponds to 
the fact that $N$ acts transitively by conjugation on 
$\{\tau_{\pm i}(s) \, | \, 1\le |i|\le n, s\in\GF(q)^{\times}\}$.
It follows from (i), (ii), (iv) and (j) that this group centralises $\tau$ in $G$, 
so $\{\tau_{\pm i}(s) \, | \, 1\le |i|\le n, s\in\GF(q)^{\times}\}$ may be
be defined in $G$, together with the action of $N$ on these elements.

The normaliser $N$ of our chosen torus in $\SU(2n+1,q)$ contains the 
normaliser of our chosen torus in 
$\SU(2n,q)$ as a subgroup of index $q+1$; but the centraliser of $\sigma$ in 
$N$ is correspondingly larger, so the definitions of the Steinberg generators 
$\sigma_{ij}(\alpha)$ are essentially the same in both cases.
Thus $G$ contains a copy of $\SU(2n,q)$, 
and our first objective is to prove 
that the Steinberg relations involving the generators $\myupsilon_i(\alpha)$, and not 
subsumed in (iv), are satisfied.  
These relations are as follows, 
where $|i|, |j|$, and $|k|$ are distinct:
\begin{enumerate}
\item[(1)] $[\myupsilon_i(\alpha),\sigma_{ij}(\beta)]=1$;
\item[(2)] $[\myupsilon_{i}(\alpha),\tau_j(s)]=1$;
\item[(3)] $[\myupsilon_i(\alpha),\myupsilon_j(\beta)]=\sigma_{i,-j}(\psi^{-1}\alpha\beta^q)$;  
\item[(4)] $[\myupsilon_i(\alpha),\sigma_{jk}(\beta)]=1$ if $n>2$;
\item[(5)] $[\myupsilon_i(\alpha),\sigma_{ji}(\beta)]=
\sigma_{i,-j}(\psi^{-1}(\phi+1)\alpha^{q+1}\beta^q)\myupsilon_j(-\alpha\beta)$.  
\end{enumerate}

We have omitted Steinberg relations whose LHS involves the same suffix twice, 
but with opposite signs, since $\sigma_{-i,j}(\alpha)=\sigma_{-j,i}(\alpha^q)$.

We may always assume that $(i,j,k)=(1,2,3)$ since conjugating a 
Steinberg generator by $t_i$ changes the sign of suffices of the form 
$\pm i$ and keeps all other suffices fixed.

We now prove relation (1).
Let $P$ be the set of pairs $(\alpha,\beta)$ for which (1) holds in $G$ with $(i,j) = (1,2)$.  
By relation (c), $(1,1)\in P$.  Clearly if $(\alpha,\beta_1)$ and $(\alpha,\beta_2)$
lie in $P$ then so does $(\alpha,\beta_1+\beta_2)$.  
If $(\alpha_1,\beta)$ and $(\alpha_2,\beta)$
lie in $P$ then so does $(\alpha_1+\alpha_2,\beta)$:  
for $\myupsilon_1(\alpha_1+\alpha_2, \beta)\equiv
\myupsilon_1(\alpha_1,\beta)\myupsilon_1(\alpha_2,\beta)\bmod
\langle\{\tau_1(s) \, | \, {s\in\GF(q)}\}\rangle$, and
$\tau_1(s)$ commutes with $\sigma_{12}(\beta)$ by Theorem \ref{SU2npres} (v)(a).
Conjugating by $\Gamma_1$ and by $\Gamma_2$ 
shows that if $(\alpha,\beta)\in P$ then 
$(\omega^{-q+2}\alpha,\omega\beta)\in P$, 
and $(\omega^{-q+1},\omega^{-1}\beta)\in P$, so
$(\omega^{-2q+3}\alpha,\beta)\in P$.  
But $\gcd(-2q+3,q-1)=1$ and $\gcd(-2q+3,q+1)=\gcd(5,q+1)$, 
so $\omega^{-2q+3}$ generates $\GF(q^2)$ as a field if $q\ne 4$.
So if $q\ne 4$ and $(\alpha,\beta)\in P$ for some $(\alpha,\beta)$, 
then $(\alpha,\beta)\in P$ for that value of $\beta$ and all $\alpha$, 
and hence for all $(\alpha,\beta)$.  
Now let $q=4$.  By the action of $\Gamma_j^{-q-1}$ 
we see that if $(\alpha,\beta)\in P$ then
$(\alpha,\omega_0\beta)\in P$; and since $\omega^{2q-3}=\omega^5=\omega_0$, it
follows that $(\omega_0\alpha,\beta)\in P$.  
Now by relation (d) we see that,
in addition to $(1,1)$, also $(\omega^{-2},1)\in P$, and hence $(\alpha,1)\in P$
for all $\alpha \in \GF(q^2)$.  It follows that $(\alpha,\beta)\in P$ for all $(\alpha,\beta)$,
and (1) is proved.

We now prove relation (2).
Let $P$ be the set of pairs $(\alpha,s)$ for which (2) holds in $G$ with $(i,j)=(1,2)$.  
Clearly if
$(\alpha_1,s)\in P$ and $(\alpha_2,s)\in P$ 
then $(\alpha_1+\alpha_2,s)\in P$, as in the proof of (1), and
if $(\alpha,s_1)\in P$ and $(\alpha,s_2)\in P$, then $(\alpha,s_1+s_2)\in P$.
Conjugating by $\Gamma_1$ and by $\Gamma_2$ shows that if $(\alpha,s)\in P$ 
then $(\omega^{-q+2}\alpha,s)\in P$, and $(\omega^{-q+1}\alpha,\omega_0s)\in P$.
If $q\ne2$ then $\omega^{-q+2}$ generates $\GF(q^2)$ as a field; 
if $q=2$ then $\omega_0=1$, and
$\omega^{-q+1}$ generates $\GF(q^2)$ as a field.  So in either case
if $(\alpha,s)\in P$ for one value of $(\alpha,s)$ 
then $(\alpha,s)\in P$ for that value of $s$ and all values 
of $\alpha$, and hence (if $s\ne 0$) for all values of 
$s$ and all values of $\alpha$.  
But $(1,1)\in P$, by relation (h), so (2) is proved.

We now prove relation (3).
Let $P$ be the set of pairs $(\alpha,\beta)$ for 
which (3) holds in $G$ with $(i,j)=(1,2)$.
Relation (a) implies that $(1,1)\in P$.
As in the proof of (1),
it follows from (2) that if $(\alpha_1,\beta)\in P$ and 
$(\alpha_2,\beta)\in P$ then $(\alpha_1+\alpha_2,\beta)\in P$,
and similarly with the roles of $\alpha$ and $\beta$ reversed.
Conjugating by $\Gamma_1^{q+1}$, one sees that if $(\alpha,\beta)\in P$
then $(\omega_0\alpha,\beta)\in P$ and similarly $(\alpha,\omega_0\beta)\in P$.
Now let $T$ be the set of $\alpha\in\GF(q^2)$ such that  
$(\alpha,\beta)\in P$ for all $\beta\in\GF(q^2)$.  
Since $T$ is closed under addition and (using conjugation by $\Gamma_1$)
under multiplication by $\omega^{q-1}$, 
and since $\omega^{q-1}$ generates $\GF(q^2)$ as a field,
 (3) holds if $T \neq \{0\}$.  
So we are reduced to the case in which, for every $\alpha\in\GF(q^2)^{\times}$, 
the pair $(\alpha,\beta)$ lies in $P$
for exactly $q-1$ elements $\beta\in\GF(q^2)^{\times}$, 
making $(q^2-1)(q-1)$ such pairs in all.  
Now conjugating by $\Gamma_1$ and $\Gamma_2$ shows that if
$(\alpha,\beta)\in P$ then $(\omega^{-q+2}\alpha,\omega^{-q+1}\beta)\in P$ and
$(\omega^{-q+1}\alpha,\omega^{-q+2}\beta)\in P$.  Since $\omega^{q+1}=\omega_0$, 
it follows that $(\omega^2\alpha,\omega^3\beta)\in P$, 
and $(\omega^3\alpha,\omega^2\beta)\in P$.
If $q^2-1$ is not a multiple of 5 then the pairs $(2,3)$ and $(3,2)$ generate
$C_{q^2-1}\times C_{q^2-1}$, so we may now assume that $q^2-1$ is a 
multiple of 5.  Now $P$ contains
at least $(q^2-1)^2/5$ pairs $(\alpha,\beta)$ with $\alpha\beta\ne 0$. But if
(3) does not hold then $P$ contains exactly $(q^2-1)(q-1)$ such pairs.  
Thus (3) follows from one instance, unless $q=4$, when a 
second instance is required, as in relation (b), where $(\alpha,\beta)=(1,\omega^2)$.

We now prove relation (4).  We may assume that $(i,j,k) = (1,2,3)$: 
the result follows trivially from the case $\alpha=\beta=1$, 
as in relation (g), by conjugating with $\Gamma_1$ and $\Gamma_2$.  

We now prove relation (5).
Note that $\sigma_{i,-j}(\psi^{-1}(\phi+1)\alpha^{q+1}\beta^q)$ and 
$\myupsilon_j(-\alpha\beta)$  commute by (4). 
As usual, we restrict to the case $(i,j) = (1,2)$.
Let $P$ be the set of pairs $(\alpha,\beta)$ for 
which (5) holds in $G$ with $(i,j)=(1,2)$.  
Conjugating by $\Gamma_1$ or 
$\Gamma_2$ shows that 
if $(\alpha,\beta)\in P$ then $(\omega^{2-q}\alpha,\omega\beta)\in P$, and
$(\omega^{1-q}\alpha,\omega^{-1} \beta)\in P$.
Thus if $(\alpha,\beta)\in P$ then $(\omega^{3-2q}\alpha,\beta)\in P$.  
It is also clear, from the above information, that
if $(\alpha_1,\beta)\in P$ and $(\alpha_2,\beta)\in P$ then 
$(\alpha_1+\alpha_2,\beta)\in P$.  
If $q\ne 2$ then $\omega^{2-q}$ generates $\GF(q^2)$ as a field, and 
hence (5) follows from any one instance (with $\alpha\beta\ne0$), such 
as relation (e), which 
shows that $(1,1)\in P$.
If $q=2$ then relation (f), which states that 
$(1,\omega)\in P$, completes the proof in this case.

\medskip 
Thus the subgroup $H$ of $G$ generated by the Steinberg 
generators is isomorphic to $\SU(2n+1,q)$,
and it remains to prove that this subgroup equals $G$.  
By (iii), $\Gamma$ and $t$ lie in $H$.
Also, $U$ and $V$ lie in $H$, as in the case of $\SU(2n,q)$. 
Thus $G = H$, as required.
\end{proof}
 
\subsection{A presentation for $\PSU(2n+1,q)$}
The centre of $\SU(2n+1,q)$ has order $\gcd (q + 1, 2n + 1)$; 
if it is non-trivial, then it is generated by 
$(\Gamma ZV)^{2n(q+1)/\gcd(q+1, 2n+1)}$.
 
\subsection{Standard generators for $\SU(2n+1,q)$}
In Table 1 of \cite{even,odd} the standard generators for 
$\SU(2n + 1, q)$ are labelled $s, t, \delta, u, v, x, y$. 
Observe that $s = Z^{-1}; t = \tau^{-1}; 
\delta = \Gamma^{-(q + 1)}; 
u = U; v = V; x = \myupsilon^{(V^{-1})}$; and $y = (\Gamma^{-1})^{(V^{-1})}$.
If $q$ is odd then the presentation generator 
$t = \Gamma^{((q^2 + q) / 2)} s^{-1}$, else $t = s$.
Also $\sigma = [\myupsilon^{tU},\myupsilon^t]^t$,
as a word in the presentation generators $\myupsilon$, $t$, and $U$, 
so it can be written as an explicit word in the standard generators.

\section{A presentation for $\Omega^+(2n,q)$ for $n > 1$}\label{SectOmegaP}

\subsection{Odd characteristic}
\subsubsection{Generators and notation}
Let $q = p^e$ for an odd prime $p$. Let $\omega$ be a primitive element of $\GF(q)$.  

Let $n>1$.  
We take a hyperbolic basis $(e_1,f_1,\ldots,e_n,f_n)$ of the natural module.
We define the following elements of $\SO^+(2n,q)$, where $1\le i,j\le n$
and $i\ne j$, and $s\in\GF(q)$.

$\delta_i=(e_i\mapsto\omega^{-1}e_i, f_i\mapsto\omega f_i)$;  

$\Delta_{ij}=\delta_i\delta_j^{-1}$;  

$\Delta_{-i,j}=\delta_i^{-1}\delta_j^{-1}$;  

$\Delta_{i,-j}=\delta_i\delta_j$; 

$\Delta_{-i,-j}=\delta_i^{-1}\delta_j$; 

$\sigma_{ij}(s)=(e_i\mapsto e_i+se_j, f_j\mapsto f_j-sf_i)$;

$\sigma_{i,-j}(s)=(e_i\mapsto e_i+sf_j, e_j\mapsto e_j-sf_i)$;

$\sigma_{-i,j}(s)=(f_i\mapsto f_i+se_j, f_j\mapsto f_j-se_i)$;

$\sigma_{-i,-j}(s)=\sigma_{ji}(-s)$;

$Z_{ij}=(e_i,f_i)(e_j,f_j)$;

$U=(e_1,e_2)(f_1,f_2)$;

$U'=(e_1,e_2)^-(f_1,f_2)^-$;

$V=(e_1,e_2,\ldots,e_n)(f_1,f_2,\ldots,f_n)$;

$V'=(e_1,e_2,\ldots,e_n)^{\epsilon_n}(f_1,f_2,\ldots,f_n)^{\epsilon_n}$.

\smallskip
All these elements of $\SO^+(2n,q)$ centralise those basis elements that 
they are not stated to move.  
All lie in $\Omega^+(2n,q)$ except for $\delta_i$, and $U$ if
$q\equiv 3 \bmod 4$, and $V$ if $n$ is even and $q\equiv 3\bmod 4$.  
These statements, except for those involving $V$ and $V'$, follow from the 
analysis of $\Omega^+(4,q)$ below.  
Since $V'$ is a product of conjugates of $U'$, it lies in $\Omega^+(2n,q)$.  
Also $U=U'\delta_1^{(q-1)/2}$ and 
$V=V'\delta_1^{(q-1)/2}$ if $n$ is even, so the statements concerning $V$ follow.

\subsubsection{A presentation for $\Omega^+(4,q)$}
Since $\Omega^+(4,q)$ is the central product of two copies of
$\SL(2,q)$, there is no difficulty in writing down the required presentation, 
but we look more closely at this case as a preparation for the general case.

Consider the spaces $\langle e_1,e_2\rangle$ and $\langle f_1,f_2\rangle$.  
They are fixed 
by $\{ \sigma_{12}(1), \Delta_{12}, U'\}$. These elements act on  
$\langle e_1,e_2\rangle$ by the matrices 
$\left(\begin{smallmatrix}1&1\cr0&1\end{smallmatrix}\right)$,
$\left(\begin{smallmatrix}\omega^{-1}&0\cr0&\omega\end{smallmatrix}\right)$, 
and
$\left(\begin{smallmatrix}0&1\cr-1&0\end{smallmatrix}\right)$ respectively; they act on 
$\langle f_1,f_2\rangle$ by the inverse transpose of these matrices.  So these 
elements generate a copy of $\SL(2,q)$.
Now consider the spaces $\langle e_1,f_2\rangle$ and $\langle e_2,f_1\rangle$.  These are
fixed by $\{ \sigma_{1,-2}(1), \Delta_{1,-2}, Z_{12}U'\}$. These 
elements 
act on $\langle e_1,f_2\rangle$ as the above three matrices, 
and as the inverse transpose on $\langle e_2,f_1\rangle$,  
thus generating a second copy of $\SL(2,q)$.  The two copies of
$\SL(2,q)$ commute with each other, and 
$\Delta_{12}^{(q-1)/2}=\Delta_{1,-2}^{(q-1)/2}=-I_4$. 
This gives the required presentation on 
$\{ \sigma_{12}(1), \Delta_{12}, U', \sigma_{1,-2}(1), 
\Delta_{1,-2}, Z_{12}U'\}$,
omitting $\Delta_{12}$ and $\Delta_{1,-2}$ if $e = 1$.

\subsubsection{A presentation for $\Omega^+(2n, q)$ for $n \ge 3$ and $e > 1$}
We give a presentation for $\Omega^+(2n,q)$ on the generating set 
$\{ \sigma=\sigma_{12}(1), \Delta=\Delta_{12}, Z=Z_{12}, U',V'\}$.  

Let $N_1$ be the subgroup of $\Omega^+(2n,q)$ generated by 
$\{ Z,U',V'\}$. This is an extension
of an elementary abelian $2$-group $A$ of rank $2(n-1)$ by a copy of $S_n$.  
Here $A$ is
the direct product of a subgroup of rank $n-1$, 
generated by $U'^2$ and its conjugates, with
a second subgroup of the same rank generated by $Z$ and its conjugates.
So if $B$ is the copy of the four-group generated by the transformations
$(e_1\mapsto -e_1, f_1\mapsto -f_1)$ and $(e_1\mapsto f_1, f_1\mapsto e_1)$, 
the other basis vectors being fixed, then 
$N_1$ is the subgroup of $B\wr S_n$ of index $4$ consisting of the 
elements of determinant $1$ and spinor norm $1$.

\begin{theorem}\label{OplusoddN1}
Let $n\ge 3$.  Let $G$ be the group generated by  
$\{Z,U',V'\}$ subject to the relations given below. Then $G$ is isomorphic to $N_1$.
\begin{enumerate}
\item[(i)] Defining relations for $\langle U',V'\rangle$ as 
in Theorems $\ref{signedsymmetricodd}$ and $\ref{signedsymmetriceven}$.
\item[(ii)] If $n>3$ then $[Z,U'^{V'^2}]=1$.
\item[(iii)] If $n>4$ then $[Z,V'U'U'^{V'}]=1$.
\item[(iv)] $Z^2=1$.
\item[(v)] $[Z, U'^{2V'}]=1$.
\item[(vi)] $ZZ^{V'}=Z^{V'U'}$.
\item[(vii)] $[Z,U']=1$.
\item[(viii)] $[Z,Z^{V'}]=1$.
\item[(ix)] If $n>3$ then $[Z,Z^{V'^2}]=1$.
\end{enumerate}
\end{theorem}

\begin{proof}
Relations (ii) and (iii) state that $Z$ is centralised by the 
subgroup of $\langle U',V'\rangle$
that centralises $1$ and $2$ as signed permutations, and also by $U'^2$.  
These relations, together with (v), also
imply that $Z$ is centralised by all $\langle U',V'\rangle$-conjugates 
of $U'^2$, so we may
label these conjugates as $Z_{ij}$, 
where $1\le i\ne j\le n$.  Now (vii) implies that $Z_{12}=Z_{21}$,
and hence, by conjugation, that $Z_{ij}=Z_{ji}$ for all $i\ne j$.

It follows similarly from (viii) and (ix) that
all $Z_{ij}$ commute, and hence, by (iv), 
generate an elementary abelian $2$-group.  
Also, (vi) states that
$Z_{12}Z_{23}=Z_{13}$, so $Z_{ij}Z_{jk}=Z_{ik}$ 
for any three distinct values of $i,j,k$, and so
the subgroup of $G$ generated by the $Z_{ij}$ is generated by
$\{Z_{12},Z_{23},\ldots,Z_{n-1,n}\}$, and hence has order at most $2^{n-1}$. 
Thus $G$ is a split extension of a normal subgroup of order 
at most $2^{n-1}$ by the group 
$\langle U',V'\rangle$ of order
$2^{n-1}n!$, by (i).
But $N_1$, as a subgroup of $\Omega^+(2n,q)$, has order $2^{2n-2}n!$, 
and the result follows.
\end{proof}

Let $N=\langle \Delta,Z,U',V'\rangle \leq \Omega^+(2n,q)$. 
It is easy to see that the normal subgroup of $N$ generated as 
normal subgroup by $\Delta$, and as subgroup 
by $\{\Delta_{\pm i,\pm j} \, | \, |i|\ne |j|\}$, 
is generated by 
$\{\Delta_{1,-2},\Delta_{12},\Delta_{23},\ldots,\Delta_{n-1,n}\}$, and 
has order $(q-1)^n/2$, the
factor of $1/2$ coming from the relation $\Delta_{12}^{(q-1)/2}=\Delta_{1,-2}^{(q-1)/2}$.  
Since this normal subgroup contains $U'^2$, it follows that $N$ has 
order $(q-1)^n2^{n-2}n!$.

\begin{theorem}\label{OmegaplusoddN}
Let $q=p^e$ for an odd prime $p$ and $e>1$, and let $n\ge 3$.  
Let $G$ be the group generated by  
$\{\Delta, Z,U',V'\}$ subject to the relations given below. Then $G$ is isomorphic to $N$.
\begin{enumerate}
\item[(i)] Defining relations for 
$N_1 = \langle Z,U',V'\rangle$ as in Theorem $\ref{OplusoddN1}$.
\item[(ii)] If $n>3$ then $[\Delta,U'^{V'^2}]=1$.
\item[(iii)] If $n>4$ then $[\Delta,V'U'U'^{V'}]=1$.
\item[(iv)] If $n>3$ then $[\Delta,Z^{V'^2}]=1$.
\item[(v)] $\Delta^{U'} = \Delta^{-1}$, and $[\Delta,U'^{2V}]=1$.  
\item[(vi)] $\Delta^{(q-1)/2}=U'^2$.
\item[(vii)] $\Delta\Delta^{V'}=\Delta^{V'U'}$.  
\item[(viii)] $[\Delta,\Delta^{V'}]=1$.
\item[(ix)] If $n>3$ then $[\Delta,\Delta^{V'^2}]=1$.
\item[(x)] $\Delta^{Z}=\Delta^{-1}$.
\end{enumerate}
\end{theorem}
\begin{proof}
We first work in $\Omega^+(2n,q)$.
Since $n>2$, $N_1$ permutes transitively the elements $\Delta_{\pm i,\pm j}$.  
Here $\langle U',V'\rangle$ permutes the suffices as unsigned permutations, so $U'^2$
centralises, and $Z_{ij}$ sends $i$ to $-i$, and $j$ to $-j$.
Since $\Delta_{ji}=\Delta_{-i,-j}$  
these $4n(n-1)$ symbols define $2n(n-1)$ distinct matrices. 
Now $\Delta=\Delta_{12}$ is centralised by the subgroup of $N_1$ that centralises  
$\langle e_1,f_1,e_2,f_2\rangle$ and has order $2^{2(n-3)}(n-2)!$, 
and also by $ZU'$ and $U'^{2V'}$, which contribute an additional factor of 
8 to the order of the centraliser. 
Thus the centraliser of $\Delta$ in $N_1$ has 
order $2^{2n-3}(n-2)!$ and index $2n(n-1)$ as required.

We now work in the presented group $G$.  
By the above remarks, and using relations (i) to (v) and (x), 
we may define the elements $\Delta_{\pm i,\pm j}$
in $G$ as the $N_1$-conjugates of $\Delta$.
Relation (v) states that $\Delta_{21} =\Delta_{12}^{-1}$, and hence  
$\Delta_{ij}=\Delta_{ji}^{-1}$ whenever $|i|\ne|j|$.
Relations (viii) and (ix) state that
$\Delta_{12}$ commutes with $\Delta_{23}$, and with $\Delta_{34}$ if $n>3$, so 
the group $K$ generated
by $\{\Delta_{ij} \, | \, i,j>0, i \ne j\}$ is abelian.  
Since $\Delta_{-1,-2}=\Delta_{12}^{-1}$ by (x), 
it follows that $\Delta_{-i,-j}=\Delta_{ij}^{-1}$ if $|i|\ne |j|$.
Relation (vii) states that $\Delta_{12}\Delta_{23}=\Delta_{13}$, 
and hence 
$\Delta_{ij}\Delta_{jk}=\Delta_{ik}$ whenever $|i|$, $|j|$, and $|k|$ are distinct. 
It is now easy to prove, since $n\ge3$, that the group $H$ generated 
by $\{\Delta_{\pm i\pm j}\,|\,|i| \ne |j|\}$ is 
abelian, and that $K$ is generated by $\{\Delta_{i, i+1} \, | \, 1\le i < n\}$.
Since $\Delta_{i,-j}=\Delta_{i1}\Delta_{1,-2}\Delta_{-2,-j}
=\Delta_{i1}\Delta_{1,-2}\Delta_{2,j}^{-1}$ if $|i|\ne |j|$,
it follows that $H$ is generated by
$\{\Delta_{1,-2},\Delta_{12},\Delta_{23},\ldots,\Delta_{n-1,n}\}$. 
Now $H$ has exponent dividing $q-1$ by (vi), and order dividing
$(q-1)^n/2$, the factor of 1/2 coming from the fact that 
$\Delta_{12}^{(q-1)/2}=U'^2=U'^{2Z}=\Delta_{1,-2}^{(q-1)/2}$.
Since the corresponding subgroup of $\Omega^+(2n,q)$ has this order, so does $H$.
The result follows.
\end{proof}

\begin{theorem}\label{Op2nsigstab}
Let $q=p^e$ for an odd prime $p$ and $e>1$, and let $n\ge 3$.  
The centraliser of $\sigma =\sigma_{12}(1)$ in $N$ 
has order $(q-1)^{n-1}2^{n-3}(n-2)!$ and is generated by the following elements:
\begin{enumerate}
\item[(i)] $U'^{V'^2}$ if $n>3$;
\item[(ii)] $V'U'^{-1}U'^{-V'}$ if $n>4$;
\item[(iii)] $Z^{V'^2}$ if $n>3$;
\item[(iv)] $\Delta^{V'^2}$ if $n>3$;
\item[(v)] $\Delta^{Z^{V'}}$;  
\item[(vi)] $ZU'$;
\item[(vii)] $\Delta^{{U'}^{V'}}\Delta^{V'}$ if $n<5$.
\end{enumerate}
\end{theorem}

\begin{proof}
Observe that $N$ acts transitively by conjugation on 
$\{\sigma_{ij}(s) \, | \, |i|\ne |j|, s \in \GF(q)^\times \}$, 
a set of size $2n(n-1)(q-1)$, as follows.  
If $g\in\langle U',V'\rangle$ sends $i$ to $\epsilon_1k$, and sends $j$ to $\epsilon_2l$, 
where $\epsilon_1,\epsilon_2 = \pm1$, then
$\sigma_{ij}(s)^g=\sigma_{kl}(\epsilon_1\epsilon_2s)$;
and $Z_{ij}$ acts on the suffices, interchanging $i$ with $-i$, and $j$ with $-j$, 
fixing all other suffices and $s$.   The action of $\Delta_{ij}$
follows from the formulae
$\sigma_{ij}(s)^{\delta_i}=\sigma_{ij}(s\omega)$ and 
$\sigma_{ij}(s)^{\delta_j}=\sigma_{ij}(s\omega^{-1})$.

Items (v) and (vii) equal $\delta_1\delta_2$ and 
$\delta_1\delta_2\delta_3^{-2}$ respectively, 
while item (vi) centralises $\sigma$ since $\sigma_{12}(s) = \sigma_{-2,-1}(-s)$.
The other itemised elements clearly centralise $\sigma$.  
It remains to prove that the group generated by the itemised elements has 
the claimed order.

Suppose first that $n=3$.  Clearly items (v) and (vii) generate a group of order 
$(q-1)^2/2$, and with item (vi) generate a group of order $(q-1)^2$, as required.



















Now let $n>3$.  The centraliser of $\sigma$ in $N$ is the direct product of
the centraliser $C$ in $N$ of $\langle e_1,f_1,e_2,f_2\rangle$, which has 
order $(q-1)^{n-2}2^{n-4}(n-2)!$,
with a copy $Q$ of $Q_{2(q-1)}$, namely $\langle\Delta_{1,-2}, Z_{12}U'\rangle$. 
For this subgroup centralises 
$\sigma$, and has order $(q-1)^{n-1}2^{n-3}(n-2)!$.  
These generators of $Q$ are supplied by items (v) and (vi).
Observe that $C$ is generated by the union of the following sets:  
the elements $\Delta_{k,k+1}$ for $k>2$, 
supplied by item (iv) and its conjugates 
under items (i) and (ii),  
which generate a group of order $(q-1)^{n-3}$;
the elements $Z_{k,k+1}$ for $k>2$, supplied by item (iii) and its 
conjugates under items (i) and (ii), which 
contribute a factor of $2^{n-3}$; 
the elements in items (i) and (ii), which contribute
a factor of $(n-2)!$;  
and $\delta_3^2$, which contributes the remaining 
factor of $(q-1)/2$.
If $n>4$ then $\delta_3^2=\Delta^{{V'}^2}\Delta^W$, 
where $W={V'}^2Z^{{V'}^3}$, so this generator is not 
required; if $n=4$ then this element is supplied by items (v) and (vii); 
that is, by $\delta_1\delta_2$ and $\delta_1\delta_2\delta_3^{-2}$.
\end{proof}

\begin{theorem}\label{OmegaPlusOdd}
Let $q=p^e$ for an odd prime $p$ and $e>1$, and let $n\ge 3$.  
Let $G$ be the group generated by 
$\{\sigma, \Delta, Z,U',V'\}$ subject to the relations given below.  
Then $G$ is isomorphic to $\Omega^+(2n,q)$.  
\begin{enumerate}
\item[(i)] Defining relations for $N=\langle\Delta,Z,U',V'\rangle$ as in 
Theorem $\ref{OmegaplusoddN}$, but omitting relation (vi) of that theorem.  

\item[(ii)] Relations that state that the elements listed in 
Theorem $\ref{Op2nsigstab}$ centralise $\sigma$.

\item[(iii)] Relations that present $\SL(2,q)$ on $\{\sigma, \Delta, U'\}$.  

\item[(iv)] The following instances of Steinberg relations:
\begin{enumerate}
\item[(a)] $[\sigma,\sigma^{V'}]=\sigma^{V'U'^{-1}}$; \q$([\sigma_{ij}(s),\sigma_{jk}(t)]=\sigma_{ik}(st))$;
\item[(b)] $[\sigma,\sigma^{V'U'^{-1}}]=1$; \q$([\sigma_{ij}(s),\sigma_{ik}(t)]=1)$;
\item[(c)] $[\sigma,\sigma^W]=1$, where $W=U'^{V'U'^{-1}}$; \q$([\sigma_{ij}(s),\sigma_{kj}(t)]=1)$;
\item[(d)] if $n>3$ then $[\sigma,\sigma^{V'^2}]=1$; \q$([\sigma_{ij}(s),\sigma_{kl}(t)]=1)$;
\item[(e)] $[\sigma,\sigma^{Z^V}]=1$; \q$([\sigma_{ij}(s),\sigma_{i,-j}(t)]=1)$;
\item[(f)] if $n=4$ then $[\sigma,\sigma^{Z^VV^2}]=1$; \q$([\sigma_{ij}(s),\sigma_{k,-l}(t)]=1)$.
\end{enumerate}
The given relations are instances of the more general relations in 
parentheses, with $s=t=1$ and $(i,j,k,l)=(1,2,3,4)$.  
\end{enumerate}
\end{theorem}
\begin{proof}
The relation omitted from (i) is supplied by (iii).

By (i) we may define the elements $Z_{ij}$ and $\Delta_{ij}$ in $G$.
Since $N$ acts transitively, in $\Omega^+(2n,q)$, by conjugation
on the set of all root elements $\sigma_{ij}(s)$,
where $s\ne 0$ and $|i|\ne|j|$,
we may define $\sigma_{ij}(s)$ in $G$ by the same conjugation, by (ii).  

Recall that $\Omega^+(4,q)$ is isomorphic to the direct product of two copies
of $\SL(2,q)$ with the centres amalgamated, the natural representation of
$\Omega^+(4,q)$ being the tensor product of two copies of the 
natural representation of $\SL(2,q)$.
Here $\langle \sigma_{12},\Delta_{1,-2}, U'\rangle$ is one copy of $\SL(2,q)$, and 
$\langle \sigma_{1,-2}, \Delta_{12},Z_{12}U'\rangle$ is the other.  
The first of these subgroups in $G$ is isomorphic to $\SL(2,q)$ by (iii), 
and, since $Z_{12}U'=U'^{Z_{23}}$, so is the second, by conjugation with $Z_{23}$.  
By (e), (i) and (ii), these subgroups commute with each other.  
The amalgamation of the centres 
follows from the presentation for $N$.  It follows that, in $G$, the subgroup 
$\langle\sigma_{\pm1,\pm2}(1),\Delta_{12},\Delta_{1,-2},Z, U'\rangle$ is 
isomorphic to $\Omega^+(4,q)$.

The parenthetical Steinberg relations under each of the headings 
(a) to (d) all follow from one 
example, by conjugation with $N$.  

All remaining Steinberg relations may be deduced as follows.  
Those that only involve positive suffices follow easily, as in the case of $\SL(n,q)$.  
Indeed, $\langle \sigma,\Delta,U',V'\rangle$ is isomorphic to  $\SL(n,q)$.  
Since $\langle\sigma_{\pm1,\pm2}(1),\Delta_{12},\Delta_{1,-2},Z,U'\rangle$ is 
isomorphic, in $G$, to $\Omega^+(4,q)$,
those Steinberg relations that only involve two suffices, 
with various signs, all hold in $G$.
Relations that involve four different suffices, with various signs, follow 
from the one case
(d) by conjugation, except in the case $n=4$, when the extra relation (f) is needed, 
as we can only change the signs of the suffices two at a time.
It remains to consider the case when three suffices occur:  namely, the LHS of
the relation is of the form $[\sigma_{ij}(s),\sigma_{uv}(t)]$, where 
$\{a,b,c\} := \{|i|, |j|, |u|, |v|\} $ forms a set of size 3.  
Conjugating by $Z_{ab}$, or $Z_{bc}$, or $Z_{ac}$
changes the signs of any two of these suffices at all their appearances.  Also 
the relation $\sigma_{ij}(s)=\sigma_{-j,-i}(-s)$ allows us to change the 
sign of $i$ and $j$ (or of $u$ and $v$).
Hence we reduce to the case in which all suffices are positive.

Thus all Steinberg relations hold in $G$, and it follows that the subgroup 
$H$ of $G$ generated by the Steinberg generators is isomorphic to $\Omega^+(2n,q)$.
By (iii) $H$ contains $U'$, so it contains all the $N$-conjugates of $U'$, 
and hence contains $V'$.  The result follows.
\end{proof}

\subsubsection{A presentation for $\Omega^+(2n,p)$}
We give a presentation for $\Omega^+(2n,p)$ on the generating set 
$\{ \sigma, Z, U',V'\}$.  
Observe that $N_1=\langle Z,U',V'\rangle$, 
as a subgroup of $\Omega^+(2n,p)$, acts transitively by conjugation 
on the set of elements 
$\sigma_{ij}(\epsilon)$, where $|i|\ne|j|$ and $\epsilon=\pm 1$. 
For example, $\sigma_{12}(1)^{{U'}^{2V'}}=\sigma_{12}(-1)$.
The number of symbols of
this form is $8n(n-1)$; but as group elements they are equal in pairs, since 
$\sigma_{ij}(s)=\sigma_{-j,-i}(-s)$.  The centraliser of
$\sigma$ in $N_1$ is the direct product of the cyclic group of order 4 generated by
$Z_{12}U'$ with the subgroup of $N_1$ that 
has index $4n(n-1)$ and 
centralises 
$\langle e_1,f_1,e_2,f_2\rangle$. 

\begin{theorem}
Let $p$ be an odd prime and let $n\ge 3$.  Let $G$ be the group generated by 
$\{\sigma,Z,U',V'\}$ subject to the relations given below.  
Then $G$ is isomorphic to $\Omega^+(2n,p)$.  
\begin{enumerate}
\item[(i)] Defining relations for $N_1=\langle Z,U',V'\rangle$ as in 
Theorem $\ref{OplusoddN1}$.

\item[(ii)] Relations that state that the elements listed in 
Theorem $\ref{Op2nsigstab}$ centralise $\sigma$, but omitting those involving $\Delta$.

\item[(iii)] Relations that present $\SL(2,p)$ on $\{\sigma,U'\}$.
\item[(iv)] The Steinberg relations (iv) of Theorem $\ref{OmegaPlusOdd}$.
\end{enumerate}
\end{theorem}
\begin{proof}

By (i) we may define the elements $Z_{ij}$ in $G$.

By (ii) we may define $\sigma_{ij}(\pm1)$ in $G$ for $|i|\ne|j|$ by conjugation, 
and then define $\sigma_{ij}(s)$ to be $\sigma_{ij}(1)^s$, 
taking $s$ as an integer, since $\sigma_{ij}(1)^p=\sigma^p=1$ by (iii).

As in the case $e>1$, 
$\langle\sigma_{12}(1), \sigma_{1,-2}(1),Z, U'\rangle$ is 
isomorphic in $G$ to $\Omega^+(4,p)$.

The parenthetical Steinberg relations under each of the headings 
(a) to (d) all follow from one example, 
as follows.   We may assume, by conjugation, that
the relation holds for the given suffices with $s=t=1$.  
In relations (b), (c), (d) and (e), where the
RHS is 1, the case of general $s$ and $t$ follows from the 
fact that if $[a,b]=1$ and $[a,c]=1$
then $[a,bc]=1$.  Relation (a) follows similarly: if $[a,b,a]=[a,b,b]=1$ then
$[a^m,b]=[a,b^m]=[a,b]^m$.

We thus deduce that all Steinberg relations hold in $G$, 
and complete the proof arguing as in the case $e>1$.
\end{proof}
\noindent 
Observe that 
$\Delta = {({\sigma^{\omega - \omega^2}})}^{U'}  \sigma^{\omega^{-1}} 
          {(\sigma^{\omega- 1})}^{U'} \sigma^{-1}$.

\subsubsection{A presentation for ${\rm P}\Omega^+(2n,q)$}
The centre of $\Omega^+(2n,q)$ is trivial, unless $q^n\equiv 1\bmod 4$, 
when it has order 2. 
If it is non-trivial, then it is generated by $V'^n$ if $n$ is even, and by
$(V'Z\Delta)^{n(q-1)/4}$ if $n$ is odd.  For $(V'Z\Delta)^n$ acts
on one $n$-dimensional space containing $e_1$ by multiplication by 
$\omega^{-2}$, and on a complementary $n$-dimensional
space containing $f_1$ by multiplication by $\omega^2$.

\subsubsection{Standard generators for $\Omega^+(2n,q)$} 
In Table 2 of \cite{odd} the non-trivial standard generators for 
$\Omega^+(2n, q)$ are labelled $s, t, \delta, s', t', \delta', v$.
Observe that $s = ZU'^{-1}; t = (\sigma^{-1})^{Z^{V'}}; 
\delta = \Delta^{(Z^{V'^{-1}})}; s'= U'; t' = \sigma;
\delta' = \Delta^{-1}$; and $v = V'$.  
The presentation generator $Z = s s'$.


\subsection{Even characteristic}
\subsubsection{Generators and notation} 
Let $q = 2^e$.  Let $\omega$ be a primitive element of $\GF(q)$.  

Let $n>1$.  
We take a hyperbolic basis $(e_1,f_1,\ldots,e_n,f_n)$  of the natural module,
so that, if $Q$ is the quadratic form that is preserved, 
then $Q(e_i)=Q(f_i)=0$ for all $i$, and the corresponding
bilinear form is defined by $e_i.f_i=1$ and vanishes 
on all other pairs of basis vectors.

We define the following elements of $\SO^+(2n,q)$, where 
$1\le i,j\le n$ and $i \not= j$, and $s\in\GF(q)$.
 
$\delta_i=(e_i\mapsto\omega^{-1}e_i, f_i\mapsto\omega f_i)$ if $e>1$;

$z_i=(e_i,f_i)$;

$Z_{ij}=z_iz_j$;

$\sigma_{ij}(s)=(e_i\mapsto e_i+se_j, f_j\mapsto f_j+sf_i)$;

$\sigma_{-i,j}(s)=(f_i\mapsto f_i+se_j, f_j\mapsto f_j+se_i)=\sigma_{ij}(s)^{z_i}$;

$\sigma_{i,-j}(s)=(e_i\mapsto e_i+sf_j, e_j\mapsto e_j+sf_i)=\sigma_{ij}(s)^{z_j}$;

$\sigma_{-i,-j}(s)=\sigma_{ji}(s) = \sigma_{ij}(s)^{Z_{ij}}$;

$U=(e_1,e_2)(f_1,f_2)$;

$V=(e_1,e_2,\ldots,e_n)(f_1,f_2,\ldots,f_n)$.

\smallskip
All these elements of $\SO^+(2n,q)$ centralise those basis elements that they are not
stated to move.  All lie in $\Omega^+(2n,q)$ except for $z_i$. 
Note that $\langle \sigma_{12}(1), 
\delta_1\delta_1^{-U}, U, V\rangle$ is isomorphic to $\SL(n,q)$.

\subsubsection{A presentation for $\Omega^+(4,q)$}
Recall that $\Omega^+(4,q)$ is the direct product of two copies of
$\SL(2,q)$.  
One acts on the spaces $\langle e_1,e_2\rangle$ and $\langle f_1,f_2\rangle$
and is generated by $\{ \sigma_{12}(1), \delta_1\delta_2^{-1}, U\}$; 
the other acts on the spaces
$\langle e_1,f_2\rangle$ and $\langle f_1,e_2\rangle$ and is generated by 
$\{ \sigma_{1,-2}(1), \delta_1\delta_2, Z_{12}U\}$.  
The two representations of each copy of $\SL(2,q)$ 
are linked by the inverse transform automorphism.  Conjugating by $z_2$ 
interchanges the generating set
of the first of these copies of $\SL(2,q)$ with that of the second,
reflecting the fact that $\SO^+(4,q)$ is isomorphic to $\SL(2,q)\wr C_2$.
This gives the required presentation on 
$\{ \sigma_{12}(1), \delta_1\delta_2^{-1}, U, 
\sigma_{1,-2}(1), \delta_1\delta_2, Z_{12}U\}$,
omitting $\delta_1\delta_2^{-1}$ and $\delta_1\delta_2$ if $e = 1$.

\subsubsection{A presentation for $\Omega^+(2n, q)$ for $n \ge 3$ and $e > 1$} 
We give a presentation for $\Omega^+(2n,q)$ on the generating set 
$\{ \sigma=\sigma_{12}(1), \delta=\delta_1, Z=Z_{12},U,V\}$. 

Let $N_1$ be the subgroup of $\Omega^+(2n,q)$ generated by $\{ Z,U,V\}$.
It is isomorphic to a subgroup of index 2 in \mbox{$C_2\wr S_n$}.
\begin{theorem}\label{OmegaplusNoneeven}
Let $n\ge 3$.  
Let $G$ be the group generated by
$\{Z,U,V\}$ subject to the relations given below.  Then $G$ is isomorphic to $N_1$.
\begin{enumerate}
\item[(i)] Defining relations for $S_n$ on $\{U,V\}$. 
\item[(ii)] If $n > 3$ then $[Z,U^{V^2}]=1$.
\item[(iii)] If $n > 4$ then $[Z,VUU^V]=1$.
\item[(iv)] $Z^2=1$.
\item[(v)] If $n > 3$ then $[Z,Z^{V^2}]=1$.
\item[(vi)] $ZZ^V=Z^{U^V}$.
\item[(vii)] $[Z,U]=1$.
\item[(viii)] $[Z,Z^V]=1$.
\end{enumerate}
\end{theorem}
\begin{proof}
Since, for the stated $n$,  
$U^{V^2}$ and $VUU^V$, as permutations of $\{1,2,\ldots,n\}$, 
stand for $(3,4)$ and $(3,4,\ldots,n)$, 
relations (ii) and (iii) allow us to define $Z_{ij}$ in $G$, 
and (vii) then implies that $Z_{12}=Z_{21}$, so
$Z_{ij}=Z_{ji}$ for all $i\ne j$.  
By (v) and (viii) $Z_{12}$ commutes with $Z_{23}$ and with $Z_{34}$,
and hence all $Z_{ij}$ commute with each other, generating an 
elementary abelian 2-group $A$ by (iv).
Finally $Z_{12}Z_{23}=Z_{13}$ by (vi), so $Z_{ij}Z_{jk}=Z_{ik}$ for all 
distinct $i, j, k$, and $A$ 
is generated by $\{Z_{12},Z_{23},\ldots,Z_{n-1,n}\}$ and has rank at most $n-1$.
But this group has rank $n-1$ in $N_1$, and the result follows.
\end{proof}

Let $N=\langle \delta,Z,U,V\rangle \leq \Omega^+(2n,q)$. 
It is isomorphic to a subgroup of index 2 in \mbox{$D_{2(q-1)} \wr S_n$}.

\begin{theorem}\label{OplusevenN}
Let $q=2^e$ with $e > 1$, and let $n\ge 3$.  
Let $G$ be the group generated by $\{\delta,Z,U,V\}$ subject to the relations given below.  
Then $G$ is isomorphic to $N$.
\begin{enumerate}
\item[(i)] Defining relations for $N_1 = \langle Z, U,V \rangle$ as 
in Theorem $\ref{OmegaplusNoneeven}$.
\item[(ii)] $[\delta,U^V]=1$.
\item[(iii)] If $n > 3$ then $[\delta, VU]=1$.
\item[(iv)] $[\delta,\delta^U]=1$.
\item[(v)] $\delta^{q-1}=1$.
\item[(vi)] $\delta^Z=\delta^{-1}$.
\item[(vii)] $[\delta,Z^V]=1$.
\end{enumerate}
\end{theorem}

\begin{proof}
Since $U^V$ and $VU$ stand for $(2,3)$ and $(2,3,\ldots,n)$ respectively, 
(ii) and (iii) allow us to define $\delta_i$ in $G$.  
Now (iv) implies that $\delta_1$ commutes with $\delta_2$,
and hence the $\delta_i$ generate an abelian group $A$, of 
exponent dividing $q-1$ by (v), and
of rank at most $n$.
Since $A$ is a group of order $(q-1)^n$ in $\Omega^+(2n,q)$ this must also
be the case in $G$.  In $G$, the subgroup of $N$ generated by the $Z_{ij}$
normalises $A$.  This is guaranteed by (vi) and (vii), which assert 
that $\delta_1$ is inverted by $Z_{12}=Z_{21}$ and centralised by $Z_{23}$.  
So $G/A$ is isomorphic to a subgroup of 
index 2 in $C_2\wr S_n$.
The result follows.
\end{proof}

Since $n \ge 3$, $N$ acts transitively by conjugation on 
$\{\sigma_{ij}(s) \, | \, |i|\ne |j|, s \in \GF(q)^\times \}$, 
a set of size $2n(n-1)(q-1)$, as follows.  
Elements of $\langle Z, U,V\rangle$ permute the suffices, 
acting as signed permutations, where
$Z_{ij}$ maps $i$ to $-i$, and $j$ to $-j$, and fixes the other suffices.

%
%
%
%
%
%
%
%
%
\begin{theorem}\label{CentOPlus-sigma}
Let $q=2^e$ with $e > 1$, and let $n\ge 3$.  
The centraliser of $\sigma=\sigma_{12}(1)$ in $N$ 
has index $2n(n-1)(q-1)$ in $N$ and is generated by the following elements:
\begin{enumerate}
\item[(i)] $U^{V^2}$ if $n>3$, 
\item[(ii)] $VUU^V$ if $n>4$;
\item[(iii)] $Z^{V^2}$ if $n > 3$;
\item[(iv)] $\delta^{V^2}$; 
\item[(v)] $ZU$;
\item[(vi)] $\delta\delta^U$.
\end{enumerate}
\end{theorem}
\begin{proof}
The centraliser of $\sigma$ in $N$ is the direct product of
the centraliser in $N$ of $\langle e_1,f_1,e_2,f_2\rangle$, 
which has order $(q-1)^{n-2}2^{n-3}(n-2)!$,
with a copy of $D_{2(q-1)}$, namely $\langle\delta_1\delta_2, Z_{12}U\rangle$. 
\end{proof}

\begin{theorem} \label{Opluseven}
Let $q=2^e$ with $e>1$, and let $n \ge 3$.  Let $G$ be the group generated by 
$\{\sigma, \delta, Z,U,V\}$ subject to the relations given below.  
Then $G$ is isomorphic to $\Omega^+(2n,q)$.  
\begin{enumerate}
\item[(i)] Defining relations for $N=\langle\delta,Z,U,V\rangle$ as in 
Theorem $\ref{OplusevenN}$.

\item[(ii)] Relations that state that the elements listed in
Theorem $\ref{CentOPlus-sigma}$ centralise $\sigma$.

\item[(iii)] Relations that present $\SL(2,q)$ on $\{\sigma, \Delta,U\}$, where 
$\Delta=[U,\delta]=\delta_1\delta_2^{-1}$, but omitting the relation $\Delta^{q-1}=1$.

\item[(iv)] The following instances of Steinberg relations:
\begin{enumerate}
\item[(a)] $[\sigma,\sigma^V]=\sigma^{VU}$; \q$([\sigma_{ij}(s),\sigma_{jk}(t)]=\sigma_{ik}(st))$;
\item[(b)] $[\sigma,\sigma^{VU}]=1$; \q$([\sigma_{ij}(s),\sigma_{ik}(t)]=1)$;
\item[(c)] $[\sigma,\sigma^W]=1$, where $W=U^{VU}$; \q$([\sigma_{ij}(s),\sigma_{kj}(t)]=1)$;
\item[(d)] if $n>3$ then $[\sigma,\sigma^{V^2}]=1$; \q$([\sigma_{ij}(s),\sigma_{kl}(t)]=1)$;
\item[(e)] $[\sigma,\sigma^{Z^V}]=1$; \q$([\sigma_{ij}(s),\sigma_{i,-j}(t)]=1)$;
\item[(f)] if $n=4$ then $[\sigma,\sigma^{Z^VV^2}]=1$; \q$([\sigma_{ij}(s),\sigma_{k,-l}(t)]=1)$.
\end{enumerate}
The given relations are instances of the more general relations in parentheses, 
with $s=t=1$ and $(i,j,k,l)=(1,2,3,4)$.  
In the parenthetical version, $i,j,k,l$ are distinct and positive.
\end{enumerate}
\end{theorem}
\begin{proof}
Relations (i) imply that $(\delta_1\delta_2)^{q-1}=1$, which supplies the 
relation omitted from (iii).

By (i) we may define the elements $Z_{ij}$ and $\delta_i$ in $G$.
Since $N$ acts transitively, in $\Omega^+(2n,q)$, by conjugation on the set 
of all root elements $\sigma_{ij}(s)$, where $s\ne 0$,
we may define $\sigma_{ij}(s)$ in $G$ by the same conjugation, by (ii).  

Recall that $\SO^+(4,q)$ is isomorphic to $\SL(2,q)\wr C_2$.
Here $\langle \sigma_{12}, \delta_1\delta_{2}^{-1}, U\rangle$ is 
one copy of $\SL(2,q)$, and 
$\langle \sigma_{1,-2}, \delta_1\delta_2, Z_{12}U\rangle$ is another.  
The first of these subgroups in $G$ is isomorphic to $\SL(2,q)$ by (iii), 
and, since $Z_{12}U=U^{Z_{23}}$, so is the second, by conjugation with $Z_{23}$.  
By (iv)(e), (i) and (ii), these subgroups commute with each other.  
It follows that, in $G$, the subgroup 
$\langle\sigma_{\pm1,\pm2}(1),\delta, Z, Z^V, U\rangle$ is 
isomorphic to $\SO^+(4,q)$, where $Z^V$ 
acts as a wreathing element.  
Thus all Steinberg relations that involve only two suffices, 
with various signs, hold in $G$.

The parenthetical Steinberg relations under each of the headings (a) to (d) all follow 
from one example, as in the case of odd $q$, the suffices all being positive.

All remaining Steinberg relations involving negative suffices may be deduced as 
in the case of odd $q$. 

The proof that $G$ is generated by the Steinberg generators follows as before. 
\end{proof}

\noindent 
Recall that $\Omega^+(2n,q)$ is simple for even $q$.

\subsubsection{A presentation for $\Omega^+(2n,2)$ for $n \ge 3$}
\begin{theorem}
Let $n\ge 3$.  Let $G$ be the group generated by 
$\{\sigma,Z,U,V\}$ subject to the relations given below.  
Then $G$ is isomorphic to $\Omega^+(2n,2)$.  
\begin{enumerate}
\item[(i)] Defining relations for $N_1=\langle Z,U,V\rangle$ as in 
Theorem $\ref{OmegaplusNoneeven}$.

\item[(ii)] Relations that state that the elements listed in
Theorem $\ref{CentOPlus-sigma}$ centralise $\sigma$, but 
omitting those involving $\delta$.

\item[(iii)] Relations that present $\SL(2,2)$ on $\{\sigma,U\}$. 

\item[(iv)] The Steinberg relations (iv) of Theorem $\ref{Opluseven}$.

\end{enumerate}
\end{theorem}
\begin{proof}
By (i) we may define the elements $Z_{ij}$ in $G$.
Since $N_1$ acts transitively, in $\Omega^+(2n,2)$, by conjugation
on the set of $2n(n-1)$ root elements $\sigma_{\pm i,\pm j}(1)$,
we may define $\sigma_{ij}(1)$ in $G$ by the same conjugation, by (ii).  

As in the case $e>1$, 
$\langle\sigma_{12}(1),U,\sigma_{1,-2}(1),Z\rangle$ 
is isomorphic in $G$ to $\Omega^+(4,2)$.

The parenthetical Steinberg relations, under each of the headings (a) to (e), 
all follow from one example,
by conjugation. We complete the proof arguing as in the case $e>1$.
\end{proof}
\noindent 
Observe that $\delta = 1$.

\subsubsection{Standard generators for $\Omega^+(2n, 2^e)$}
In Table 1 of \cite{even} the non-trivial standard generators for 
$\Omega^+(2n, 2^e)$ are labelled $s, t, \delta, u, x, y, v$.
Observe that $s = ZU; t = (\sigma^{-1})^{Z^V}; 
u = U; x = \sigma; 
y = (\delta \delta^{ZU})^{-1}; v = V$; and 
the standard generator $\delta = \delta^{-1} (\delta^{-1})^U$.
The presentation generators
$\delta = (\delta y)^{(q - 2)/2}$ and $Z = s u$.


\section{A presentation for $\Omega(2n+1,q)$} 
We may assume that $q$ is odd, since $\Omega(2n+1,q)$ is isomorphic to $\Sp(2n,q)$ if
$q$ is even.  Since $\Omega(3,q)$ is isomorphic to $\PSL(2,q)$ for odd $q$, 
we assume that $n> 1$.

\subsection{Generators and notation}
Let $q = p^e$ for an odd prime $p$. Let $\omega$ be a primitive element of $\GF(q)$.  

Let $n>1$.  
We take a basis $(e_1,f_1,e_2,f_2,\ldots,e_n,f_n,w)$ of the natural 
module, with symmetric bilinear form defined by 
$e_i.f_i=f_i.e_i=1$ for all $i$, 
and $w.w=-2$, and the form vanishes on all other  
pairs of basis vectors.  
Note that, with our standard generators, we assume that $w.w=-1/2$;  
we have made this change for compatibility with our 
treatment of $\Omega^-(2n,q)$ for odd $q$.  

We define the following elements of $\SO(2n+1,q)$, 
where $1\le i,j\le n$ and $i\ne j$, and $s\in\GF(q)$.

$\delta_i=(e_i\mapsto\omega^{-1}e_i, f_i\mapsto\omega f_i)$ if $e>1$;

$\Delta_{ij}=\delta_i\delta_j^{-1}$;

$\Delta_{-i,j}=\delta_i^{-1}\delta_j^{-1}$;

$\Delta_{i,-j}=\delta_i\delta_j$;

$\Delta_{-i,-j}=\delta_i^{-1}\delta_j$;

$z_i=(e_i,f_i)(w)^-$;


$\sigma_{ij}(s)=(e_i\mapsto e_i+se_j, f_j\mapsto f_j-sf_i)$;

$\sigma_{i,-j}(s)=(e_i\mapsto e_i+sf_j, e_j\mapsto e_j-sf_i)=\sigma_{ij}(s)^{z_j}$;

$\sigma_{-i,j}(s)=(f_i\mapsto f_i+se_j, f_j\mapsto f_j-se_i)=\sigma_{ij}(s)^{z_i}$;

$\sigma_{-i,-j}(s)=\sigma_{ji}(-s)=\sigma_{ij}(s)^{z_iz_j}$;

$\tau_i(s)=(e_i\mapsto e_i+s^2 f_i+sw, w\mapsto w + 2s f_i)$;  

$\tau_{-i}(s)=(f_i\mapsto f_i + s^2 e_i-sw, w\mapsto w - 2s e_i)  = \tau_{i}(s)^{z_i}$;

$U=(e_1,e_2)(f_1,f_2)$;

$U'=(e_1,e_2)^-(f_1,f_2)^-$;

$V=(e_1,e_2,\ldots,e_n)(f_1,f_2,\ldots,f_n)$;

$V'=(e_1,e_2,\ldots,e_n)^{\epsilon_n}(f_1,f_2,\ldots,f_n)^{\epsilon_n}$.

\smallskip
All these elements of $\SO(2n+1,q)$ centralise those basis elements 
that they are not stated to move.  
All lie in $\Omega(2n+1,q)$ except for $\delta_i$, and $U$ if
$q\equiv 3 \bmod 4$, and $V$ if $n$ is even and $q\equiv 3\bmod 4$.  

If $n > 2$ then 
$\langle \sigma_{12}(1), \Delta_{12}, z_1z_1^{U'}, U',V'\}$ acts as the natural copy of 
$\Omega^+(2n,q)$ on $\langle e_1,f_1,\ldots,e_n,f_n\rangle$. 

\subsection{A presentation for $\Omega(2n + 1, q)$ for $n \ge 2$ and $e > 1$}
We give a presentation for $\Omega(2n+1, q)$ on the generating set 
$\{ \sigma=\sigma_{12}(1), \tau=\tau_1(1), \Delta=\Delta_{12}, z=z_1, U',V'\}$, 
omitting $V'$ if $n=2$.

Let $N_1$ be the subgroup of $\Omega(2n+1,q)$ generated by 
$\{ z,U',V'\}$, omitting $V'$ if $n=2$.  
This is an extension
of an elementary abelian $2$-group $A$ of rank $2n-1$ by a copy of $S_n$.  Here $A$ is
the direct product of a subgroup of rank $n-1$, 
generated by $U'^2$ and its conjugates, with 
a second subgroup of rank $n$ generated by $z$ and its conjugates.

\begin{theorem}\label{OoddoddN1}
Let $n\ge 2$.  Let $G$ be the group generated by 
$\{z,U',V'\}$ subject to the relations given below.  
Then $G$ is isomorphic to $N_1$.
\begin{enumerate}
\item[(i)] Defining relations for $\langle U',V'\rangle$
as in Theorems $\ref{signedsymmetricodd}$ and $\ref{signedsymmetriceven}$,
omitting $V'$ if $n=2$.
\item[(ii)] If $n>2$ then $[z,U'^{V'}]=1$.
\item[(iii)] If $n>3$ then $[z,V'U'^{-1}]=1$.
\item[(iv)] $z^2=1$.
\item[(v)] $[z, U'^2]=1$.
\item[(vi)] $[z,z^{U'}]=1$.
\end{enumerate}
\end{theorem}

\begin{proof}
Relations (ii) and (iii) state that $z$ is centralised by the 
subgroup of $\langle U',V'\rangle$
that centralises 1 as signed permutations, since $U'^{V'}$ and $V'U'^{-1}$ stand for
$(2,3)^-$ and $(2,3,\ldots,n)^{-\epsilon_n}$ respectively as signed permutations.

These relations, together with (v), also
imply that $z$ is centralised by all $\langle U',V'\rangle$-conjugates of $U'^2$, 
so we may label the conjugates of $z$ as $z_{i}$, where $1\le i\le n$.  
Relation (vi) states that $z=z_1$ commutes with $z_2$; hence, by conjugation, 
all $z_i$ commute and,  by (iv), generate an elementary abelian $2$-group.  
Thus $G$ is a split extension of a normal subgroup of order at most $2^n$ by the group 
$\langle U',V'\rangle$ of order $2^{n-1}n!$, by (i).
But $N_1$, as a subgroup of $\Omega(2n+1,q)$, has order $2^{2n-1}n!$, and the result follows.
\end{proof}

Let $N=\langle \Delta,z,U',V'\rangle \leq \Omega(2n+1,q)$, 
omitting $V'$ if $n=2$.  
As with $\Omega^+(2n,q)$, the normal subgroup of $N$ generated as 
normal subgroup by $\Delta$, and as subgroup by 
$\{\Delta_{\pm i,\pm j} \, | \, |i|\ne |j|\}$, has order $(q-1)^n/2$.
Since this normal subgroup contains $U'^2$, 
it follows that $N$ has order $(q-1)^n2^{n-1}n!$.

\begin{theorem}\label{OpoddoddN}
Let $q = p^e$ for an odd prime $p$ and $e > 1$, and let $n \ge 2$. 
Let $G$ be the group generated by
$\{\Delta, z,U',V'\}$, omitting $V'$ if $n=2$, subject to the relations given below.
Then $G$ is isomorphic to $N$.
\begin{enumerate}
\item[(i)] Defining relations for $N_1 = \langle z,U',V'\rangle$
as in Theorem $\ref{OoddoddN1}$, omitting $V'$ if $n=2$. 
\item[(ii)] If $n>3$ then $[\Delta,U'^{V'^2}]=1$.
\item[(iii)] If $n>4$ then $[\Delta,V'U'U'^{V'}]=1$.
\item[(iv)] If $n>2$ then $[\Delta,z^{V'^2}]=1$.
\item[(v)] $\Delta^{U'} = \Delta^{-1}$, and if $n>2$ then $[\Delta,U'^{2V}]=1$.
\item[(vi)] $\Delta^{(q-1)/2}=U'^2$.
\item[(vii)] If $n>2$ then $\Delta\Delta^{V'}=\Delta^{V'U'}$.
\item[(viii)] If $n>2$ then $[\Delta,\Delta^{V'}]=1$.
\item[(ix)] If $n>3$ then $[\Delta,\Delta^{V'^2}]=1$.
\item[(x)] $\Delta^{zz^{U'}}=\Delta^{-1}$.
\item[(xi)] If $n=2$ then $[\Delta,\Delta^z]=1$.
\end{enumerate}
\end{theorem}
\begin{proof}
Suppose first that $n \geq 3$.
This presentation is very similar to that for $N$ in Theorem \ref{OmegaplusoddN}.
The subgroups $N$ and $N_1$ of $\Omega^+(2n,q)$ have index 2 in the corresponding
subgroups of $\Omega(2n+1,q)$, the generator $Z=Z_{12}$ being replaced by $z=z_1$.
The presentation for $N$ given in this theorem differs from the presentation in Theorem
\ref{OmegaplusoddN} in items (i) and (iv).  Item (i) gives a presentation for  
the enlarged version of $N_1$
that results from replacing $Z$ by $z$.  
Item (iv) asserts that $\Delta$ is centralised by $z_3$, and hence, 
by conjugation, by $z_3,z_4,\ldots,z_n$, rather than by 
$Z_{34}, Z_{45},\ldots,Z_{n-1,n}$ in the case
of $\Omega^+(2n,q)$.  This reflects the fact that the 
centraliser of $\Delta$ in $N_1$ is twice as large
in the case of $\Omega(2n+1,q)$, 
since both groups $N_1$ conjugate transitively 
the same set of elements $\Delta_{ij}$.  
The proof now follows that of Theorem \ref{OmegaplusoddN}. 

The proof for $n=2$ is easy.
\end{proof}

The centraliser of $\sigma=\sigma_{12}(1)$ in $N$ is the direct product of the
subgroup of $N$ that centralises $\langle e_1,f_1,e_2,f_2\rangle$
with the copy of $Q_{2(q-1)}$ generated by 
$\{\Delta_{1,-2}, U'^z\}$.   
If $n = 2$ then the centraliser is $Q_{2(q-1)}$. 
If $n > 2$ then the centraliser 
has order $(q-1)^{n-1}2^{n-2}(n-2)!$, 
and index $2(q-1)n(n-1)$ in $N$; 
this corresponds to the fact that $N$ acts transitively 
by conjugation on the set of elements
$\sigma_{\pm i, \pm j}(s)$, where $1\le i\ne j\le n$ 
and $s\in\GF(q)^{\times}$, and these are equal in pairs.

\begin{theorem} \label{Omega2n+1-tau-cent}
The centraliser of $\tau=\tau_1(1)$ in $N$ 
has index $2(q-1)n$ in $N$ and 
is generated by the following elements:
\begin{enumerate}
\item[(i)] $U'^{V'}$ if $n>2$; 
\item[(ii)] $V'U'^{-1}$ if $n>3$; 
\item[(iii)] $\Delta^{V'}$ if $n>2$; 
\item[(iv)] $U'^2z^{U'}$; 
\item[(v)] $\Delta\Delta^z$ if $n=2$.
\end{enumerate}
\end{theorem}

\begin{proof}
The centraliser of $\tau$ in $N$ is the direct product of the subgroup of $N$ that centralises 
$\langle e_1,f_1,w\rangle$ with $\langle U'^2z_2\rangle$, and so has order 
$(q-1)^{n-1}2^{n-2}(n-1)!$, and index $2(q-1)n$ in $N$.  This corresponds to the fact that
$N$ acts transitively by conjugation on the $2(q-1)n$ root 
elements $\tau_{\pm i}(s)$, where $s \in \GF(q)^{\times}$.  
\end{proof}

\begin{theorem}\label{omegaodd}
Let $q = p^e $ for an odd prime $p$ and $e > 1$, and let $n \ge 2$. 
Let $G$ be the group generated by 
$\{\sigma, \tau, \Delta,z,U',V'\}$, omitting $V'$ if $n=2$, 
subject to the relations given below.  
Then $G$ is isomorphic to $\Omega(2n+1,q)$.  
\begin{enumerate}
\item[(i)] Defining relations for $N=\langle\Delta,z,U',V'\rangle$, 
omitting $V$ if $n = 2$, as in Theorem $\ref{OpoddoddN}$, 
but omitting relation (vi) of that theorem.

\item[(ii)] $[\sigma,\Delta^{z^{U'}}]=1$.

\item[(iii)] Relations that state that the elements listed in
Theorem $\ref{Omega2n+1-tau-cent}$ centralise $\tau$.

\item[(iv)] Relations that present $\SL(2,q)$ on $\{\sigma,\Delta,U'\}$.

\item[(v)] Relations that present $\PSL(2,q)$ on $\{\tau,\Delta\Delta^{z^{U'}},z\}$.

\item[(vi)] The following instances of Steinberg relations:
\begin{enumerate}
\item[(a)] if $n>2$ then $[\sigma,\sigma^{V'}]=\sigma^{V'U'^{-1}}$; 
\q$([\sigma_{ij}(s),\sigma_{jk}(t)]=\sigma_{ik}(st))$;
\item[(b)] if $n>2$ then $[\sigma,\sigma^{V'U'^{-1}}]=1$; \q$([\sigma_{ij}(s),\sigma_{ik}(t)]=1)$;
\item[(c)] if $n>2$ then $[\sigma,\sigma^W]=1$, where $W=U'^{V'U'^{-1}}$; \q$([\sigma_{ij}(s),\sigma_{kj}(t)]=1)$;
\item[(d)] if $n>3$ then $[\sigma,\sigma^{V'^2}]=1$; \q$([\sigma_{ij}(s),\sigma_{kl}(t)]=1)$;
\item[(e)] $[\sigma,\sigma^{z^{U'}}]=1$; \q$([\sigma_{ij}(s),\sigma_{i,-j}(t)]=1)$;
\item[(f)] $[\tau,\tau^{U'}]=\sigma^{2z^{U'}}$; \q$([\tau_i(s),\tau_j(t)]=\sigma_{i,-j}(2st))$;
\item[(g)] $[\sigma,\tau]=1$; \q$([\sigma_{ij}(s),\tau_i(t)]=1)$;
\item[(h)] $[\sigma^z,\tau]=\sigma\tau^{zU'}$;\q$([\sigma_{-i,j}(s),\tau_i(t)]=\sigma_{ij}(st^2)\tau_{-j}(st))$;
\item[(j)] if $n>2$ then $[\sigma, \tau^{{V'}^2}]=1$; \q$([\sigma_{ij}(s), \tau_k(t)]=1)$.  
\end{enumerate}
The given relations are instances of the more general relations in parentheses, 
with $s=t=1$, and $(i,j,k,l)=(1,2,3,4)$.  
In the parenthetical version, $i,j,k,l$ are distinct and positive.
\end{enumerate}
\end{theorem}
\begin{proof}
Relation (vi) of Theorem \ref{OpoddoddN}, omitted from 
 (i) of this theorem, is provided by (iv).
Since, in the matrix group, $N$ acts transitively 
by conjugation on $\{\tau_{\pm i}(s) \, | \,
1\le i\le n, s \in\GF(q)^{\times}\}$, we may define the 
elements $\tau_{\pm i}(s)$ in $G$, using (i) and (iii).

It follows from (f) that $\sigma^2= [\tau_1(-1),\tau_{-2}(1)]$.
So $\sigma^2$ is centralised by the subgroup $A$ of $N$ that
centralises $\langle e_1,f_1,e_2,f_2\rangle$, and also by $U'^z$.
Also, $\sigma$ has odd order by (iv); so, 
since $\Delta_{1,-2}=\Delta^{z^{U'}}$,  
relations (ii) and (iv) complete a generating set for the centraliser of $\sigma$ in $N$.  
Thus we may define the elements $\sigma_{ij}(s)$ 
for all $i$ and $j$, positive or negative, in $G$.  

The Steinberg relations for the
individual root groups may be deduced from (iv) and (v).  
We must show that the Steinberg relations that give the commutator of two long root elements, 
that is, of two elements of the form $\sigma_{ij}(s)$, are all consequences of (a) to (e).
These relations correspond precisely to the Steinberg relations (iv)(a)--(iv)(e) of 
Theorem \ref{OmegaPlusOdd}, except for an insignificant change to (iv)(e).  
Since the torus normaliser $N$ in
the case of $\Omega(2n+1,q)$ contains the corresponding torus normaliser in $\Omega^+(2n,q)$
as a subgroup of index 2, the fact that these five instances of 
Steinberg relations in $\Omega(2n+1,q)$ suffice to
imply all Steinberg relations that give the commutator of two long root elements follows
{\it a fortiori} from Theorem \ref{OmegaPlusOdd}.
Clearly the Steinberg relations that give the commutator of two short root elements all
follow, by conjugation in $N$, from the one instance (f).
Similarly the Steinberg relations that give the commutator of a long and a short root element 
may be deduced from the instances (g), (h) and (j). 

Thus the subgroup $H$ of $G$ generated by 
the Steinberg generators is isomorphic to $\Omega(2n+1,q)$;
that $U'$ and $V'$ lie in $H$, which is thus the whole of $G$, follows by
the now familiar argument.
\end{proof}
\noindent 
Recall that $\Omega^+(2n+1,q)$ is simple.  

\subsection{A presentation for $\Omega(2n+1,p)$ for $n \ge 2$}  
\begin{theorem}
Let $p$ be an odd prime and let $n\ge 2$.  Let $G$ be the group generated by 
$\{\sigma, \tau, z,U',V'\}$, 
omitting $V'$ if $n=2$,
subject to the relations given below.  
Then $G$ is isomorphic to $\Omega(2n+1,p)$.  
\begin{enumerate}
\item[(i)] Defining relations for $N_1=\langle z,U',V'\rangle$
as in Theorem $\ref{OoddoddN1}$, omitting $V'$ if $n = 2$. 
\item[(ii)] Relations that state that the elements listed in
Theorem $\ref{Omega2n+1-tau-cent}$ centralise $\tau$, but omitting those involving $\Delta$.
\item[(iii)] Relations that present $\SL(2,p)$ on $\{\sigma,U'\}$.
\item[(iv)] Relations that present $\PSL(2,p)$ on $\{\tau,z\}$.
\item[(v)] The Steinberg relations (vi) of Theorem $\ref{omegaodd}$.
\item[(vi)] If $n=p=3$ then $[\tau,\sigma^{V'}]=1$.
\end{enumerate}
\end{theorem}
\begin{proof}
By (i) we may define the elements $z_i$ in $G$.

If $n > 2$ then 
$N_1$ acts transitively by conjugation on 
$\{\sigma_{\pm i,\pm j}(\pm1) \, | \, 1\le i\ne j\le n\}$; 
the $8n(n-1)$ symbols are equal, as matrices, in pairs.  
Since $N_1$ has order $2^{2n-1}n!$, 
the centraliser in $N_1$ of $\sigma$ has order $2^{2n-3}(n-2)!$.
Thus this centraliser is generated by the subgroup of $N_1$ that centralises 
$\langle e_1,f_1,e_2,f_2\rangle$, which has order $2^{2n-5}(n-2)!$ as $n>2$, 
together with $U'^2$ and $U'^{z_1}$. 
It follows, from item (ii) of this theorem and the Steinberg relation (vi)(f)
of Theorem $\ref{omegaodd}$, that these elements of $N_1$ centralise $\sigma$ in
$G$.  Thus the elements $\sigma_{ij}(\pm 1)$ may be defined in $G$ for all $i$ and $j$,
positive or negative, and hence all elements $\sigma_{ij}(s)$ for $s$ in $\GF(p)$.

Now let $n=2$.  By (ii), we may define the elements $\tau_i(s)$ in $G$.
Moreover $N_1$, as matrices, permutes the $8$ elements 
$\{\sigma_{\pm i,\pm j}(\pm1) \, | \, 1\le i\ne j\le 2\}$
in 2 orbits of size $4$.  Since $N_1$ has order 16, the centraliser of 
$\sigma$ in $N_1$ has order 4 and is 
generated by $\{U'^2, U'^{z_1}\}$.  
As in the case $n>2$ these elements centralise $\sigma$ as elements of $G$, and
so we may now define the elements $\sigma_{ij}(s)$ in this group.

The general cases of the Steinberg relations now follow easily from
the cases given in the presentation.

We complete the proof arguing as in the case $e>1$.

Relation (vi) is needed because the Steinberg presentation gives the 
universal Chevalley group, and
$\Omega(7,3)$ has a Schur multiplier of order 6.  
The involution in the multiplier is killed by our other relations,
and (vi) is needed to kill the remaining 3-cycle.
We prove correctness of the presentation for this group by coset enumeration.
\end{proof}
\noindent 
Observe that 
$\Delta = {({\sigma^{\omega - \omega^2}})}^{U'}  \sigma^{\omega^{-1}} 
          {(\sigma^{\omega- 1})}^{U'} \sigma^{-1}$.

\subsection{Standard generators for $\Omega(2n+1,q)$} 
In Table 2 of \cite{odd} the standard generators for 
$\Omega(2n + 1, q)$ are labelled $s, t, \delta, u, v$.
Observe that $s = z^{V'^{-1}}; 
\delta = [z,\Delta^{-1}]^{V'^{-1}}; u = U'; v = V'$;  
if $d \equiv 3 \bmod 4$ then  $t = (\tau)^{V'^{-1}}$,
else $t = (\tau^{-1})^{V'^{-1}}$.
 
The presentation generator $\sigma=[\tau^{zU'},\tau^{(p+1)/2}]$.
If $q\equiv3\bmod4$ then 
$\Delta={U'}^2([\delta^{V'},U'])^{(q+1)/4}$.
Also $\Delta^2=[\delta^{V'},U']$.
If $q \ne 9$, then 
$\Delta \in 
\langle\Delta^2,\sigma,U'\rangle \cong \SL(2,q)$. 
In particular, 
$$\Delta = 
\prod_{i=0}^{e-1} \sigma^{c_i\Delta^{2i}U'} 
\Delta^2 
\sigma^{U'}  
\prod_{i=0}^{e-1} \sigma^{b_i\Delta^{2i} \sigma^{U'}} 
\prod_{i=0}^{e-1}\sigma^{a_i\Delta^{2i}}$$
where $\sum_{i=0}^{e-1} a_i \omega^{4i} = -\omega^{-1} + 1$, and 
$\sum_{i=0}^{e-1} b_i \omega^{4i} = 1 - \omega$, 
and $\sum_{i=0}^{e-1} c_i \omega^{4i} = -\omega^3$.
If $q=9$, then $\Delta\in\langle \Delta^2,z,U',\tau\rangle$, and this 
group is isomorphic to $\PSp(4,9)$, the natural module for $\Omega(5,9)$ being 
isomorphic to an irreducible constituent of the exterior square
of the natural module for $\Sp(4,9)$;
we record $\Delta$ as an explicit fixed word in these generators. 

\section{A presentation for $\Omega^-(2n,q)$ for $n > 1$}
Since $\Omega^-(4,q) \cong \PSL(2, q^2)$, we assume that $n > 2$.

\subsection{Odd characteristic}
\subsubsection{Generators and notation}
Let $q = p^e$ for an odd prime $p$. 
Let $\omega$ be a primitive element of $\GF(q^2)$, and let $\omega_0=\omega^{q+1}$, 
so $\omega_0$ is a primitive element of $\GF(q)$. Define $\psi = \omega^{(q+1)/2}$.   

Let $n>2$.  
We take a basis $(e_1,f_1,e_2,f_2,\ldots,e_{n-1},f_{n-1},w_1,w_2)$ of 
the natural module, with orthogonal bilinear
form defined by $e_i.f_i=f_i.e_i=1$, $w_1.w_1=-2$, $w_2.w_2=2\omega_0$, 
and the form vanishes on all other pairs
of basis elements.  This form is of type $-$.

We define the following elements of $\SO^-(2n,q)$, 
where $1\le i,j\le n-1$ and $i\ne j$, and $s\in\GF(q)$,
and $\alpha \in \GF(q^2)$.  

$\delta_i=(e_i\mapsto\omega_0^{-1}e_i, f_i\mapsto\omega_0 f_i, 
w_1\mapsto Aw_1-Cw_2, w_2\mapsto-Bw_1+Aw_2)$, where
 \begin{eqnarray*}
A & = & (\omega^{q-1}+\omega^{1-q})/2 \\
B & = & \psi(\omega^{1-q}-\omega^{q-1})/2 \\
C & = & \psi^{-1}(\omega^{1-q}-\omega^{q-1})/2;
\end{eqnarray*}

$\Delta_{i,-j}=\delta_i\delta_j^{-1}$;

$z_i=(e_i,f_i)(w_1)^-$;

$\sigma_{ij}(s)=(e_i\mapsto e_i+se_j, f_j\mapsto f_j-sf_i)$;

$\sigma_{i,-j}(s)=(e_i\mapsto e_i+sf_j, e_j\mapsto e_j-sf_i)=\sigma_{ij}^{z_j}$;

$\sigma_{-i,j}(s)=(f_i\mapsto f_i+se_j, f_j\mapsto f_j-se_i)=\sigma_{ij}^{z_i}$;

$\sigma_{-i,-j}(s)=\sigma_{ji}(-s)=\sigma_{ij}(s)^{z_iz_j}$;

$\tau_i(\alpha)=(e_i\mapsto e_i+\alpha^{q+1}f_i+\frac{1}{2}(\alpha^q+\alpha)w_1 + 
\frac{1}{2}\psi^{-1}(\alpha^q-\alpha)w_2, \\ \hspace*{2.15cm}
w_1\mapsto w_1+(\alpha+\alpha^q)f_i,\, 
w_2\mapsto w_2+\psi(\alpha-\alpha^q)f_i)$; 

$\tau_{-i}(\alpha)=(f_i\mapsto f_i+\alpha^{q+1}e_i-\frac{1}{2}(\alpha^q+\alpha)w_1 + 
\frac{1}{2}\psi^{-1}(\alpha^q-\alpha)w_2,\\ \hspace*{2.15cm}
w_1\mapsto w_1-(\alpha+\alpha^q)e_i,  \,
w_2\mapsto w_2+\psi(\alpha-\alpha^q)e_i)=\tau_i(\alpha)^{z_i}$; 

$U=(e_1,e_2)(f_1,f_2)$;

$U'=(e_1,e_2)^-(f_1,f_2)^-$;

$V=(e_1,e_2,\ldots,e_{n-1})(f_1,f_2,\ldots,f_{n-1})$;

$V'=(e_1,e_2,\ldots,e_{n-1})^{\epsilon_{n-1}}(f_1,f_2,\ldots,f_{n-1})^{\epsilon_{n-1}}$.

\smallskip
All these elements of $\SO^-(2n,q)$ centralise those basis elements 
that they are not stated to move.  
All lie in $\Omega^-(2n,q)$ except for $U$ if
$q\equiv 3 \bmod 4$, and $V$ if $n$ is odd and $q\equiv 3\bmod 4$.  

All those elements of $\Omega^-(2n,q)$ that fix $w_2$ are defined as 
the corresponding elements of $\Omega(2n-1,q)$.   
In particular, $\tau_i(s)$ for $s\in\GF(q)$, and its conjugates 
under $\delta_i\delta_j^{-1}$ for $i \ne j$, 
are the same in $\Omega(2n-1,q)$ and in $\Omega^-(2n,q)$.   

We give a presentation for $\Omega^-(2n, q)$ on the generating set 
$\{ \sigma=\sigma_{12}(1), \tau=\tau_1(1), \delta=\delta_1, z=z_1, U',V'\}$, 
omitting $V'$ if $n=3$.

\subsubsection{A presentation for $\Omega^-(6, q)$}
We exploit the isomorphism 
between $\Omega^-(6,q)$ and $\SU(4,q)/\langle-I_4\rangle$.
Let $W$ be the natural module for $\SU(4,q)$. 
Take a hyperbolic basis $(u_1,v_1,u_2,v_2)$ for $W$ and 
define a basis $(e_1,f_1,e_2,f_2,w_1,w_2)$ for $W\wedge W$ thus:
$$e_1=u_1\wedge v_2, f_1=v_1\wedge u_2, 
e_2=-\psi(v_1\wedge v_2), f_2=\psi^{-1}(u_1\wedge u_2)$$ 
$$w_1=u_1\wedge v_1-u_2\wedge v_2, w_2=\psi(u_1\wedge v_1+u_2\wedge v_2).$$ 
The natural action of $\SU(4,q)/\langle -I_4\rangle$ on $W\wedge W$ 
coincides with the natural action of $\Omega^-(6,q)$,
with the basis defined above, under the isomorphism defined below.  
The kernel of this homomorphism is, of course, $\langle-I_4\rangle$.  Equating an element of 
$\SU(4,q)$ with its image modulo $\langle -I_4\rangle$, and using our notation for 
elements of $\SU(4,q)$,
this homomorphism maps our chosen generators for $\SU(4,q)$ to $\Omega^-(6,q)$ as follows:
$$\sigma\mapsto\tau^{-1},\; \tau\mapsto\sigma,\; Z\mapsto U',\;
\delta\mapsto\delta\delta^{-U'},\;\Delta\mapsto\delta,\;U\mapsto z{U'}^2.$$

\begin{theorem}\label{Omega6minus}
If $q$ is odd, then a presentation for $\Omega^-(6,q)$ on our chosen generators is given 
by taking the presentation 
for $\SU(4,q)/\langle-I_4\rangle$ in Corollary $\ref{SU(4,q)/-I4}$
with the above substitutions.  
\end{theorem}

\subsubsection{A presentation for $\Omega^-(2n,q)$ for $n \ge 4$}
Let $N_1$ be the subgroup of $\Omega^-(2n,q)$ generated by $\{ z,U',V'\}$.  
This is the same generating set as used in Theorem \ref{OoddoddN1} 
for a subgroup of $\Omega(2n-1,q)$, so the same presentation holds.




Let $N$ be the subgroup of $\Omega^-(2n,q)$ generated by  
$\{ \delta,z,U',V'\}$. 
Now $\delta$ and its conjugates
generate a subgroup of $N$ that is the direct product of the 
cyclic group $\langle\delta\rangle$ 
of order $(q^2-1)/2$ with the direct product of $n-2$ cyclic groups of 
order $q-1$ generated by 
$\delta_1\delta_2^{-1},\delta_2\delta_3^{-1},\ldots,\delta_{n-2}\delta_{n-1}^{-1}$.  This
subgroup contains 
$U'^2=(\delta_1\delta_2^{-1})^{(q-1)/2}$.  Thus $N$ has order $(q+1)(q-1)^{n-1}2^{n-2}(n-1)!$, 
where a factor of $2^{n-1}$ arises from the $z_i$.  

\begin{theorem}\label{OminusoddN}
Let $n \ge 4$. Let $G$ be the group generated by
$\{\delta, z, U',V'\}$ subject to the relations given below. Then $G$ is isomorphic to $N$.
\begin{enumerate}
\item[(i)] Defining relations for $N_1 = \langle z,U',V'\rangle$
as in Theorem $\ref{OoddoddN1}$. 
\item[(ii)] $[\delta,U'^{V'}]=1$.
\item[(iii)] If $n>4$ then $[\delta,V'U'^{-1}]=1$.  
\item[(iv)] $U'^2=(\delta\delta^{-U'})^{(q-1)/2}$. 
\item[(v)] $[\delta,z^{U'}]=\delta^{q-1}$.
\item[(vi)] $\delta^{(q^2-1)/2}=1$.
\item[(vii)] $[\delta,\delta^{U'}]=1$.
\item[(viii)] $\delta^z=\delta^{-1}$.
\item[(ix)] $[\delta^{q-1},U'] = 1$.
\end{enumerate}
\end{theorem}
\begin{proof}  
By (i) we may take $G$ to contain $N_1$.  By (iv) and (vii) it follows 
that $U'^2$
commutes with $\delta$, so, using (ii) and (iii), we may assume that
$\delta_i\in G$
for all $i$, and that $\langle U',V'\rangle$ permutes these elements as unsigned
permutations of the suffices.  By (vii) it then follows that the $\delta_i$ all commute.
The actions of $z_1$ and of $z_2$ on $\delta_1$ are given by (v) and (viii), and
this gives the action of $z_i$ on $\delta_j$ for all $i$ and $j$.
If $n>4$ it follows from (iii) and (ix), or from (ii) and (ix) if $n=4$, 
that $[\delta^{q-1},V']=1$.
Hence $\delta_i^{q-1}=\delta_j^{q-1}$ for all $i$ and $j$.
Thus the order of the normal subgroup of $G$ generated by the $\delta_i$
has order dividing $(q-1)^{n-1}(q+1)/2$.  But $N_1$ has order $2^{2n-3}(n-1)!$,
and contains a normal subgroup of order $2^{n-2}$, namely the subgroup generated
by $U'^2$ and its conjugates, that is contained in the subgroup generated by the $\delta_i$.
So $G$ has order at most $(q+1)(q-1)^{n-1}2^{n-2}(n-1)!$, and as this is the order of $N$ 
the result follows.
\end{proof}

\begin{theorem}\label{OmegaMinusOdd-cent-sigma}
The centraliser of $\sigma = \sigma_{12}(1)$ in $N$  
has index $2(n-1)(n-2)(q-1)$ in $N$ and is generated by the following elements:
\begin{enumerate}
\item[(i)] $U'^{{V'}^2}$ if $n>4$; 
\item[(ii)] $V'U'^{-1}U'^{-V'}$ if $n>5$; 
\item[(iii)] $z^{{V'}^2}$;
\item[(iv)] $\delta^{{V'}^2}$;
\item[(v)] $\delta\delta^{U'}$;
\item[(vi)] $zz^{U'}U'$.
\end{enumerate}
\end{theorem}

\begin{proof}  
The first claim follows from the fact that $N$ permutes transitively the 
$4(n-1)(n-2)(q-1)$ elements
$\sigma_{ij}(s)$, 
for $1\le |i|\ne|j|\le n-1$ and $s\in\GF(q)^{\times}$, which are equal in pairs.
The listed elements all centralise $\sigma$, and generate a group $G$
that contains a normal subgroup $K$ of order $(q+1)/2$ generated by $\delta^{q-1}$.
Observe that $G/K$ is isomorphic to the direct product of a copy of $D_{2(q-1)}\wr S_{n-3}$ 
generated by the images of the elements (i) to (iv) with a copy of $Q_{2(q-1)}$
generated by the images of the elements (v) and (vi).  
So $G$ has order $(q+1)(q-1)^{n-2}2^{n-3}(n-3)!$, 
which is the order of the centraliser.
\end{proof}

Since $\tau_i(\alpha)^{\delta_i}=\tau_i(\alpha\omega^2)$ for $i>0$, 
and $\tau_{i}(\alpha)^{\delta_i}=\tau_i(\alpha\omega^{-2})$ for $i<0$, 
it follows that 
$\{\tau_{\pm i}(\alpha) \, | \, 1\le i\le n-1, \alpha\in\GF(q^2)^{\times}\}$ is 
permuted under conjugation by $N$, and falls into two orbits, the elements
for which $\alpha$ is a square forming one orbit, and those for which $\alpha$ is a 
non-square forming the other.

\begin{theorem}\label{OmegaMinusOdd-cent-tau}
The centraliser of $\tau = \tau_1(1)$ in $N$ 
has index $(n-1)(q^2-1)$ in $N$ and 
is generated by the following elements: 
\begin{enumerate}
\item[(i)] $\delta^{(q+1)/2}\delta^{U}$;
\item[(ii)] $U'zU'$;
\item[(iii)] $U'^{V'}$;
\item[(iv)] $V'U'^{-1}$ if $n > 4$.
\end{enumerate}
\end{theorem}
\begin{proof}
These elements all centralise $\tau$.  
The first two generators, which as matrices are
$\diag(-1,-1,\omega_0^{-1},\omega_0,1,1,\ldots,1,-1,-1)$ and 
$(e_1)^-(f_1)^-(e_2,-f_2)(w_1)^-$, lie in $\Omega^-(6,q)$;
they generate a copy of $D_{2(q-1)}$. So
these elements, together with items (iii) and (iv), generate
a copy of $D_{2(q-1)} \wr S_{n-2}$, which has 
the required order ${(2(q-1))}^{n-2}(n-2)!$.
\end{proof}

\begin{theorem} \label{OmegaMinusOdd}
Let $q = p^e$ for an odd prime $p$, and let $n\ge 4$.  
Let $G$ be the group generated by 
$\{\sigma,\tau,\delta,z,U',V'\}$ subject to the relations given below.  
Then $G$ is isomorphic to $\Omega^-(2n,q)$.  
\begin{enumerate}
\item[(i)] Defining relations for $N=\langle\delta,z,U',V'\rangle$ as in 
Theorem $\ref{OminusoddN}$, 
but omitting relations (iv) to (ix) of that theorem, 
relations (iv), (v) and (vi) of Theorem $\ref{OoddoddN1}$, 
and the relation $U'^4=1$ from the presentation for $\langle U',V'\rangle$.

\item[(ii)] Relations that state that the elements listed
in Theorem $\ref{OmegaMinusOdd-cent-sigma}$ centralise $\sigma$, 
omitting relations (v) and (vi).  

\item[(iii)] Relations that state that the 
elements listed in Theorem $\ref{OmegaMinusOdd-cent-tau}$ centralise $\tau$,
omitting relations (i) and (ii).  

\item[(iv)] Relations that present $\Omega^{-}(6,q)$ on $\{ \sigma,\tau,\delta,z,U'\}$ as in 
Theorem $\ref{Omega6minus}$.

\item[(v)] The Steinberg relations (iv)(a) to (iv)(c), and (iv)(d) when $n>4$, 
of Theorem $\ref{OmegaPlusOdd}$.

\item[(vi)] The Steinberg relation $[\sigma^{V'},\tau]=1$.
\end{enumerate}
\end{theorem}
\begin{proof}
The relations omitted from (i), (ii), and (iii) are supplied by (iv).  

The relations (iv)(e) and (iv)(f) of Theorem $\ref{OmegaPlusOdd}$ 
are supplied respectively by (iv) of this theorem, and by (iv)(d) of 
Theorem $\ref{OmegaPlusOdd}$ on conjugation by $z_l$.

All Steinberg relations involving $\sigma_{ij}(s)$ and $\tau_k(\alpha)$, 
when at most two values for
the moduli of the suffices $i,j,k$ arise, may be deduced from (i) to (iv).  

We now consider Steinberg relations whose LHS is of the form 
$[\sigma_{ij}(s),\sigma_{kl}(t)]$, 
where the moduli of the suffices take at least three distinct values.  
Let $H$ denote the subgroup of $\Omega^-(2n,q)$ generated by $\{\sigma,\delta,z,U',V'\}$.  
Observe that $H$ contains $N$ and fixes
$\langle e_1,f_1,\ldots,e_{n-1},f_{n-1}\rangle$, acting on this space 
as ${\Ot}^+(2(n-1),q)$.  Let $N_0$ denote the torus normaliser in 
$\Omega^+(2(n-1),q)$ denoted by $N$ in Section \ref{SectOmegaP}.  
So $N$ contains $N_0$ as a subgroup of index $2(q+1)$, and the 
Steinberg relations (iv) of Theorem \ref{OmegaPlusOdd} 
suffice {\it a fortiori} to imply all the Steinberg relations of the type 
under consideration.  

It remains to consider the Steinberg relations $[\sigma_{ij}(s),\tau_k(\alpha)]=1$, 
where the suffices have distinct moduli.  
These follow from the one instance $[\sigma_{23}(1),\tau_1(1)]=1$ given by (v).

The standard argument now shows that $G$ is generated by the Steinberg generators.
\end{proof}

\subsubsection{A presentation for ${\rm P}\Omega^-(2n,q)$}
The centre of $\Omega^-(2n,q)$ is trivial, unless $q^n\equiv 3\bmod 4$, 
when it has order 2. If it is non-trivial, then it is generated by
$V'^{(n-1)}\delta^{(q^2-1)/4}$. 
For $V'^{(n-1)}=\diag(-1,-1,\ldots,-1,1,1)$ as $n$ is odd, and 
$\delta^{(q^2-1)/4}=\diag(1,1,\ldots,1,-1,-1)$ as $q\equiv 3\bmod 4$.

\subsubsection{Standard generators for $\Omega^-(2n,q)$}
In Table 2 of \cite{odd} the standard generators for 
$\Omega^{-}(2n, q)$ are labelled $s, t, \delta, u, v$.
Observe that $s = z^{V'^{-1}}; \delta = (\delta^{V'^{-1}})^{((q- 1)^2 / 2) - 1}; 
u = U'$; $v = V'$;  
if $d \equiv 0 \bmod 4$ then  
$t = \tau^{V'^{-1}}$, else $t = {(\tau^{-1})}^{V'^{-1}}$.
The presentation generators $\sigma = [\tau,\tau^{zU'}]^{(p-1)/2}$
and $\delta = (\delta^v)^{((q-1)^2/2)-1}$. 

\subsection{Even characteristic}
\subsubsection{Generators and notation}
Let $q = 2^e$. Let $\omega$ be a primitive element of $\GF(q^2)$, and 
let $\omega_0=\omega^{q+1}$, 
so $\omega_0$ is a primitive element of $\GF(q)$.  

Let $n>2$.  
We take a basis $(e_1,f_1,e_2,f_2,\ldots,e_{n-1},f_{n-1},w_1,w_2)$ for 
the natural module, where the associated quadratic form has matrix 
$$\diag\left(\left(\begin{smallmatrix}0&1\cr0&0\end{smallmatrix}\right),\ldots,
\left(\begin{smallmatrix}0&1\cr0&0\end{smallmatrix}\right),
\left(\begin{smallmatrix}1&\omega+\omega^q\cr0&\omega_0\end{smallmatrix}\right)\right),$$ 
and the corresponding bilinear form has matrix 
$$\diag\left(\left(\begin{smallmatrix}0&1\cr1&0\end{smallmatrix}\right),\ldots,
\left(\begin{smallmatrix}0&1\cr1&0\end{smallmatrix}\right),
\left(\begin{smallmatrix}0&\omega+\omega^q\cr\omega+\omega^q&0\end{smallmatrix}\right)\right).$$
This form is of type $-$.

We define the following elements 
of $\Omega^-(2n,q)$, where $1\le i,j\le n-1$ and $i\ne j$, and $s\in\GF(q)$,
and $\alpha \in \GF(q^2)$.

$\delta_i=(e_i\mapsto\omega_0^{-1}e_i, f_i\mapsto\omega_0 f_i, 
           w_1\mapsto Aw_1 + Cw_2, w_2\mapsto Bw_1 + w_2)$, where
\begin{eqnarray*}
A & = & \omega^{q-1} + \omega^{1 -q}+1 \\
B & = & \omega+\omega^q \\
C & = & \omega^{-1}+\omega^{-q}; 
\end{eqnarray*}





$z_i=(e_i\mapsto f_i, f_i\mapsto e_i, w_2\mapsto w_2+Bw_1)$;


$\sigma_{ij}(s)=(e_i\mapsto e_i+se_j, f_j\mapsto f_j+sf_i)$;

$\sigma_{-i,j}(s)=(f_i\mapsto f_i+se_j, f_j\mapsto f_j+se_i)=\sigma_{ij}(s)^{z_i}$;

$\sigma_{i,-j}(s)=(e_i\mapsto e_i+sf_j, e_j\mapsto e_j+sf_i)=\sigma_{ij}(s)^{z_j}$;

$\sigma_{-i,-j}(s)=\sigma_{ji}(s) = \sigma_{ij}(s)^{z_iz_j}$;

$\tau_i(\alpha)=(e_i\mapsto e_i + \alpha^{q+1}f_i +(\alpha^q\omega^q+\alpha\omega)B^{-1}w_1
+(\alpha^q+\alpha)B^{-1}w_2,  \\
\hspace*{2cm} w_1\mapsto w_1 +(\alpha+\alpha^q)f_i, 
w_2\mapsto w_2 +(\omega\alpha+\omega^q\alpha^q)f_i)$;

$\tau_{-i}(\alpha)=(f_i\mapsto f_i+\alpha^{q+1}e_i+(\alpha^q\omega+\alpha\omega^q)B^{-1}w_1
+(\alpha^q+\alpha)B^{-1}w_2, \\
\hspace*{2cm} w_1\mapsto w_1 +(\alpha+\alpha^q)e_i,
w_2\mapsto w_2+(\alpha^q\omega+\alpha\omega^q)e_i)=\tau_i(\alpha)^{z_i}$;

$U=(e_1,e_2)(f_1,f_2)$;

$V=(e_1,e_2,\ldots,e_{n-1})(f_1,f_2,\ldots,f_{n-1})$.

\smallskip
All these elements of $\Omega^-(2n,q)$ centralise those basis elements that they 
are not stated to move.   Observe that $\delta_i^{z_i}=\delta_i^{-1}$.

We give a presentation for $\Omega^-(2n, q)$ on the generating set 
$\{\sigma= \sigma_{12}(1), \tau= \tau_1(1),\delta = \delta_1, z=z_1, U, V\}$, 
omitting $V$ if $n=3$.
 
\subsubsection{A presentation for $\Omega^-(6, q)$}
We exploit the isomorphism between $\Omega^-(6,q)$ and $\SU(4,q)$.  
The exterior square of the natural $\SU(4,q)$-module $W$
may be written over $\GF(q)$, 
giving rise to the natural representation of $\Omega^-(6,q)$, as follows.
Take a hyperbolic basis $(u_1,v_1,u_2,v_2)$
 for $W$, and define a basis $(e_1,f_1,e_2,f_2,w_1,w_2)$ for $W\wedge W$ thus:
$$ e_1 = u_1\wedge v_2, f_1 = v_1\wedge u_2, e_2 = v_1\wedge v_2, f_2=u_1\wedge u_2$$ 
$$ w_1 = u_1\wedge v_1 + u_2\wedge v_2,\, w_2 = \omega u_1\wedge v_1+\omega^q u_2\wedge v_2.$$
 The resulting isomorphism
 maps our chosen generators for $\SU(4,q)$ to $\Omega^-(6,q)$ as follows:
$\sigma \mapsto \tau$,
$\tau\mapsto \sigma$,
$Z \mapsto U$,
$\delta\mapsto\delta\delta^{-U}$,
$\Delta \mapsto \delta$,
$U\mapsto z$.
Theorem \ref{SU2npres} now provides the desired presentation for $\Omega^-(6,q)$. 

\subsubsection{A presentation for $\Omega^-(2n, q)$ for $n \ge 4$}
Let $N_1$ be the subgroup of $\Omega^-(2n, q)$ generated by $\{ z,U,V\}$.  
This is a copy of $C_2\wr S_{n-1}$.
\begin{theorem}\label{OminusevenN1}
Let $n \ge 4$. Let $G$ be the group generated by
$\{z,U,V\}$ subject to the relations given below.  Then $G$ is isomorphic to $N_1$.
\begin{enumerate}
\item[(i)] Defining relations for $S_{n - 1}$ on $\{ U,V\}$. 
\item[(ii)] $[z,U^{V}]=1$.
\item[(iii)] If $n>4$ then $[z,VU]=1$.
\item[(iv)] $z^2=1$.
\item[(v)] $[z,z^{U}]=1$.
\end{enumerate}
\end{theorem}

\begin{proof}
This is a standard wreath product presentation.
\end{proof}

Let $N$ be the subgroup of $\Omega^-(2n,q)$ generated by $\{\delta,z,U,V\}$. 
Now $\delta$ and its conjugates
generate a subgroup of $N$ that is the direct product of 
the cyclic group $\langle\delta\rangle$ 
of order
$q^2-1$ with the direct product of $n-2$ cyclic groups of order $q-1$ generated by 
$\delta_1\delta_2^{-1},\delta_2\delta_3^{-1},\ldots,\delta_{n-2}\delta_{n-1}^{-1}$.  
Thus $N$ has order $(q+1)(q-1)^{n-1}2^{n-1}(n-1)!$, twice that for odd $q$, 
reflecting the fact that 
$\delta$ has order $q^2-1$ if $q$ is even, and order $(q^2-1)/2$ if $q$ is odd.

\begin{theorem}\label{OminusevenN}
Let $n \ge 4$. Let $G$ be the group generated by
$\{\delta, z,U,V\}$ subject to the relations given below.  Then $G$ is isomorphic to $N$.
\begin{enumerate}
\item[(i)] Defining relations for $N_1 = \langle z,U,V\rangle$ as in Theorem
$\ref{OminusevenN1}$.
\item[(ii)] $[\delta,U^{V}]=1$.
\item[(iii)] If $n>4$ then $[\delta,VU]=1$.  
\item[(iv)] $[\delta,z^{U}]=\delta^{q-1}$.
\item[(v)] $\delta^{q^2-1}=1$.
\item[(vi)] $[\delta,\delta^{U}]=1$.
\item[(vii)] $\delta^z=\delta^{-1}$.
\item[(viii)] $[\delta^{q-1},U] = 1$.
\end{enumerate}
\end{theorem}
\noindent
The proof is essentially identical to that of Theorem \ref{OminusoddN}.

\begin{theorem}\label{Ominus-even-sigma}
The centraliser of $\sigma = \sigma_{12}(1)$ in $N$ 
has index $2(n-1)(n-2)(q-1)$ in $N$ and 
is generated by the following elements: 
\begin{enumerate}
\item[(i)] $U^{{V}^2}$ if $n>4$; 
\item[(ii)] $VUU^V$ if $n>5$; 
\item[(iii)] $z^{{V}^2}$;
\item[(iv)] $\delta^{{V}^2}$;
\item[(v)] $\delta\delta^{U}$;
\item[(vi)] $zz^{U}U$.
\end{enumerate}
\end{theorem}
\noindent
The proof is essentially identical to that of 
Theorem \ref{OmegaMinusOdd-cent-sigma}.

As in the case of odd characteristic, 
$\tau_i(\alpha)^{\delta_i}=\tau_i(\alpha\omega^2)$ for $i>0$, 
and $\tau_{i}(\alpha)^{\delta_i}=\tau_i(\alpha\omega^{-2})$ for $i<0$. 
Now $N$ acts transitively by conjugation on 
$\{\tau_{\pm i}(\alpha) \, | \, 1 \le i \le n-1, \alpha\in\GF(q^2)^{\times}\}$,
a set of size $2(q^2-1)(n-1)$. 

\begin{theorem}\label{Ominus-even-tau}
The centraliser of $\tau = \tau_1(1)$ in $N$ 
has index $2(q^2-1)(n-1)$  in $N$ and is generated by 
the following elements:
\begin{enumerate}
\item[(i)] $\delta^{V^2}\delta^{-V}$; 
\item[(ii)] $z^U$;
\item[(iii)] $U^V$;
\item[(iv)] $VU$ if $n>4$.
\end{enumerate}
\end{theorem}
\begin{proof}
The centraliser of $\tau$ in $N$ is generated by the 
subgroup of $\langle U,V\rangle$ that
as permutations of $\{1,2,\ldots,n-1\}$ fixes $1$, 
together with $\{z_i \, | \, i>1\}$ and $\{\delta_i\delta_2^{-1} \, | \, i>2\}$ and 
$\delta_1^{q-1}\delta_2^2$, and hence is generated by 
$\{\delta^{V^2}\delta^{-V}, z^U, U^V,
VU, \delta^{q-1}\delta^{2U}\}$, omitting
$VU$ if $n=4$.  This group has index $2(q^2-1)(n-1)$ 
in $N$, as required, the final generator having order $q-1$.
But $\delta^{q-1}\delta^{2U}=[\delta^{V^2}\delta^{-V},z^U]$, 
so this generator is redundant.
\end{proof}

\begin{theorem} 
Let $q$ be even and let $n \ge 4$.  Let $G$ be the group generated by 
$\{\sigma,\tau,\delta,z,U,V\}$ subject to the relations given below.  
Then $G$ is isomorphic to $\Omega^-(2n,q)$.  
\begin{enumerate}
\item[(i)] Defining relations for $N=\langle\delta,z,U,V\rangle$
as in Theorem $\ref{OminusevenN}$, but omitting relations (iv) and (v)
from Theorem $\ref{OminusevenN1}$, 
relations (iv) to (viii) of Theorem $\ref{OminusevenN}$, 
and the relation $U^2=1$ from the presentation for $\langle U,V\rangle$.

\item[(ii)] Relations that state that the elements listed
in Theorem $\ref{Ominus-even-sigma}$ centralise $\sigma$,
omitting relations (v) and (vi).  

\item[(iii)] Relations that state that the elements listed
in Theorem $\ref{Ominus-even-tau}$ centralise $\tau$, 
but omitting the relation $[\tau,z^{U}]=1$.

\item[(iv)] Relations that present $\Omega^{-}(6,q)$ 
on $\{ \sigma,\tau,\delta,z,U\}$ as in Theorem $\ref{SU2npres}$.

\item[(v)] The Steinberg relations (iv)(a) to (iv)(c), and (iv)(d) when $n>4$, 
of Theorem $\ref{Opluseven}$.

\item[(vi)] The Steinberg relation $[\sigma^{V},\tau]=1$.
\end{enumerate}
\end{theorem}

\begin{proof}
The proof of Theorem \ref{OmegaMinusOdd} carries over with no 
significant change.
\end{proof}
\noindent
If $q$ is even, then $\Omega^-(2n,q)$ is simple.

\subsubsection{Standard generators for $\Omega^-(2n,q)$} 
In Table 1 of \cite{even} the non-trivial standard generators for 
$\Omega^{-}(2n, q)$ are labelled $s, t, \delta, u, v$.
Observe that $s = z^{V^{-1}}; u = U; v = V$;  
the standard generator $\delta$ is $(\delta^{-1})^{V^{-1}}$; 
if $d \equiv 0 \bmod 4$ then  $t = \tau^{V^{-1}}$, else 
$t = {(\tau^{-1})}^{V^{-1}}$.
As a word in the remaining presentation generators, 
$\sigma = [\tau^{\delta^V},\tau^U]^{zU ({\delta^{-V}})^{(q-m)}}$,
where $m$ is defined by $\omega_0^m=\omega^2+\omega^{2q}$,  
so $\sigma$ can be written as an explicit word in the standard generators.
(This arises from relation (v)(b) of Theorem \ref{SU2npres}.)  

\section{Verification and access to results}\label{verify}
Both the presentations
and definitions of their (presentation and standard) generators
are publicly available \cite{code} in {\sc Magma}.
We store the relators as straight-line programs on the 
generating set. 

We have run extensive checks to verify their correctness. 
By evaluating the relations, it is easy to demonstrate
that a classical group of specified dimension $d$ and field size $q$ 
is a quotient of the corresponding finitely-presented group. 

For each type, we have a presentation for $N$, the 
normaliser of a maximal torus in a split $BN$ pair. 
For an extensive range of $d$ and $q$,
typically where the corresponding $N$ has 
order at most $10^9$, we verified 
using coset enumeration that this presentation defines $N$.
We constructed the permutation action of $N$ on 
unordered pairs of root elements, confirming that 
the orbits and representatives are as claimed.
Extensive checks were also performed in
the classical group: in particular,
the claimed (generating sets and structure of) centralisers
in $N$ of root elements have been checked. 
For relevant $d \leq 10$ and $q \leq 9$
(and sometimes for larger $d$ and smaller $q$), 
we also carried out coset enumeration in 
each finitely-presented group 
over the preimage of the maximal
subgroup of smallest index,
and verified that the resulting permutation representation 
defines (the central quotient of) the claimed classical group.  

We observe that our presentations seem particularly well-suited
to coset enumeration, often defining precisely the 
required number of cosets with no redundancy.


\end{document}